\documentclass[12pt, reqno]{amsart}
 
\usepackage{amssymb,amsmath,amsopn,textcomp, thmtools}

\usepackage[backgroundcolor=white, bordercolor=blue, 
linecolor=blue]{todonotes}
\usepackage[normalem]{ulem}
\usepackage[a4paper,margin=1in,footskip=0.25in]{geometry}
\usepackage[ansinew]{inputenc}
\usepackage[english]{babel}
\usepackage{mathrsfs}
\usepackage{setspace}
\usepackage[shortlabels]{enumitem}
\usepackage{stmaryrd}
\usepackage[titletoc]{appendix}
\usepackage[colorlinks=true, linkcolor=blue, citecolor=blue]{hyperref}

\allowdisplaybreaks

\usepackage{dsfont}
\usepackage{bbm}
\usepackage{graphicx}
\usepackage{color}
\usepackage{caption,rotating}

\newtheorem{theorem}{Theorem}[section]
\newtheorem{corollary}[theorem]{Corollary}
\newtheorem{proposition}[theorem]{Proposition}
\newtheorem{lemma}[theorem]{Lemma}
\newtheorem*{theorem*}{Theorem}

\theoremstyle{definition}
\newtheorem{definition}[theorem]{Definition}

\newtheorem{remark}[theorem]{Remark}

\renewenvironment{proof}[1][Proof]{\noindent\textit{#1.} }{\hfill 
	\rule{0.5em}{0.5em}}
\newcommand{\fop}{{\hfill 
		\rule{0.5em}{0.5em}}}

\addto\extrasenglish{}
\addto\extrasenglish{} 

\numberwithin{equation}{section}

\usepackage{color}%for comments
\newcommand{\cA}{\mathcal{A}}

\newcommand{\cM}{\mathcal{M}}
\newcommand{\cN}{\mathcal{N}}
\newcommand{\fN}{\mathfrak N}
\newcommand{\cW}{\mathcal{W}}
\newcommand{\cE}{\mathcal{E}}
\newcommand{\cF}{\mathcal{F}}
\newcommand{\cG}{\mathcal{G}}
\newcommand{\cH}{\mathcal{H}}

\newcommand{\N}{\mathbb{N}}
\newcommand{\R}{\mathbb{R}}

\newcommand{\Rd}{{\mathbb{R}^d}}

\newcommand{\Z}{\mathbb Z}
\newcommand{\E}{\mathbb E}
\newcommand{\Pp}{\mathbb P}
\newcommand{\1}{\mathbbm 1}

\newcommand{\I}{{\rm I}}
\newcommand{\II}{{\rm II}}

\def\nn{\nonumber}

\newcommand{\switch}{\epsilon}

\newcommand{\HHq}[2]{\big(\mathbf H_{#1}^{#2}\big)}

\newcommand{\Tj}[2]{\sum_{j={#1}}^{#2} \lengthR_j}
\newcommand{\TT}[2]{\mathcal N({#1},{#2})}
\newcommand{\gj}[2]{\sum_{j={#1}}^{#2} \gamma^j}
\newcommand{\gam}[2]{\Upsilon({#1},{#2})}
\newcommand{\lengthR}{n}

\newcommand\EE{\mathcal {E}}
\newcommand\FF{\mathcal {F}}

\newcommand{\lC}{{\underline{c}}}
\newcommand{\uC}{{\overline{C}}}
\newcommand{\la}{{\underline{\alpha}}}
\newcommand{\ua}{{\overline{\alpha}}}

\newcommand{\newa}{\la}%%estimates_of_the_borderline_q
\newcommand{\newb}{\frac{\newa}{\newa+1}}%%\trac{1}{1+\newa}+\newb=1
\newcommand{\newA}{\frac{\la+1}{\ua+1}}%%

\makeatletter
\@namedef{subjclassname@2020}{%
	\textup{2020} Mathematics Subject Classification}
\makeatother
\author{Kyung-Youn Kim}
\email{kykim@gate.sinica.edu.tw}
\address{Institute of Mathematics, Academia Sinica, Taiwan}

\author{Lidan Wang}
\email{lidanw.math@gmail.com}
\address{School of Statistics and Data Science, Nankai University, Tianjin, P.R. China}

\title[Heat kernel bounds for systems of nonlocal equations]
{Heat kernel bounds for a large class of Markov process with singular jump}
\thanks{LW's Research partially supported by NNSF of China(Grant No. 11801283) and Key Laboratory for Medical Data Analysis and Statistical Research of Tianjin (KLMDASR)} 

\begin{document}

	\begin{abstract}
	Let $Z=(Z^{1}, \ldots, Z^{d})$  be the $d$-dimensional L\'evy processes where $Z^{i}$'s are independent $1$-dimensional L\'evy processes with jump kernel $J^{\phi, 1}(u,w) =|u-w|^{-1}\phi(|u-w|)^{-1}$ for $u, w\in \R$. Here $\phi$ is an increasing function with weak scaling condition of order $\underline \alpha, \overline \alpha\in (0, 2)$. 
	Let $J(x,y) \asymp  J^\phi (x,y)$ be the symmetric measurable function where
\begin{align*}
J^\phi(x,y):=\begin{cases}
J^{\phi, 1}(x^i, y^i)\qquad&\text{ if  $x^i \ne y^i$ for some $i$ and $x^j = y^j$ for all $j \ne i$}\\
0\qquad&\text{ if $x^i \ne y^i$ for more than one index $i$.}
\end{cases}
\end{align*}
Corresponding to the jump kernel $J$, we show the existence of non-isotropic Markov processes $X:=(X^{1}, \ldots, X^{d})$ and obtain  sharp two-sided heat kernel estimates for the transition density functions.

	\end{abstract}
	
	\maketitle
\section{Introduction}
Suppose that $X$ is a symmetric Markov process in $\R^d$, with transition density $p(t,x,y)$ and generator $\mathcal L$. It is known that $p(t,x,y)$ is the fundamental solution to $\partial_t u=\mathcal L u$, hence it is also called the heat kernel of $\mathcal L$. Since most heat kernels do not have explicit expressions, establishing sharp two-sided heat kernel estimate is a fundamental problem and it has received  intensive attention in the theory of analysis as well as that of probability. Correspondingly, there are studies of regularities and potential properties  for diverse types of Markov processes.
In this paper, we consider a large class of non-isotropic pure jump Markov processes and investigate the extended version of the conjecture formulated in \cite{KKKpre} as follows:
\medskip

{\bf Conjecture:} \emph{Let $L_t$ be a L\'evy process(a non-degenerate $\alpha$-stable process) in $\mathbb{R}^d$ with L\'{e}vy measure $\mu$. Let $M_t$ be a symmetric Markov process whose Dirichlet form has a symmetric jump intensity $j(x, d y)$ that is comparable to the one of $L_t$, i.e., $j(x, d y) \asymp \mu(x-d y)$.  Then the heat kernel of $M_t$ is comparable to the one of $L_t$.}
\medskip

We use the notation $f\asymp g$ if the quotient $f/g$ is comparable to some positive constants.

\medskip
There is a long history of studies on the above conjecture. It is shown in \cite{Aro68} that the uniformly elliptic operator  in divergence form, $\cA:=\sum\partial_i(a_{i, j}(x) \partial_j)$, is related to  the Laplacian $-\Delta$,  that is, the fundamental solutions of $\partial_t u=\cA u$ and $\partial_t u=-\Delta u$ are comparable.
In \cite{BaLe02} and \cite{ChKu03},  analogous results are obtained for  non-local operators related to $\alpha$-stable processes and non-degenerate jump processes with the jump kernel $j(x, y)\asymp |x-y|^{-d-\alpha}$ for $\alpha\in (0, 2)$.
Very recently, there are researches on the non-isotropic case by \cite{Xu13, KKKpre} with the L\'evy measure
$\mu(dh):=\sum\limits_{i=1}^d |h^i|^{-1-\alpha} dh^i \prod\limits_{j \ne i} \delta_{\{0\}}(dh^j)$
and the jump intensity $j(x, dy)\asymp \mu(x-dy)$. Here, $\delta_{\{0\}}$ is the Dirac measure at $\{0\}$. 

In \cite{ChKu08},  Chen-Kumagai  introduce a large class of non-local symmetric Markov processes of variable order using an increasing function with weak scaling conditions. 
Motivated by their research, we consider the following non-isotropic Markov processes.

For any $0<\la\le \ua<2$, let $\phi:[0, \infty)\rightarrow[0, \infty)$ be an increasing function with the following condition:
there exist positive constants $\lC\le 1$ and $\uC\ge 1$ such that 
\begin{align*}
{\bf(WS)}
\qquad\qquad	
\lC\left(\frac{R}{r}\right)^{\la}\le \frac{\phi(R)}{\phi(r)}\le \uC \left(\frac{R}{r}\right)^{\ua} \qquad \mbox{ for }\,\, 0<r\, \le  R.
\end{align*}
Using this $\phi$,  define
${\nu^1}(r):={(r \phi(r))}^{-1}\mbox{ for}\,\, r>0$.
Then {\bf(WS)} implies 
$$\int_{\R}(1\wedge |s|^2){\nu^1}(|s|)d s\le c \left(\int_0^{1}r^{-\ua+1} d r+ \int_{1}^{\infty} r^{-\la-1} \right)d r<\infty,$$
so ${\nu^1}(d s):={\nu^1}(|s|)d s$ is a L\'{e}vy measure.
Consider a non-isotropic L\'evy process $Z:=(Z^{1}, \ldots, Z^{d})$ where each coordinate process $Z^{i}$ is an independent $1$-dimensional symmetric L\'evy process with L\'evy measure ${\nu^1}(d s)$.
Then the corresponding L\'evy measure $\nu$ of $Z$ is represented by
\begin{align*}
\nu(d h) = \sum\limits_{i=1}^d {\nu^1}(|h^i|) d h^i \prod\limits_{j \ne i} \delta_{\{0\}}(d h^j).
\end{align*}
Intuitively, $\nu$ only measures  the sets containing the line which is parallel to one of the coordinate axes.
The corresponding Dirichlet form $(\mathcal{E}^\phi,  \mathcal F^\phi)$ on $L^2(\R^d)$ is given by 
\begin{align}
\label{Diriphi}
\mathcal{E}^\phi (u,v) &=\int_{\Rd}\Big(\sum_{i=1}^d\int_{\R}\big(u(x+e^i \tau) - u(x)\big)\big(v(x+e^i \tau) -  v(x)\big) J^\phi (x, x+e^i \tau) d \tau \Big)d x,\notag \\
\mathcal F^\phi &=\{u\in L^2(\Rd)| \mathcal{E}^\phi (u,u)<\infty\},
\end{align}
with the jump kernel
\begin{align*}
J^\phi (x, y):=
\begin{cases}
|x^i-y^i|^{-1}\phi(|x^i-y^i|)^{-1} \ &
\text {if $x^i \ne y^i$ for some $i$; $x^j = y^j$ for all $j \ne i$,}\\
0\,&\text{otherwise}.
\end{cases}
\end{align*}
By \cite[Theorem 1.2]{ChKu08},  the transition density $p^{Z^{i}}(t, x^i, y^i)$ of $Z^{i}$ has the following estimates
$$p^{Z^{i}}(t, x^i, y^i)\asymp \left( [\phi^{-1}(t)]^{-1}\wedge t{\nu^1}(|x^i-y^i|)\right),$$
where $a\wedge b:=\min\{a,b\}$. 
Since $Z^{i}$'s are independent, it is easy to obtain the upper and lower bounds for the transition density $p^Z(t, x, y)$ of $Z$, that is,
\begin{align}\label{eq:hkez}
p^Z(t, x, y)\asymp \prod_{i=1}^d\left( [\phi^{-1}(t)]^{-1}\wedge t{\nu^1}(|x^i-y^i|)\right).
\end{align}
Now we consider a symmetric measurable function $J$ on $\R^d \times \R^d \setminus \operatorname{diag}$ such that for every $(x, y)\in \R^d \times \R^d \setminus \operatorname{diag}$,
\begin{align}\label{eq:J-ellipticity-assum}
\Lambda^{-1} J^\phi (x,y) \leq J(x,y) \leq \Lambda J^\phi (x,y),
\end{align}
for some $\Lambda>1$, and define $(\mathcal{E}, \FF)$ on $L^2(\Rd)$ as follows: 
\begin{align}\label{eq:df}
\mathcal{E} (u,v)&
:=\int_{\Rd}\Big(\sum_{i=1}^d\int_{\R}\big(u(x+e^i \tau) - u(x)\big)\big(v(x+e^i \tau) -  v(x)\big) J (x,x+e^i \tau) d \tau \Big) d x \,,\nn\\
\mathcal F&:=\{u\in L^2(\Rd)| \; \mathcal{E}(u,u)<\infty\}.
\end{align}
Then we  state our main result of this paper.
\begin{theorem}
	\label{mainthm}
	Assume the condition \eqref{eq:J-ellipticity-assum} holds. Then there exists a conservative Feller process $X=(X^{1}, \ldots, X^{d})$ associated with $(\cE, \cF)$ that starts from every point in $\R^d$. Moreover, $X$ has a jointly continuous transition density function $p(t,x,y)$ on $\R_+\times\R^d\times\R^d$, which enjoys the following estimates: there exist positive constants $c_1, c_2>0$ such that for any $t>0, x,y\in\R^d$, 
	\begin{align}\label{eq:hke}
c_1 [\phi^{-1}(t)]^{-d}\prod_{i=1}^d&\left(1\wedge \frac{t\phi^{-1}(t)}{|x^i-y^i|\phi(|x^i-y^i|)}\right)\nn\\
	&
	\le p(t,x,y)\le c_2	[\phi^{-1}(t)]^{-d}\prod_{i=1}^d\left(1\wedge \frac{t\phi^{-1}(t)}{|x^i-y^i|\phi(|x^i-y^i|)}\right).
	\end{align}
\end{theorem}

Comparing \eqref{eq:hke} with \eqref{eq:hkez}, we give the answer that the {\bf Conjecture}  holds for a large class of non-isotropic Markov processes.
%The value of our research is that we cover diverse types of Markov processes, even including processes of which each coordinate $M^{i}$ in $M=(M^{1}\ldots, M^{d})$  is not the same. 

\begin{remark}	 
We give some examples of $L:=(L^{1},\ldots, L^{d})$ (in our discussion, $Z=(Z^1,\ldots,Z^d)$)
% and $M:=(M^{1},\ldots, M^{d})$ 
in the {\bf Conjecture} which our study covers.
\begin{itemize}
%\item [1.] 
%Let $M^{i}:=M^{\alpha_i}$ be a $1$-dimensional $\alpha_i$-stable-like process having the jump kernel $j^i(h)\asymp |h|^{-1-\alpha_i}, h\in \R$ with possibly different $\alpha_i\in (0, 2)$. 
%Consider an  increasing function $\phi(r), r>0$ satisfying ({\bf WS}) with $\ua:=\min\limits_{i\in \{1, \ldots, d\}}\alpha_i$ and $\la:=\max\limits_{i\in \{1, \ldots, d\}}\alpha_i$.
%Then,
%\begin{align*}
%\lC |h|^{\alpha_i}\le  \lC|h|^{\la}\le
%&\ \phi(|h|)/\phi(1)\le \uC|h|^{\ua}\le\uC |h|^{\alpha_i}&\text{for } |h|\ge 1,\\
%\lC |h|^{-\alpha_i}\le  \lC|h|^{-\la}\le
%&\ \phi(1)/\phi(|h|)\le \uC|h|^{-\ua}\le\uC |h|^{-\alpha_i}&\text{for } |h|\le 1,
%\end{align*}
%and so that  $\mu^1(h):=(|h|\phi(|h|))^{-1}\asymp j^i(h)$. Our result \autoref{mainthm} offers a type of two-sided heat kernel estimates for $M=(M^{\alpha_1},\cdots,M^{\alpha_d})$.
\item[1.]
Let $ m>0$ and $\alpha\in (0, 2)$.
Let $L^{i}$ be a $1$-dimensional $m$-relativistic $\alpha$-stable process  with the characteristic function
$\E^0\big[\exp\big(i \xi\cdot L^{i}_t\big)\big]
=\exp\big(t\big(m^{\alpha}-(\xi^2+m^2)^{\alpha/2}\big)\big)$.
By \cite[(3.9)--(3.10)]{ChSo03}, the corresponding L\'evy kernel $\mu^{1}(h)$ of $ L^{i}$ satisfies that 
$$\mu^{1}(h)\asymp {\Phi(m|h|)}/{|h|^{1+\alpha}},$$
where $\Phi(r)\asymp e^{-r}(1+r^{\alpha/2})$ near $r=\infty$ and $\Phi(r)=1+(\Phi)^{''}(0)r^2/2+o(r^4)$ near $r=0$.
%Therefore we consider $d$-dimensional Markov process $M$ with the kernel $j^i(h)\asymp \mu^1(h)$ where $\mu^1(h)\asymp $|h|^{-1-\alpha}$.
\item[2.]
 For $\alpha_{ k}\in(0, 2),  m>0$, the independent sum of $1$-dimensional $\alpha_{k}$-stable processes $L^{ i}:=L^{\alpha_1}+\ldots+ L^{\alpha_n}$ and of $1$-dimensional relativistic $\alpha_{k}$-stable processes $L^{ i}:=L^{m, \alpha_1}+\ldots+L^{m, \alpha_n}$ are typical examples of $1$-dimensional L\'evy process $L^{i}$ having the kernel $\mu^1(h)=(h\phi(|h|))^{-1}$ with {\bf (WS)} for $\phi$.
For more details, see \cite[page 279]{ChKu08}.

\item[3.]
For more examples, one can refer to \cite[Example 2.3]{ChKu08}.
\end{itemize}
\end{remark}

The  paper is organized as follows: In \autoref{sec:pre}, we prove the existence of a conservative Hunt process $X$ associated to Dirichlet form $(\EE, \FF)$ defined in \eqref{eq:df}, and obtain on-diagonal upper heat kernel estimates for $p(t, x, y)$ with the help of Nash-type inequality. Also, we discuss the scaled process $Y^{(\kappa)}_t=\kappa^{-1}X_{\phi(\kappa)t}$, $\kappa>0$ which gives the clue of the exit time estimates for $X$ displayed in \autoref{cor:exit_new}. 
In \autoref{sec:uphke}, we aim to construct off-diagonal upper bound estimates. However, unlike isotropic jump processes,  Meyer's decomposition method does not work  for our non-isotropic case. Therefore, we present a strategy of the proof for off-diagonal upper bound estimates for $p(t, x, y)$ with the main technical result \autoref{prop:main}. The proof of \autoref{prop:main} is established in Appendix.
We  adopt the main idea shown in \cite{KKKpre} for the upper bound estimates in \autoref{sec:uphke} and Appendix. 
 However, we give a new  refined exponent $\{\theta_l:  l\in\{0, 1, \ldots, d-1\}\}$ of $\HHq{q+\theta_l}{l}$ (see, \eqref{eq:Hi}) to obtain the estimates in an accurate way, and introduce an efficient notation $\mathfrak N(\delta), \delta\in \Z$ in \eqref{eq:newmm} 
to deal with the diverse orders induced by $\phi$.
The H\"older continuity for the bounded harmonic function is shown in \autoref{sec:low}, and it implies our process $X$ is Feller. 
Moreover, we obtain the H\"older continuity of $p(t,x,y)$, with which we are able to refine the process $X$ to start from everywhere. The matching lower bound heat kernel estimate for $p(t, x, y)$ is then obtained in \autoref{sec:low}.

\medskip
{\bf Notations.}
For $k\ge 0$, let $C_c^k(\Rd)$ and  $C_0^k(\Rd)$ denote spaces of  $C^k$-functions on $\Rd$ with compact support, and
with vanishing at infinity, respectively.
We use $\langle \cdot, \cdot\rangle$  to denote the inner product in $\Rd$, $\|\cdot\|_r$ to denote the $L^r(\Rd)$ norm, and $|A|$ to denote the Lebesgue measure of $A\subset \Rd$.
For $x\in\R^d, r>0$, we use $B(x, r)$ to denote  a ball centered at $x$ with  radius $r$, and $Q(x,r)$ to denote  a cube centered at $x$ with  side length $r$.
 The letter $c=c(a, b, \ldots)$ will denote a positive constant depending on $a,b,\ldots$, and it may change at each appearance.
The labeling of the constants $c_1, c_2, \ldots$ begins anew in the proof of each statement.
 The notation $:=$ is to be read as ``is defined to be".

\section{Preliminaries}\label{sec:pre}
In this section, we will discuss the existence of a process $X$  corresponding to the jump kernel $J$ defined in \eqref{eq:J-ellipticity-assum}, the transition density $p(t, x, y)$, as well as the distribution of exit times for $X$. In \autoref{subsec:Dirichlet}, we argue that $(\cE,\cF)$ is a regular Dirichlet form, hence , according to Theorem 7.2.1 in \cite{MR2778606}, there exists a Hunt process that starts from almost every point. In \autoref{subsec:Nash}, we discuss the Nash-type inequality and obtain on-diagonal upper heat kernel estimates of $p(t,x,y)$. In \autoref{subsec:upall}, we introduce scaled and truncated processes to study the distributions of exit times.
\subsection{The Dirichlet form}
\label{subsec:Dirichlet}
\begin{definition}
	\label{def:Dirichlet}
	We say that $ (\mathcal G, \mathcal H)$ is a \textit{symmetric Dirichlet form }on $L^2(\Rd)$  if $\cG$ is a non-negative definite symmetric bilinear form on $L^2(\Rd)$ which is closed  and Markovian, that is, 
	\begin{itemize}
		\item [($\cG$.1)] ({\it $\cG$ is closed}): if $ \mathcal H$ is complete with respect to $\cG_1:=\cG+\|\cdot\|_2^2$ metric.
		\item [($\cG$.2)] ({\it $\cG$ is Markovian}):
		for each $\varepsilon>0$, there exists a function $\eta_{\varepsilon}(t)\in[-\varepsilon,1+\varepsilon]$ for $t\in \R$ such that
		\begin{align*}
		\eta_{\varepsilon}(t)=t\ \text{ for } &t\in [0, 1];\quad
		0\leq\eta_\varepsilon(t)-\eta_\varepsilon(s)\leq t-s\  \text{ for } s<t; \\
		\text{ for any }u\in  { \mathcal H}, &\quad { \eta_\varepsilon\circ u}\in {\mathcal H}\ \text{ and } \ \cG( {\eta_\varepsilon\circ u},  {\eta_\varepsilon\circ u})\le \cG(u, u).
		\end{align*}
	\end{itemize} 
	A subset $\mathcal{C}$ of $\cH\cap C_c(\Rd)$  is called a {\it core} of a symmetric form $\cG$ if it is dense in $\cH$ with $\cG_1$-norm and dense in $C_c(\Rd)$ with uniform norm, and 
	$\cG$ is called {\it regular} if
	\begin{itemize}
		\item [($\cG$.3)] $\cG$ possesses a core. 
	\end{itemize}
\end{definition}

\begin{theorem}
	\label{thm:regular}
	Let $J$ be the symmetric measurable function defined in \eqref{eq:J-ellipticity-assum}. 
	Then  $(\mathcal E, \mathcal F)$  defined in \eqref{eq:df} is a regular Dirichlet form {on $L^2(\Rd)$}.
\end{theorem}
\begin{proof}
Since $C_c^1(\R^d)\subset  \mathcal F$ and $C_c^1(\R^d)$ is a dense subspace in $L^2(\R^d)$, $\mathcal F$ is clearly dense in $L^2(\R^d)$.
The linearity and symmetricity of $\cE$ are clear, and we will show that $(\cE, \cF)$ is a regular Dirichlet form on $L^2(\Rd)$ by the following $3$-steps.

(1) {\it  $\mathcal E$ is closed.}
	Let $\{u_n\}_{n\geq1}$ be a $\mathcal E_1$-Cauchy sequence in $\mathcal F$. Since $\mathcal F$ is dense in $L^2(\R^d)$, there exists $u\in L^2(\R^d)$ with $\|u_n-u\|_2\to0$ as $n\to \infty$. It remains to show $u\in \mathcal F$ and $\mathcal E(u_n-u,u_n-u)\to0$ as $n\to \infty$. 
	For any $n\in \N$, denote $v_n=u_n-u$. By $\bf (WS)$, we have that  for any $\varepsilon>0$
		\begin{align}\label{eq:e-closed}
		&\int_{\Rd}\Big(\sum_{i=1}^d\int_{|\tau|>\varepsilon}\big(v_n(x+e^i \tau) - v_n(x)\big)^2 J (x,x+e^i \tau) d \tau \Big) d x\nn\\
		\leq&\  \frac{\Lambda\uC}{\phi(1)}\int_{\Rd}\left(\sum_{i=1}^d\int_{\varepsilon<|\tau|\leq 1}\frac{\big(v_n(x+e^i \tau) - v_n(x)\big)^2}{|\tau|^{1+\ua}} d\tau\right)d x\nn\\
		&\qquad+\frac{\Lambda}{\lC \phi(1)} \int_{\Rd}\left(\sum_{i=1}^d\int_{|\tau|> 1}\frac{\big(v_n(x+e^i \tau) - v_n(x)\big)^2}{|\tau|^{1+\la}}d\tau\right)d x\nn\\
		\leq&\  \frac{\Lambda\uC}{\phi(1)}\sum_{i=1}^d\int_{\varepsilon<|\tau|\leq 1}\int_{\Rd}\frac{\big(v_n(x+e^i \tau) - v_n(x)\big)^2}{|\tau|^{1+\ua}}d xd\tau+\frac{4d\Lambda}{\lC\phi(1)\la}\|v_n\|_2^2,
		\end{align}
and the right-hand side of \eqref{eq:e-closed} converges to $0$ as $n\to \infty$.
		Since the left-hand side of  \eqref{eq:e-closed} converges to $0$ as $n\to \infty$ for any arbitrary $\varepsilon>0$,   $(\mathcal E, \mathcal F)$ is closed.
\medskip
	
(2) {\it  $\mathcal E$ is Markovian.}
		For any $\varepsilon>0$, let
		 \begin{equation}	\label{psieps}
		 \tilde\eta_\varepsilon(x):=\begin{cases}
	1+\varepsilon& \text{if }x\in(1+\varepsilon,\infty);\\
	x& \text{if }x\in[-\varepsilon,1+\varepsilon];\\
	-\varepsilon& \text{if }x\in(-\infty,-\varepsilon), 
	\end{cases}\text{ and }\
	\psi_\varepsilon(x):=\begin{cases}
	ce^{-1/(\varepsilon^2-x^2)}& \text{if }x\in[-\varepsilon,\varepsilon];\\
	0& \text{otherwise,}
	\end{cases}
	\end{equation}
where $c$ is a constant such that $\int_\R \psi_\varepsilon d x=1$.
Then $\eta_\varepsilon:=\tilde{\eta}_\varepsilon*\psi_\varepsilon$ has the following properties:
\begin{itemize}
\item$\eta_\varepsilon(t)=t$ for all $t\in[0,1]$;
\item $\eta_\varepsilon(t)\in[-\varepsilon,1+\varepsilon]$ for all $t\in\R$;
\item $0\leq\eta_\varepsilon(t)-\eta_\varepsilon(s)\leq t-s$ if $s<t$.
\end{itemize}
Thus for any $u\in\mathcal F$, 
\begin{align*}
\mathcal E(\eta_\varepsilon(u),\eta_\varepsilon(u))
=&\int_{\Rd}\Big(\sum_{i=1}^d\int_{\R}\big(\eta_\varepsilon(u(x+e^i\tau))-\eta_\varepsilon(u(x))\big)^2J(x,x+e^i\tau)d \tau\Big)d x\\
\leq&\int_{\Rd}\Big(\sum_{i=1}^d\int_{\R}\big(u(x+e^i\tau)-u(x)\big)^2J(x,x+e^i\tau)d \tau\Big)d x=\mathcal E(u,u),
\end{align*}
 and hence $\mathcal E$ is Markovian.
\medskip

(3) {\it $\mathcal E$ is regular}. 
 Our claim is to prove  $\mathcal F=\overline{C_c^1(\Rd)}^{\cE_1}$.  Clearly $C_c^1(\Rd)\subset\mathcal F$, and using the similar arguments  as in (1), we have $\overline{C_c^1(\Rd)}^{\cE_1}\subset \mathcal F$. On the other hand, to obtain $ \mathcal F\subset \overline{C_c^1(\Rd)}^{\cE_1}$, since
$(\mathcal E, \mathcal F)$ is comparable to $(\mathcal E^\phi,\mathcal F^\phi)$ defined in \eqref{Diriphi} with $\mathcal F=\mathcal F^\phi$, it suffices to show that $\mathcal F^\phi\subset \overline{C_c^1(\Rd)}^{\cE_1^\phi}$ where $\mathcal{E}_1^\phi(u, u):=\mathcal{E}^\phi(u, u)+\|u\|_2^2$.
Let ${U_\lambda^Z} u(x)$ be the $\lambda$-resolvent for $Z$ defined as
     $${U_\lambda^Z} u(x):=\mathbb E^x\int_0^\infty e^{-\lambda t}u(Z_t)d t=\int_0^\infty\int_{\R^d}e^{-\lambda t}u(y){ p^Z}(t,x,y)d yd t,$$
     where $p^Z (t,x,y)$ is the transition density of $Z$.
%    Since $Z^{1}, Z^{2},\dots, Z^{d}$ are independent, the transition density ${p^Z}(t,x,y)$ of $Z$ has the following estimates {as in \eqref{eq:hkez}}:     
%for any $t>0$, $x,y\in\Rd$,
%     \begin{equation*}
%     { p^Z}(t,x,y)\asymp\prod_{i=1}^d\left([\phi^{-1}(t)]^{-1}\wedge \frac{t}{\phi(|x^i-y^i|)|x^i-y^i|}\right).
%     \end{equation*}
     Since ${U_\lambda^Z}(C_c(\Rd))$ is dense in $\mathcal F^\phi$ with respect to $\mathcal E^\phi_1$-metric,  it remains to show ${U_\lambda^Z}(C_c(\Rd))\subset \overline{C_c^1(\Rd)}^{\mathcal{E}_1^\phi}$.
     For any $f\in {U_\lambda^Z}(C_c(\Rd))$, there exists $u\in C_c(\Rd)$ such that supp$[u]\subset B(0,M)$ for some $M>0$ and $f={U_\lambda^Z} u$. Then  for any $t_0>0$,
     \begin{align*}
     |f(x)|=|{U_\lambda^Z} u(x)|&\leq \int_0^\infty \mathbb E^x[e^{-\lambda t}|u(Z_t)|]d t\\
     &=\int_0^{t_0}e^{-\lambda t}\mathbb E^x[|u(Z_t)|]d t+\int_{t_0}^\infty \mathbb E^x[e^{-\lambda t}|u(Z_t)|]d t\\
     &\leq \frac{1-e^{-\lambda t_0}}{\lambda}\|u\|_{\infty}+\int_{t_0}^\infty \int_{B(0,M)}e^{-\lambda t}|u(y)| {p^Z}(t,x,y)d yd t.
     \end{align*}
    For any $\varrho>0$, we choose $t_0$ small enough so that
    $$\frac{1-e^{-\lambda t_0}}{\lambda}\|u\|_{\infty}<\frac{\varrho}{2}.$$
    Also by \eqref{eq:hkez}, there exists $M_1$ large enough such that for all $x$ with $|x|>M_1$, 
     $$\int_{t_0}^\infty \int_{B(0,M)}e^{-\lambda t}|u(y)| { p^Z}(t, x, y)d yd t<\frac{\varrho}{2}.$$
     Therefore, we have $f\in L^2(\Rd)\cap C_0(\Rd)$.
    Let $a\vee b:=\max\{a, b\}$.
    For any $\delta>0$ and $\varepsilon>0$, let
     $$f^\delta:=f-[(-\delta)\vee f]\wedge \delta \quad \text{ and }\quad f^\delta_\varepsilon(x):=\int_{B(0,\varepsilon)}\psi_{\varepsilon}(|y|)f^\delta(x-y)d y,$$ where $\psi_\varepsilon$ is given in \eqref{psieps}. By this mollification, we see that
     $$f^\delta_\varepsilon\in C_c^1(\Rd)\qquad \text{ and }\qquad
      \|f_\varepsilon^\delta\|_2\to\|f^\delta\|_2.$$
		Moreover, we have
		\begin{align*}
		&\mathcal E^\phi(f_\varepsilon^\delta,f_\varepsilon^\delta)=\int_{\R^d}\sum_{i=1}^d\int_\R\big(f_\varepsilon^\delta(x+e^i\tau)-f_\varepsilon^\delta(x)\big)^2J^\phi(x,x+e^i\tau)d \tau d x\\
		&\leq \int_{\R^d}\sum_{i=1}^d\int_\R\int_{B(0,\varepsilon)}\big(f^\delta(x+e^i\tau-z)-f^\delta(x-z)\big)^2J^\phi(x,x+e^i\tau)\psi_{\varepsilon}(|z|)d zd \tau d x\\
		&=\int_{B(0,\varepsilon)}\psi_{\varepsilon}(|z|)\int_{\R^d}\sum_{i=1}^d\int_\R\big( f^\delta(x+e^i\tau)-f^\delta(x)\big)^2J^\phi(x,x+e^i\tau)d \tau d xd z
		=\mathcal E^\phi(f^\delta,f^\delta).
		\end{align*}
		We now fix $\delta$.  Note that for any $v\in L^2(\Rd)$,
		$$\cE_1^\phi(f_\varepsilon^\delta, {U^Z_1}v)= \langle f_\varepsilon^\delta,v\rangle\to \langle f^\delta,v\rangle =\cE_1^\phi(f^\delta,{ U^Z_1}v)\qquad\text{ as }\varepsilon\to0. $$  
		Since $U^Z_1(L^2(\Rd))$ is dense in $\cF^\phi$ with respect to $\cE_1^\phi$ and $f^\delta\in \cF^\phi$ by \cite[Theorem 1.4.2 (iv)]{MR2778606}, we have that $\cE_1^\phi(f_\varepsilon^\delta,f^\delta)\to \mathcal E^\phi_1(f^\delta,f^\delta)$ as $\varepsilon\to0$, therefore,
		\begin{align*}
		\mathcal E_1^\phi(f_\varepsilon^\delta-f^\delta,f_\varepsilon^\delta-f^\delta)
		\leq&\ 2\mathcal E^\phi_1(f^\delta,f^\delta)-2\mathcal E_1^\phi(f_\varepsilon^\delta,f^\delta)
		\xrightarrow{\varepsilon\to0}0.
		\end{align*}
Hence, $f_\varepsilon^\delta\to f^\delta$ with respect to $\mathcal E_1^\phi$ norm as $\varepsilon\downarrow0$. Since $f^\delta\to f$ with respect to $\mathcal E_1^\phi$ norm as $\delta\downarrow0$ by \cite[Theorem 1.4.2 (iv)]{MR2778606}, we conclude that $f\in \overline{C_c^1(\Rd)}^{\mathcal{E}_1}$.
\end{proof}

\medskip
	
	\begin{remark}\label{rem:reg}
		By \autoref{thm:regular},  $(\mathcal E, \mathcal F)$ is a regular Dirichlet form on $L^2(\R^d)$. Therefore, there exists $\mathcal N_0\subset\R^d$ having zero capacity with respect to the Dirichlet form $(\mathcal E, \mathcal F)$ and there exists a Hunt process $(X, \mathbb P^x)$ that can start from any point in $\R^d\setminus\mathcal N_0.$ For more details on this, refer to \cite{MR2778606}.
		Also the following L\'evy system for $X$ holds:
		\begin{equation}\label{eq:LSd}
		\E^x \left[\sum_{s\le S} f(s,X_{s-}, X_s) \right] = \E^x \left[ \int_0^S \big(\sum_{i=1}^{d}\int_{\R} 
		f(s,X_s, X_s+e^ih) 
		J(X_s,X_s+e^ih) dh \big) ds \right],
		\end{equation}
		where $f$ is non-negative, vanishing on the diagonal, and $S$ is a stopping time (with respect to the filtration of $X$). 
		For the L\'evy system, refer to \cite{ChKu03} and \cite{ChKu08} for the proof.
Moreover, we will show $X$ is a strong Feller process in \autoref{rem:sFP}, after obtaining  the H\"older continuity of resolvents.
		\end{remark}

\subsection{ Nash-type  inequality and on-diagonal upper heat kernel estimates}
\label{subsec:Nash}
\begin{proposition}
	\label{propNash}
There exist constants $c_1,c_2$ such that for any $f\in \mathcal F$ with $\|f\|_1=1$, we have
	$$\|f\|_2^2\leq c_1\mathcal E(f,f)\phi(c_2\|f\|_2^{-2/d}).$$
\end{proposition}
\begin{proof}
For $f\in \mathcal F=\mathcal F^\phi$, denote by $f_t=P^Z_tf$ where $P_t^Z$ is the semigroup associated with $Z$.
 By \eqref{eq:hkez}, there exists a constant $c_1>0$ such that
		$$\|f_t\|_\infty\leq c_1\|f\|_1[\phi^{-1}(t)]^{-d}.$$
Let $A^Z$ be the generator of $P_t^Z$ so that $f_t=f +\int_0^tA^Zf_sd s.$
Then for any $t>0$ and  {for $f\in \FF$} with $\|f\|_1=1$, we have that 
\begin{align}
		\label{fL2}
		\|f\|_2^2=&\langle f, f_t\rangle-\int_0^t\langle f,A^Zf_s\rangle d s
	\leq\|f_t\|_\infty\|f\|_1+t\mathcal E^{\phi}(f,f)\nn\\
		\leq&\ c_1[\phi^{-1}(t)]^{-d}+t\mathcal E^{\phi}(f,f).
\end{align}
		We want to minimize the right-hand side with respect to $t$. Since $\phi$ is increasing, there exists  a unique $t_0>0$ such that $ \mathcal E^{\phi}(f,f)=t_0^{-1}[\phi^{-1}(t_0)]^{ -d}$ which implies
		$\|f\|_2^2\leq c_2[\phi^{-1}(t_0)]^{-d}$ and $t_0\leq \phi(c_2^{1/d}\|f\|_2^{-2/d})$.
Therefore, $\|f\|_2^2$ is bounded above at $t=t_0$ in \eqref{fL2} so that 
$$\|f\|_2^2\leq (c_1+1) t_0{\mathcal E^\phi}(f,f)
\le		(c_1+1) \phi(c_2^{1/d}\|f\|_2^{-2/d}){\mathcal E^\phi}(f,f).$$
Since $(\mathcal E, \mathcal F)$ and $(\mathcal E^\phi,\mathcal F^\phi)$ are comparable, we get our assertion.
\end{proof}

\medskip

Let $P_t$ be the transition semigroup for $X$ associated with $(\mathcal E, \mathcal F)$. From the Nash-type inequality shown in \autoref{propNash},  in the following, we conclude that $P_t$ has a kernel $p(t,x,y)$ with the on-diagonal upper bound estimates almost everywhere.
 The exceptional set $\cN_0$ of next proposition was introduced in \autoref{rem:reg}.
 \begin{proposition}
	\label{p:on_upper}
	There  exist a properly exceptional set $\mathcal N\supset\mathcal N_0$ of $X$ and a positive symmetric kernel $p(t,x,y)$ defined on $(0,\infty)\times(\R^d\setminus\mathcal N)\times(\R^d\setminus\mathcal N)$ of $P_t$ satisfying that 
	\begin{equation*}		p(t+s,x,y)=\int_{\R^d}p(t,x,z)p(s,z,y)dz\qquad\text{ for every } x,y\in\R^d\setminus\mathcal N\text{ and }t,s>0.
	\end{equation*}
	Also there exists a positive constant $c>0$ such that 
	$$p(t,x,y)\leq c[\phi^{-1}(t)]^{-d},\qquad\text{ for any }t>0,\  x,y\in\R^d\setminus\mathcal N.$$
	 Moreover, there is an $\cE$-nest $\{F_k, k\ge 1\}$\footnote{ For the definition of $\cE$-nest, see e.g., \cite[p.69]{MR2778606}.} of compact sets so that $\cN=\Rd\backslash(\bigcup_{k=1}^\infty F_k)$ and that for every $t>0$ and $y\in \Rd\backslash \cN$, $x\to p(t, x, y)$ is continuous on each $F_k$.
\end{proposition}
\begin{proof}
 For any $f\geq0$, consider the semigroup $P_tf(x):=\mathbb E^x[f(X_t)]$ for $x\in\R\setminus \mathcal N_0$. According to \cite[Proposition II.1]{C96}, the Nash-type inequality in \autoref{propNash} implies 
		$$\|P_tf\|_{\infty}\leq m(t)\|f\|_1,\qquad \text{ for any }t>0\text{ and }f\in { L^1(\Rd)},$$
		where $m(t)$ is the inverse function of  $h(t)$ given in the following equation:
		\begin{align*}
			h(t):=&\int_t^\infty \frac{c_2\phi(c_1^{1/d}x^{-1/d})}{x}d x.
		\end{align*}
		Then by \textbf{(WS)},
		\begin{align*}
		h(t)=&\ c_3\int_0^{c_1t^{-1}}\frac{\phi(y^{1/d})}{y}d y
		\leq c_3\sum_{k=0}^{\infty}\frac{\phi(2^{-k/d}c_1^{ 1/d}t^{-1/d})}{2^{-(k+1)}c_1t^{-1}}\cdot (2^{-k}c_1t^{-1}-2^{-(k+1)}c_1t^{-1})\\
		\leq& \ c_3\lC^{-1}\phi(c_1^{1/d}t^{-1/d})
		\sum_{k=0}^\infty2^{-k\la/d}\leq c_4\phi(c_1^{1/d}t^{-1/d}).
		\end{align*}
		Since $\phi$ is increasing, {the inverse function $m(t)\leq c_1[\phi^{-1}(t/c_4)]^{-d}$ and hence}
		$$\|P_tf\|_\infty\leq \frac{c_1}{[\phi^{-1}(t/c_4)]^d}\|f\|_1,\qquad \text{ for any }t>0 \text{ and }f\in {L^1(\Rd)}.$$
		The rest of proof follows from  \cite[ Theorem 3.1]{MR2465826} with \textbf{(WS)}.
\end{proof}

\medskip

\subsection{Scaled process $Y^{(\kappa)}$.}\label{subsec:upall}

For any $\kappa>0$,
we first introduce a $\kappa$-scaled process 
$Y^{(\kappa)}_t:=\kappa^{-1}X_{\phi(\kappa)t}$
with the transition density 
\begin{align}\label{eq:scail_cut}
 q^{(\kappa)}(t, x, y)=\kappa^d p(\phi(\kappa) t, \kappa x, \kappa y).
\end{align}
Let $Q^{(\kappa)}_t$ be  a semigroup and  
\begin{align}\label{eq:bili_cut}
\mathcal{E}_t^{(\kappa)}(f, f):=t^{-1}\langle
f-Q^{(\kappa)}_tf, f\rangle
\end{align}
be a bilinear form corresponding to $Y^{(\kappa)}_t$ and $q^{(\kappa)}(t, x, y)$.

\begin{lemma}\label{lem:4.2CK08}
	The Dirichlet form $(\cW, \mathcal{D}(\cW))$ of $Y^{(\kappa)}$ in $L^2(\Rd)$ is given by 
	\begin{align*}
	\cW(u, u)&:=\int (u(x)-u(y))^2 J^{(\kappa)}(x, y) dx dy,\\
	\mathcal{D}(\cW)&:=\overline{\{u\in C_c(\Rd):  \cW(u, u)<\infty\}}^{\cW_1},	\end{align*}
	where 
	\begin{align*}
	&J^{(\kappa)}(x, y):=\kappa \phi(\kappa)  J(\kappa x, \kappa y)\qquad\text { for every } x, y\in \Rd, 
	\end{align*}
	and $\cW_1(u, u):=\cW(u, u)+\|u\|_2^2$. The Dirichlet form $(\cW, \mathcal{D}(\cW))$ is regular in $L^2(\Rd)$.
\end{lemma} 
\begin{proof}
	For any $t>0$, consider a bilinear form
	\begin{align*}
	\mathcal{E}_t(f, f)
	&:=\frac{1}{2t}\int_\Rd\int_\Rd (f(y)-f(x))^2 p(t, x, y) dydx\,.
	\end{align*}
	By \eqref{eq:scail_cut}--\eqref{eq:bili_cut}, using the change of variables $w=\kappa x$ and $v=\kappa y$ with $g(z)=f(\kappa^{-1}z)$, we have that
	\begin{align*}
	\mathcal{E}_t^{(\kappa)}(f, f)&=\frac{1}{2t}\int_\Rd\int_\Rd (f(y)-f(x))^2 \kappa^d p(\phi(\kappa) t, \kappa x, \kappa y) dydx\\
	&=\frac{\kappa^{-d}}{2t}\int_\Rd\int_\Rd (g(v)-g(w))^2 p(\phi(\kappa) t, w, v) dvdw=\phi(\kappa)\kappa^{-d}\mathcal{E}_{\phi(\kappa)t}(g ,g).
	\end{align*}
	Since $\mathcal{E}_{\phi(\kappa)t}(g ,g)\to \mathcal{E}(g,g)$ as $t\to 0$
	where
	\begin{align*}
	\mathcal{E}(g,g)\asymp 	\mathcal{E}^{\phi}(g,g)&=\int_\Rd\sum_{i=1}^d\int_\R \frac{(g(w+e^i \xi)-g(w))^2}{|\xi|\phi(|\xi|)} d\xi dw\nn\\
	&=\int_\Rd\sum_{i=1}^d\int_\R \frac{(f(\kappa^{-1}(w+e^i \xi))-f(\kappa^{-1}w))^2}{|\xi|\phi(|\xi|)} d\xi dw\\
	&=\kappa^{1+d}\int_\Rd\sum_{i=1}^d\int_\R \frac{(f(x+e^iu))-f(x))^2}{|\kappa u|\phi(|\kappa u|)} du dx,
	\end{align*}
	we have that 
	\begin{align*}
	\mathcal{E}_t^{(\kappa)}(f, f)\to \cW(f, f)=
	\int_\Rd\int_\Rd (f(y)-f(x))^2J^{(\kappa)}(x, y)dx dy,
	\quad \text{ as } t\to 0,
	\end{align*}
	with the jump kernel
	$J^{(\kappa)}(x, y):=\kappa \phi(\kappa)  J(\kappa x, \kappa y)$.
	This follows that $(\cW, \mathcal{D}(\cW))$ is in the class of  $(\mathcal{E}, \mathcal{F})$ corresponding to $Y^{(\kappa)}$, and therefore we obtain our assertion.
\end{proof}

\begin{remark}\label{r:kappa}
By the definition of $J^{(\kappa)}(x,y)=\kappa \phi(\kappa)  J(\kappa x, \kappa y)$, 
	\begin{align}\label{J_kappa}
\small J^{(\kappa)}(x, y)\asymp\begin{cases}
	|x^i-y^i|^{-1}\phi^{(\kappa)}(|x^i-y^i|)^{-1}
	&\small\text {if $x^i \ne y^i$ for some $i$; }\text{$x^j = y^j$ for all $j \ne i$,}\\
	0&\text{otherwise},
	\end{cases}
		\end{align}
	where $\phi^{(\kappa)}(r):=\phi(\kappa r)/\phi(\kappa)$ satisfies {\bf(WS)}.
	So	\autoref{p:on_upper} with \eqref{eq:scail_cut} yields 
	\begin{align*}
	q^{(\kappa)}(t, x, y)\le \frac{c\kappa^d}{[\phi^{-1}(\phi(\kappa)t)]^d}= \frac{c}{[(\phi^{(\kappa)})^{-1}(t)]^d}\quad\text { for } t>0,\  x, y\in \Rd\backslash {\kappa^{-1}\cN},
	\end{align*}
	where $(\phi^{(\kappa)})^{-1}$ is the inverse of $\phi^{(\kappa)}$.
\end{remark}
\medskip

We also introduce the truncated processes. Let $\lambda>0$. Consider the  jump kernel  $J_\lambda(x ,y):=J(x ,y)\1_{\{|x-y|\le \lambda\}}$, and the biliniear form 
%$(\mathcal{E}^{\lambda}, \mathcal{F}^{\lambda})$ where  
	\begin{align*}
	\mathcal{E}^{\lambda} (u,v) :=\int_{\Rd}\Big(\sum_{i=1}^d\int_{\R}\big(u(x+e^i h) - u(x)\big)\big(v(x+e^i h) -  v(x)\big) J_{\lambda} (x,x+e^i h) d h \Big) d x .
	\end{align*}
Then {\bf (WS)} implies
\begin{align}\label{eq:E_EL}
0\le \EE(u, u)-\EE^\lambda(u,u)
=&\int_{\Rd}\Big(\sum_{i=1}^d\int_{|h|\ge \lambda}\big(u(x+e^i h) - u(x)\big)^2\frac{1}{|h|\phi(|h|)} dh \Big) d x\nn\\
\le & \ \frac{c_1}{\phi(\lambda)} \|u\|_2^2\int_\lambda^\infty |h|^{-1-\la} dh \le \frac{c_1}{\la \phi(\lambda)}\|u\|_2^2,
\end{align}
therefore,
\begin{align}\label{eq:comE}
 \EE^\lambda_1(u, u)\le \EE_1(u, u)\le(1+c_2) \EE^\lambda_1(u, u)\qquad \text{ for every } u\in \FF.
\end{align}
It follows that $(\EE^\lambda, \FF)$ is a regular Dirichlet form on $L^2(\Rd)$, so there is a Hunt processes $X^\lambda$ corresponding to $(\EE^\lambda, \FF)$.
In addition, \autoref{propNash} with \eqref{eq:comE} implies $X^\lambda$ has a transition density function $p_\lambda(t, x, y)$.
By the similar proof of the on-diagonal upper bounds in \autoref{p:on_upper} and {\bf(WS)},  there exists an exceptional set $\cN$ and a constant $c>0$ such that  
\begin{align}\label{eq:neardia_lam}
	p_\lambda(t, x, y)\le c [\phi^{-1}(t)]^{-d}\quad\text { for } t>0,\  x, y\in \Rd\backslash \cN.
\end{align}
	
\medskip
Now we give the proof of the conservativeness for $X$ in \autoref{mainthm}.
\begin{theorem}\label{thm:cons}
	The process $X$ is conservative, that is, $X$ has infinite lifetime.
\end{theorem}
\begin{proof}
	By \cite[Theorem 2.4]{MU11}, 
	the truncated process $X^1$ corresponding to $(\mathcal{E}^1, \mathcal{F} )$  is conservative.
	Note that $J_1(x, y)\le J(x, y)$, and for
    \begin{align}\label{eq:cJ}
	\mathfrak J_1(x,y):=J(x,y)-J_1(x,y)
	\end{align}
	$\int_\Rd \mathfrak J_1(x,y) dy\le c \int_{\{|u|\ge 1\}}\frac{1}{ |u|\phi(|u|)}du<\infty$.
	By the Meyer's construction(see, \cite[Section 4.1]{ChKu08}), the processes $X$ can be constructed from $X^1$.
	Therefore, the conservativeness of $X$ follows from the conservativeness of $X^1$.
\end{proof}

\medskip

In the following, we first obtain the distribution of exit times for the
$\kappa$-scaled process $Y^{(\kappa)}$, and then present that of exit times for $X$ in \autoref{cor:exit_new}. 
 For any Borel set $A\subset\R^d$,  denote by $\tau_A^{(\kappa)}:=\inf\{t>0:Y_t^{(\kappa)}\notin A\}$  the first exit time of $Y^{(\kappa)}$ from $A$ and the following proposition gives the distribution of $\tau_\cdot^{(\kappa)}$.
\begin{proposition}\label{p:exit_new}
	Let $T>0$. Then there exists a positive constant $c=c(T, \phi)>0$ such that
	\begin{align*}
	\Pp^x\left(\tau^{(\kappa)}_{B(x, r)}< 1\right)\le \frac{c}{\phi^{(\kappa)}(r)}
	\end{align*}
	for all $r\in [T, \infty)$ and $x\in \Rd\backslash\kappa^{-1}\cN$.
\end{proposition}
\begin{proof}
For any $\lambda>0$, we consider the truncated jump kernel $J^{(\kappa)}_\lambda(x ,y):=J^{(\kappa)}(x ,y)\1_{\{|x-y|\le \lambda\}}$ for $Y^{(\kappa)}$.
	Using the similar argument as in \eqref{eq:E_EL} with $\cW^\lambda(u, u):=\int (u(x)-u(y))^2 J^{(\kappa)}_\lambda(x, y) dx dy$,
	we obtain 
	\begin{align}\label{eq:W_WL}
	0\le \cW(u, u)-\cW^\lambda(u,u) \le \frac{c_1}{\la \phi^{(\kappa)}(\lambda)}\|u\|_2^2,
	\end{align}
	so that $\cW^\lambda_1(u, u)\asymp \cW_1(u, u)$ for $u\in \mathcal{D}(\cW)$.
	Therefore, $(\cW^\lambda, \mathcal{D}(\cW))$ is a regular Dirichlet form and there is a Hunt process $Y^{(\kappa),\lambda}$ corresponding to $(\cW^\lambda, \mathcal{D}(\cW))$. 
	In addition, by \eqref{eq:W_WL} and the same observation of \autoref{propNash} corresponding to $\cW$,  Nash-type inequality holds, and there exists a transition density $q_\lambda^{(\kappa)}(t, x, y)$ for $Y^{(\kappa), \lambda}$. Also by  \eqref{eq:scail_cut}, \eqref{eq:neardia_lam} and {\bf(WS)} condition of $\phi^{(\kappa)}(r)$, we have that for any $x, y\in \Rd\backslash\kappa^{-1}\cN$ and $t\in (0, 2]$,
	\begin{align}\label{eq:neardia_lam_ka}
	q^{(\kappa)}_\lambda(t, x, y)\le c_1 [(\phi^{(\kappa)})^{-1}(t)]^{-d} \le c_2t^{-d/\la}
	\end{align}
for some $c_2:=c_2(\phi)>0$.
	Then \eqref{eq:neardia_lam_ka} together with \cite[Theorem 3.2]{MR2465826} and
	\cite[Theorem 3.25]{CKS87}  implies that there exist constants $c_3, c_4>0$ such that 
	\begin{align}\label{eq:pka}
	q^{(\kappa)}_\lambda(t, x, y)\le c_3 t^{-d/\la}\exp\left(-|\chi(y)-\chi(x)|+c_4 \Gamma(\chi)^2 t\right)
	\end{align}
	for all $t\in (0, 2]$ and $x, y\in \Rd\backslash\kappa^{-1}\cN$.
	Here
	\begin{align*}
	\chi(\xi)&:=\frac\beta{3}\big(|\xi-x|\wedge |x-y|\big),\qquad\qquad\qquad\qquad\quad   \xi\in \Rd,\\
	\Gamma(\chi)^2&:=\|e^{-2\chi}\Gamma_{\lambda}(e^{\chi})\|_{\infty}\vee \|e^{-2\chi}\Gamma_{\lambda}(e^{-\chi})\|_{\infty},\\
	\Gamma_{\lambda}(v)(\xi)&:=\int_{\{|\xi-\eta|\le \lambda\}}\big(v(\xi)-v(\eta)\big)^2 J^{(\kappa)}(\xi, \eta) d\eta \,,\qquad \xi\in \Rd.
	\end{align*}
	We will choose the constant $\beta>0$ later.
	Since $|\chi(\eta)-\chi(\xi)|\le (\beta/3)|\eta-\xi|$ for $\xi, \eta\in \Rd$, by \eqref{J_kappa} and {\bf(WS)} condition of $\phi^{(\kappa)}$,
	we have that for $\xi\in \Rd$,
	\begin{align*}
	\left(e^{-2\chi}\Gamma_{\lambda}(e^{\chi})\right)(\xi)
	=& \ c_5\ \int_{\{|\eta-\xi|\le \lambda\}}
	(1-e^{\chi(\eta)-\chi(\xi)})^2J^{(\kappa)}(\xi, \eta) d\eta\\
	\le &\ c_6 \sum_{i=1}^d \int_{\{|h|\le \lambda\}}({\chi(\xi+he^i)-\chi(\xi)})^2
	e^{2|\chi(\xi+he^i)-\chi(\xi)|}J^{(\kappa)}(\xi, \xi+he^i) dh\\
	\le &\ c_7 \left(\frac\beta{3}\right)^2e^{2\beta\lambda/3}\int_{\{|h|\le \lambda\}}\frac{|h|}{\phi^{(\kappa)}(|h|)} dh	\le  \ c_8e^{2\lambda/3}\frac{(\beta\lambda)^2}{\phi^{(\kappa)}(\lambda)}\le  \ c_9  \frac{e^{\beta\lambda}}{\phi^{(\kappa)}(\lambda)}.
	\end{align*}
	Similarly, we obtain the upper bound of $\left(e^{2\chi}\Gamma_{\lambda}(e^{-\chi})\right)(\xi)$. 	
Now  for any $r, t>0$, we choose 
	$$\lambda:=\frac{r\la}{3(d+\la)}\quad\text{and }\quad \beta:=\frac{1}{\lambda}\log\left(\frac{\phi^{(\kappa)}(\lambda)}{t}\right).$$
Then 
	\begin{align}\label{eq:exp}
-|\chi(y)-\chi(x)|+c_4 \Gamma(\chi)^2 t\le \frac{-\beta|x-y|}{3}+\frac{c_{10} t e^{\beta\lambda}}{\phi^{(\kappa)}(\lambda)}= \log \left(\frac{t}{\phi^{(\kappa)}(\lambda)}\right)^{\frac{|x-y|}{3\lambda}}+c_{10}.
\end{align}	
Therefore,  \eqref{eq:pka} and \eqref{eq:exp} with $\lambda:=\frac{r\la}{3(d+\la)}$  imply that for $x, y\in \Rd\backslash\kappa^{-1}\cN$ and $t\in (0, 2]$,
\begin{align*}
q^{(\kappa)}_\lambda(t, x, y)\le c_{11} t^{-d/\la}\left(\frac{t}{\phi^{(\kappa)}(\lambda)}\right)^{\frac{d+\la}{\la}\cdot\frac{|x-y|}{r}}.
\end{align*}
	Let $\zeta^{(\kappa),\lambda}$ be the lifetime of $Y^{(\kappa), \lambda}$.
Using the same proof in  \autoref{thm:cons}, $Y^{(\kappa), \lambda}$ is conservative, and $\Pp^x(\zeta^{(\kappa),\lambda}\le r)=0$ for any $r>0$.
Let $T>0$. Then for any $x\in \Rd$, $ t\in(0, 2)$
 and $r\in [T, \infty)$,
\begin{align*}
\Pp^x\big(|Y^{(\kappa), \lambda}_{t}-x|\ge r\big)&=\int_{B(x, r)^c} q^{(\kappa)}_\lambda(t, x, y) dy +\Pp^x({\zeta^{(\kappa),\lambda}}\le r)\nn\\
&\le c_{11} t^{-d/\la} \int_{B(x, r)^c} \left(\frac{t}{\phi^{(\kappa)}(\lambda)}\right)^{\frac{d+\la}{\la}\cdot\frac{|x-y|}{r}}dy \le \frac{c_{12}tr^d}{\phi^{(\kappa)}(\lambda)^{d/\la+1}}\le \frac{c_{13}t}{\phi^{(\kappa)}(r)}.
\end{align*}
The last inequality follows from $\lambda=\frac{r\la}{3(d+\la)}$ and \textbf{(WS)} for $\phi^{(\kappa)}$. 

Let $\sigma_r:=\sigma_r^{(\kappa),\lambda}:=\inf\{s\ge 0:|Y^{(\kappa), \lambda}_{s}-Y^{(\kappa), \lambda}_{0}|>r\}$. Then the strong Markov property implies that
\begin{align*}
\Pp^x\big(\sigma_r<1, |Y^{(\kappa), \lambda}_{2}-x|\le r/2\big)&\le 
\Pp^x\big(\sigma_r<1, |Y^{(\kappa), \lambda}_{2}-Y^{(\kappa), \lambda}_{\sigma_r}|> r/2\big)\nn\\
&\le \E^x\left[\1_{\{\sigma_r<1\}}\ \Pp^{Y^{(\kappa), \lambda}_{\sigma_r}}( |Y^{(\kappa), \lambda}_{2-\sigma_r}-x|> r/2)\right]\nn\\
&\le \sup_{\substack{l<1\ \&\\ y\in\{z\in\Rd:|x-z|>r\} }} \Pp^y\big(|Y^{(\kappa), \lambda}_{2-l}-x|>r/2\big).
\end{align*}
Therefore, {\bf (WS)} for $\phi^{(\kappa)}$ with the above two inequalities yields that
for any  $x\in \Rd$ and $r\in [T, \infty)$,
\begin{align*}
\Pp^x\big(\sup_{s\le 1}|Y^{(\kappa),\lambda}_s-x|>r\big)&=
\Pp^x\left(\sigma_r<1\right)\nn\\
&\le \Pp^x\big(\sigma_r<1, |Y^{(\kappa), \lambda}_{2}-x|\le r/2\big)+\Pp^x\big(|Y^{(\kappa), \lambda}_{2}-x|\ge r/2\big)\nn\\   
&\le \sup_{\substack{y\in\{z\in\Rd:|x-z|>r\}\\ l<1 }} \Pp^y\big(|Y^{(\kappa), \lambda}_{2-l}-y|>r/2\big)+\frac{2c_{13}}{\phi^{(\kappa)}(r/2)}\le \frac{c_{14}}{\phi^{(\kappa)}(r)}.
\end{align*}
Also note that
\begin{align*}
	\Pp^x\big(Y^{(\kappa)}\neq Y^{(\kappa),\lambda}\, \text{ for some }s\le t\big)
	\le t\sup_{x'}\int\big|J^{(\kappa)}(x', y)-J^{(\kappa)}_\lambda(x', y)\big|dy\leq \frac{c_{15}t}{\phi^{(\kappa)}(\lambda)}\leq \frac{c_{16}t}{\phi^{(\kappa)}(r)}.
\end{align*}
Therefore,	for any $x\in \Rd$ and $r\in [T, \infty)$,
\begin{align*}
\Pp^x\left(\tau^{(\kappa)}_{B(x, r)}< 1\right)&= 	\Pp^x\big(\sup_{s\le 1}|Y^{(\kappa)}_s-Y^{(\kappa)}_0|>r\big)\\
&\le \Pp^x\big(\sup_{s\le 1}|Y^{(\kappa),\lambda}_s-x|>r\big)+\Pp^x\big(Y^{(\kappa)}\neq Y^{(\kappa),\lambda}\, \text{ for some }s\le 1\big)\\
&\le   \frac{c_{17}}{\phi^{(\kappa)}(r)}.	
\end{align*}

\end{proof}

Let $\tau_A:=\inf\{t>0:X_t\notin A\}$ be the first exit time of $X$ from $A$. Using the relation between $Y^{(\kappa)}$ and $X$, we have the following distribution for $\tau_\cdot$.
\begin{corollary}\label{cor:exit_new}
Let $T>0$. Then there exists $c=c(T, \phi)>0$ such that for any $x\in \Rd\backslash\cN$, $r\in [T, \infty)$ and $t>0$,
% with $\kappa=\phi^{-1}(t)$,
$$\Pp^x\left(\tau_{B(x, r \phi^{-1}(t))}<t\right)\leq \frac{c t}{\phi(\phi^{-1}(t)r)}.$$
\end{corollary}
\begin{proof}
For any $\kappa>0$, let $\tau^{(\kappa)}$ and $\tau$ be the first exit times for $Y^{(\kappa)}$ and $X$, receptively.
Then for any $x\in \Rd\backslash \kappa^{-1}\cN$ and $r>0$, we have that
	\begin{align}\label{eq:exit_ka_n}
	\Pp^x\left(\tau^{(\kappa)}_{B(x, r)}< 1\right)
	=&\	\Pp^x\left(\sup_{s\le 1}|Y^{(\kappa)}_s-x|>r\right)\nn\\
	%=\Pp^{\kappa x}\left(\sup_{u\le \phi(\kappa)}|\kappa Y^{(\kappa)}_{u/\phi(\kappa)}-\kappa x|>r\kappa \right)\nn\\
	=&\ \Pp^{ w}\left(\sup_{u\le \phi(\kappa)}|X_{u}-{w}|>r\kappa\right)
	=\Pp^w\left(\tau_{B(w, r \kappa )}<\phi(\kappa)\right)
	\end{align}
for $w\in \Rd\backslash \cN$.
 Let $T>0$. For any $t>0$, applying \autoref{p:exit_new} and \eqref{eq:exit_ka_n} with $\kappa=\phi^{-1}(t)$, there exists $c=c(T, \phi)>0$ such that for any $w\in \Rd\backslash \cN$ and $r\in [T, \infty)$,
	$$	\Pp^{w}\left(\tau_{B({ w}, r \phi^{-1}(t))}<t\right)
	= \Pp^x\left(\tau^{(\phi^{-1}(t))}_{B(x, r)}< 1\right)\le \frac{c}{\phi^{(\phi^{-1}(t))}(r)}=\frac{c t}{\phi(\phi^{-1}(t)r)}.$$
\end{proof}

\begin{remark}
	In the following sections, lower and upper heat kernel estimates for $p(t,x,y)$ will be given for $x,y\in\R^d\setminus \mathcal N$. In \autoref{thm:Holder}, we will show that $p(t,x,y)$ is H\"older continuous in $(x,y)$. Hence, it can be extended continuously to $\R_+\times\R^d\times\R^d$, which completes the proof of \autoref{mainthm}. In the remaining sections, for the convenience of discussions, we will directly write $\R^d$ instead of $\R^d\setminus\mathcal N$. 
\end{remark}

%%%%%%%%%%%%%%
\section{Upper bound estimates}\label{sec:uphke}

In this section, we introduce the upper bound estimates and give the sketch of the proof with a technical result \autoref{prop:main}. The proof of \autoref{prop:main} is similar to  \cite[Proposition 3.3]{KKKpre} with some modifications, we postpone it to Appendix.

\subsection{Upper bounds and sketch of the proof}
\begin{theorem}
	\label{thm:uphke}
	There is a positive constant $C=C(d,\Lambda, \phi)$ such that for all  $t>0$ and $x, y\in\R^d$ the following estimate holds: 
	\begin{align*}
	%%KYdelet	%\label{eq:upper-bd-mulit}
	p(t, x, y)\le C [\phi^{-1}(t)]^{-d}\prod_{i=1}^d\Big(1\wedge \frac{t\phi^{-1}(t)}{|x^i-y^i|\phi(|x^i-y^i|)}\Big)\,.
	\end{align*}
\end{theorem}

We first explain the sketch of the proof for \autoref{thm:uphke}. 
For any $q > 0$ and $l\in \{1, \ldots, d-1\}$,  consider the following conditions.

\medskip

\begin{itemize}
	\item[{\bf $\HHq{q}{0}$}]
	There exists a positive constant 
	$C_0=C_0(q,\Lambda, d, \phi)$ such that for all $t>0$ and $x, y\in \R^d$,
	\begin{align*}
	%\label{eq:H0}
	p(t, x, y)\le C_0 [\phi^{-1}(t)]^{-d}\prod_{i=1}^d\Big(1\wedge \frac{t\phi^{-1}(t)}{|x^i-y^i|\phi(|x^i-y^i|)}\Big)^{q}.
	\end{align*}
	\item[{\bf $\HHq{q}{l}$}]
	There exists a positive constant $C_l=C_l(l,q,\Lambda, d, \phi)$ such that for all $t>0$ and all $x, y\in \R^d$ with $|x^{1}-y^{1}| \leq \ldots \leq |x^{d}-y^{d}|$, the following holds true:
	\begin{align}
	p(t, x, y)&\le C_l  [\phi^{-1}(t)]^{-d}
	\nn\\
	&\times\prod_{i=1}^{d-l}\Big(1\wedge \frac{t\phi^{-1}(t)}{|x^i-y^i|\phi(|x^i-y^i|)}\Big)^{q}
	\prod_{i= d-l+1}^d\Big(1\wedge \frac{t\phi^{-1}(t)}{|x^i-y^i|\phi(|x^i-y^i|)}\Big)\,. \label{eq:Hi}
	\end{align}
\end{itemize}

Then we have the following Remark easily.
\begin{remark}\label{rem:H-fundamentals}{\ } 
	\begin{enumerate}
		\item{} The assertion of \autoref{thm:uphke} is equivalent to $\HHq{1}{l}$ for any $l\in \{0, \ldots, d-1\}$.
		\item{}	
		%$\HHq{q}{l}$ gets stronger  as $q$ increases, 
		%when $q$ is increasing, 
		For $q<q'$, $\HHq{q'}{l}$ implies $\HHq{q}{l}$.
		\item{} For $q\in [0, 1]$ and $l\le l'$, $\HHq{q}{l'}$ implies $\HHq{q}{l}$.
	\end{enumerate}
\end{remark}

By the similar arguments as in \cite[Section 5]{KKKpre}, we have 
the following Lemma.

\begin{lemma}\label{lem:symmetry}
	Assume condition $\HHq{q}{l}$ holds true for some $q\in[0, 1]$. Then with the same constant 
	$C_l=C_l(l,q,\Lambda,d,\phi)$, 
	for every $t>0$ and $x, y\in \R^d$, and any   permutation $\sigma$ satisfying $|x^{\sigma(1)}-y^{\sigma(1)}| \leq \ldots \leq |x^{\sigma(d)}-y^{\sigma(d)}|$,
	the following holds true:
	\begin{align*}
	p(t, x, y) \le C_l  [\phi^{-1}(t)]^{-d}
	&\prod_{i=1}^{d-l}\Big(1\wedge \frac{t\phi^{-1}(t)}{|x^{\sigma(i)}-y^{\sigma(i)}|\phi(|x^{\sigma(i)}-y^{\sigma(i)}|)}\Big)^{q}\nn\\
	&\times\prod_{i= d-l+1}^d\Big(1\wedge \frac{t\phi^{-1}(t)}{|x^{\sigma(i)}-y^{\sigma(i)}|\phi(|x^{\sigma(i)}-y^{\sigma(i)}|)}\Big)\\
	%\label{eq:Hi-permute} \\
	\le C_l  [\phi^{-1}(t)]^{-d}
	&\prod_{i=1}^{d-l}\Big(1\wedge \frac{t\phi^{-1}(t)}{|x^i-y^i|\phi(|x^i-y^i|)}\Big)^{q}
	\prod_{i= d-l+1}^d\Big(1\wedge \frac{t\phi^{-1}(t)}{|x^i-y^i|\phi(|x^i-y^i|)}\Big). \end{align*}
\end{lemma}
\medskip

\autoref{lem:symmetry} tells us that the condition $\HHq{q}{l}$  with $q\in[0, 1]$ implies condition $\HHq{q}{l}'$  with $q\in[0, 1]$, which is defined as follows:
\begin{itemize}
	\item [{\bf $\HHq{q}{l}'$}]
	Given $q > 0$ and $l\in \{1, \ldots, d-1\}$, there exists a positive constant 
	$C_l=C_l(l,q,\Lambda,d,\phi)$ 
	%$C_l=C(l,q,\Lambda)$ 
	such that for all $t>0$ and $x, y\in \R^d$,
	\begin{align*} 
	p(t, x, y)&\le C_l  [\phi^{-1}(t)]^{-d}\nn\\
	&\times
	\prod_{i=1}^{d-l}\Big(1\wedge \frac{t\phi^{-1}(t)}{|x^i-y^i|\phi(|x^i-y^i|)}\Big)^{q}
	\prod_{i= d-l+1}^d\Big(1\wedge \frac{t\phi^{-1}(t)}{|x^i-y^i|\phi(|x^i-y^i|)}\Big). 	\end{align*}
\end{itemize}
Therefore, it suffices to consider  special points $x, y\in \Rd$ satisfying conditions in $\HHq{q}{l}$.
%\end{remark}

\medskip

Next we introduce the following three lemmas.
\begin{lemma}\label{lem:Gl}
	Assume condition $\HHq{q}{l}$ holds true for some $l\in\{0, \ldots, d-2\}$ and  $q\in\big[0, \frac{1}{1+\newa}\big)$. Then there exists $\theta_l>0$ such that $\HHq{q+\theta_l}{l}$  holds true.
\end{lemma}

\begin{lemma}\label{lem:Fl}
	Assume condition $\HHq{q}{l}$ holds true for some $l\in\{0, \ldots, d-2\}$ and  $q\in\big(\frac{1}{1+\newa}, 1\big)$. Then $\HHq{0}{l+1}$ holds true.
\end{lemma}

\begin{lemma}\label{lem:GF-dminus1}	{\ } \\
	(i) Assume condition $\HHq{q}{d-1}$ holds true for some $q\in\big[0, \frac{1}{1+\newa}\big)$. Then there exists $\theta_{d-1}>0$ such that 
	$\HHq{q+\theta_{d-1}}{d-1}$ holds true. \\
	(ii) Assume condition $\HHq{q}{d-1}$ holds true for some $q\in\big(\frac{1}{1+\newa}, 1\big)$. Then $\HHq{1}{d-1}$ holds true.
\end{lemma}

Assuming the above Lemmas,  we finalize our assertion $\HHq{1}{d-1}$ in the following order:
\begin{alignat*}{5}
&\HHq{0}{0} &&\hookrightarrow \HHq{\theta_0}{0} &&\hookrightarrow \HHq{2\theta_0}{0} \ldots &&\hookrightarrow \HHq{N_0\theta_0}{0} &&\\
\hookrightarrow &\HHq{0}{1} &&\hookrightarrow \HHq{\theta_1}{1} &&\hookrightarrow \,\, \ldots  \ldots \ldots \,\,  &&\hookrightarrow \HHq{N_1\theta_1}{1} &&\\
&&&\vdots&&&\vdots&&&\\
\hookrightarrow &\HHq{0}{d-1} &&\hookrightarrow \HHq{\theta_{d-1}}{d-1} &&\hookrightarrow \ldots  \ldots\ldots &&\hookrightarrow \HHq{N_{d-1}\theta_{d-1}}{d-1} && \hookrightarrow \HHq{1}{d-1}
\end{alignat*}
where $ N_l:=\inf\{n\in \N_+:n\theta_l\geq \frac{1}{1+\la}\}$ for $l\in \{0, \ldots, d-1\}$. 
Note that $\HHq{0}{0}$ is proved in \autoref{p:on_upper}.
We will give the detailed proofs of \autoref{lem:Gl}, \autoref{lem:Fl} and \autoref{lem:GF-dminus1} in  \autoref{subsec:proofs-lemmas} with specific definitions of $\theta_l$ and $\theta_{d-1}$,
\begin{align*}
& \theta_l:=\frac{\la}{2+\la+\ua}\left(\sum_{i=1}^{d-l-1}\Big(\frac{\ua+1}{\la}\Big)^i\right)^{-1}\ \mbox{for} \ l\in \{0,1,\ldots, d-2\},\\
&\theta_{d-1}:=\newb.
\end{align*}
In the following, we introduce the concrete condition of two points $x_0, y_0\in \Rd$ according to the time variable $ \kappa:=\phi^{-1}(t)$.
\begin{definition}\label{def:theta_and_R}
	Let $x_0, y_0 \in \R^d$ satisfy $|x_0^{i}-y_0^{i}|\le |x_0^{i+1}-y_0^{i+1}|$ for every $i\in\{1,\ldots, d-1\}$. 
	For any $t>0$, set $\kappa:=\phi^{-1}(t)$. 
	For $i\in\{1,\ldots, d\}$, define $\lengthR_i\in \Z$ and $R_i > 0$ such that 
	\begin{align}
	\label{d:thi_Ri}
	\tfrac{5}{4}2^{\lengthR_i} \kappa\le  |x_0^{i}-y_0^{i}|< \tfrac{10}{4}2^{\lengthR_i}\kappa \qquad\mbox{and}\qquad	R_i=2^{\lengthR_i}\kappa\,.
	\end{align}
	Then $\lengthR_i\le \lengthR_{i+1}$ and $R_i\le R_{i+1}$. We say that condition $\mathcal{R}(i_0)$ holds if 
	\begin{align*}
	%\label{eq:case-R-i0}
	\lengthR_1\le \ldots \le \lengthR_{i_0-1}\le 0<1\le \lengthR_{i_0}\le \ldots\le \lengthR_d\, 
	\tag*{$\mathcal{R}(i_0)$} .
	\end{align*}
	We say that condition $\mathcal{R}(d+1)$ holds if $\lengthR_1\le \ldots \le \lengthR_d \leq 0 < 1$.
\end{definition}

The following Lemma tells us that in specific geometric situations,  \autoref{thm:uphke} follows directly.

\begin{lemma}\label{lem:reduction_cases}
	Let $t>0$ and $x_0, y_0\in\R^d$.  Assume  \eqref{eq:Hi}
	holds for some $l \in \{0,\ldots, d-1\}$ and $q \geq 0$. In addition, if $x_0,y_0$ satisfy condition $\mathcal{R}(i_0)$ for $i_0 \in \{d-l+1, \ldots, d+1\}$.
	Then 
	\begin{align}\label{eq:reduction_cases}
	p(t, x_0, y_0) \le c  [\phi^{-1}(t)]^{-d}\prod_{i=1}^{d} \Big(1\wedge \frac{t\phi^{-1}(t)}{ |x_0^i-y_0^i|\phi(|x_0^i-y_0^i|)}\Big)
	\end{align}
	for some constant $c=c(l,q,d,\Lambda,\phi)> 0$, that is, independent of $t$ and $x_0, y_0$.
\end{lemma}
\begin{proof}
	For $t>0$, let $\kappa:=\phi^{-1}(t)$.
	(i) The case $\mathcal{R}(d+1)$ (that is, $|x_0^{d}-y_0^{d}|< \tfrac{10}{4}  \kappa^{-1} $) is simple since \autoref{p:on_upper} implies $p(t, x_0, y_0) \le c  \kappa^{-d}$. Estimate \eqref{eq:reduction_cases} follows directly.  
	
	(ii) Assume condition $\mathcal{R}(i_0)$ holds true for some $i_0 \in \{d-l+1, \ldots, d\}$. Then for any $j\le d-l\le i_0-1$ 
	%by \eqref{eq:phi_theta}, 
	since $|x_0^{j}-y_0^{j}|<\tfrac{10}{4}2^{\lengthR_j}\kappa \le \tfrac{10}{4}2^{\lengthR_{i_0-1}}\kappa\le \tfrac{10}{4}\kappa$,
	\begin{align*}
	\Big(\frac{t\phi^{-1}(t)}{|x_0^j-y_0^j|\phi(|x_0^j-y_0^j|)}\wedge 1\Big)\asymp 1 \,.
	\end{align*}
	Hence, $\eqref{eq:Hi}$ implies 
	\begin{align*}
	p(t, x_0, y_0) &\le C_l  [\phi^{-1}(t)]^{-d}\prod_{i=1}^{d-l}	\Big(\frac{t\phi^{-1}(t)}{|x_0^j-y_0^j|\phi(|x_0^j-y_0^j|)}\wedge 1\Big)^{q} \prod_{i= d-l+1}^d	\Big(\frac{t\phi^{-1}(t)}{|x_0^j-y_0^j|\phi(|x_0^j-y_0^j|)}\wedge 1\Big) \\
	&\asymp  [\phi^{-1}(t)]^{-d}\prod_{i=1}^{d} 	\Big(\frac{t\phi^{-1}(t)}{|x_0^j-y_0^j|\phi(|x_0^j-y_0^j|)}\wedge 1\Big)\,.
	\end{align*}		
\end{proof}

By \autoref{lem:reduction_cases}, when $\HHq{q}{l}$ holds, it suffices to consider the cases $\mathcal{R}(i_0)$ for $i_0\in \{1,\ldots, d-l\}$. Here is our main technical result.

\begin{proposition}\label{prop:main}
	Assume that  $\HHq{q}{l}$ holds true for some $l\in \{0,1,\ldots, d-1\}$ and  $q \in [0,1)$.
	For any $t>0$, set $\kappa= \phi^{-1}(t)$ and consider  $x_0, y_0\in \Rd$ satisfying the condition $\mathcal{R}(i_0)$ for some $i_0 \in \{1, \dots, d-l\}$.
	Let $j_0 \in \{i_0,\ldots, d-l\}$ and define an exit time $\tau=\tau_{B(x_0, 2^{\lengthR_{j_0}}\kappa/8)}$.
	Let $f$ be a non-negative Borel function on $\Rd$ supported in $B(y_0, \tfrac{\kappa}{8})$.
	Then there exists $C_{\ref{prop:main}}>0$ independent of  $x_0, y_0$ and $t$ such that for every $x\in B(x_0, \frac{\kappa}{8})$,
	\begin{align}\label{eq:main1}
	&\E^{x}\left[\1_{\{\tau\le t/2\}}P_{t-\tau}f(X_{\tau})\right]\le C_{\ref{prop:main}} \|f\|_1 [\phi^{-1}(t)]^{-d}
	\prod_{j=d-l+1}^d\Big(\frac{t\phi^{-1}(t)}{|x_0^j-y_0^j|\phi(|x_0^j-y_0^j|)}\wedge 1\Big)\\
	&\cdot\prod_{j=j_0+1}^{d-l}\Big(\frac{t\phi^{-1}(t)}{|x_0^j-y_0^j|\phi(|x_0^j-y_0^j|)}\wedge 1\Big)^q\cdot
	\begin{cases}
	\Big(\frac{t\phi^{-1}(t)}{|x_0^{j_0}-y_0^{j_0}|\phi(|x_0^{j_0}-y_0^{j_0}|)}\wedge 1\Big)^{\newb+q}
	&\mbox{ if } q\in\big[0, \tfrac{1}{1+\newa}\big)\\
	\Big(\frac{t\phi^{-1}(t)}{|x_0^{j_0}-y_0^{j_0}|\phi(|x_0^{j_0}-y_0^{j_0}|)}\wedge 1\Big)
	&\mbox{ if } q\in\big(\frac{1}{1+\newa}, 1\big)
	\end{cases}.\nn
	\end{align}
\end{proposition}

\subsection{Proofs of \autoref{lem:Gl}, \autoref{lem:Fl} and \autoref{lem:GF-dminus1}} \label{subsec:proofs-lemmas} 
In this subsection,
we give the proofs of \autoref{lem:Gl}, \autoref{lem:Fl} and \autoref{lem:GF-dminus1} using \autoref{prop:main}. 
We first recall the following result.
\begin{lemma}\cite[Lemma 2.1]{BGK09}\label{lem:LBGK}
	Let $U$ and $V$ be two disjoint non-empty open subsets of $\Rd$ and $f, g$ be non-negative Borel functions on $\Rd$. 
	Let $\tau=\tau_U$ and $\tau^{'}=\tau_V$ be the first exit times from $U$ and $V$, respectively. Then, for all $a, b, t\in \R_+$ such that $a+b=t$, we have
	\begin{align}\label{eq:LBGK}
	\langle P_tf, g\rangle \le \langle \E^{ \cdot}\big[\1_{\{\tau\le a\}}P_{t-\tau}f(X_{\tau})\big], g\rangle+ \langle\E^{ \cdot}\big[\1_{\{\tau^{'}\le b\}}P_{ t-\tau'}g(X_{\tau^{'}})\big], f\rangle\,.
	\end{align}
\end{lemma}

Now we apply \autoref{lem:LBGK} with non-negative Borel functions $f, g$ on $\R^d$ supported in $B(y_0, \frac{\kappa}{8})$ and $B(x_0, \frac{\kappa}{8})$, respectively, and  subsets $U:=B(x_0, s), V:=B(y_0, s)$ for some $s>0$, $a=b=t/2$ and $\tau=\tau_{U}, \tau^{'}=\tau_{V}$.
The first term of the right-hand side of \eqref{eq:LBGK} is
\begin{align}
\label{eq:m1}
\langle\E^{\cdot}\left[\1_{\{\tau\le t/2\}}P_{t-\tau}f(X_{\tau})\right], g\rangle=\int_{B(x_0, \frac{\kappa}{8})}\E^{x}\left[\1_{\{\tau\le t/2\}}P_{t-\tau}f(X_{\tau})\right]\,g(x)dx ,
\end{align}
and a similar identity holds for the second term. 

\medskip

\begin{proof}[Proof of \autoref{lem:GF-dminus1}]
	By \autoref{lem:reduction_cases}, for any $t>0$, we only consider the case that $x_0, y_0$ satisfy condition $\mathcal{R}(1)$. 
	Recall that $\kappa=\phi^{-1}(t)$ and $R_1=2^{\lengthR_1}\kappa$. Applying \autoref{prop:main} with $i_0=j_0=1$, 
	for $x\in B(x_0, \frac{\kappa}{8})$ and $\tau=\tau_{B(x_0, {R_{1}}/{8})}$, we obtain
	\begin{align*}
	\E^{x}\left[\1_{\{\tau\le t/2\}}P_{t-\tau}f(X_{\tau})\right]
	\le C_{\ref{prop:main}}&[\phi^{-1}(t)]^{-d}\|f\|_1 \mathbf{K}_{1}(t, x_0, y_0),
	\end{align*}
	where
	$${\mathbf K}_{1}(t, x_0, y_0):=	\prod_{j=2}^d\Big(\frac{t\phi^{-1}(t)}{|x_0^{j}-y_0^{j}|\phi(|x_0^{j}-y_0^{j}|)}\Big)
	\begin{cases}
	\Big(\frac{t\phi^{-1}(t)}{|x_0^{1}-y_0^{1}|\phi(|x_0^{1}-y_0^{1}|)}\Big)^{\newb+q}
	&\mbox{ if } q\in \big[0, \frac{1}{1+\newa}\big),\\
	\Big(\frac{t\phi^{-1}(t)}{|x_0^{1}-y_0^{1}|\phi(|x_0^{1}-y_0^{1}|)}\Big)
	&\mbox{ if } q\in \big(\frac{1}{1+\newa}, 1\big),
	\end{cases}$$
	and by \eqref{eq:m1}, 
	\begin{align*}
	&\langle\E^{\cdot}\left[\1_{\{\tau\le t/2\}}P_{t-\tau}f(X_{\tau})\right], g\rangle\le  c[\phi^{-1}(t)]^{-d}\|f\|_1\|g\|_1{\mathbf K}_{1}(t, x_0, y_0).
	\end{align*}
	Similarly we obtain the second term of right-hand side in \eqref{eq:LBGK} and therefore,
	\begin{align*}
	\langle P_tf, g\rangle &\le c[\phi^{-1}(t)]^{-d}\|f\|_1\|g\|_1{\mathbf K}_{1}(t, x_0, y_0).
	\end{align*} 
	
	Since $P_tf(x)=\int_{\Rd} p(t, x, y)f(y)dy$ with the continuous function $p(t,x,y)$(see, \autoref{p:on_upper}), we obtain the following estimate for $t>0$ 
	and $x_0, y_0$ satisfying the condition $\mathcal{R}(1)$,
	\begin{align*}
	&p(t, x_0,y_0) \le c [\phi^{-1}(t)]^{-d}{\mathbf K}_{1}(t, x_0, y_0)\nn\\
	\asymp 
	&\ [\phi^{-1}(t)]^{-d}\prod_{j=2}^d\Big(\frac{t\phi^{-1}(t)}{|x_0^{j}-y_0^{j}|\phi(|x_0^{j}-y_0^{j}|)}\wedge 1\Big)
	\begin{cases}
	\Big(\frac{t\phi^{-1}(t)}{|x_0^{1}-y_0^{1}|\phi(|x_0^{1}-y_0^{1}|)}\wedge 1\Big)^{\newb+q}
	&\mbox{if } q\in \big[0, \frac{1}{1+\newa}\big),\\
	\Big(\frac{t\phi^{-1}(t)}{|x_0^{1}-y_0^{1}|\phi(|x_0^{1}-y_0^{1}|)}\wedge 1\Big)
	&\mbox{if } q\in \big(\frac{1}{1+\newa}, 1\big).
	\end{cases}
	\end{align*}
	This proves  \autoref{lem:GF-dminus1} with $\theta_{d-1}=\frac{\la}{\la+1}$.
\end{proof}

\medskip
\begin{proof}[Proof of \autoref{lem:Fl}]
	Given that $\HHq{q}{l}$ holds for some $l\in \{0, 1, \ldots, d-2\}$   and  $q\in\big(\frac{1}{1+\newa}, 1\big)$. For any $t>0$, consider $x_0, y_0$ satisfying the condition $\mathcal{R}(i_0)$ for some $i_0\in\{1,\ldots, d-l\}$. 
	Recall that $\kappa=\phi^{-1}(t)$ and $R_i=2^{\lengthR_i}\kappa$. Applying \autoref{prop:main} with $j_0=d-l$, for $x\in B(x_0,\frac{\kappa}{8})$ and $\tau=\tau_{B(x_0,{R_{d-l}}/{8})}$, we obtain 
	\begin{align*}
	\E^{x} &\left[\1_{\{\tau\le t/2\}}P_{t-\tau}f(X_{\tau})\right]\le \, C_{\ref{prop:main}}\,[\phi^{-1}(t)]^{-d}\|f\|_1 \prod_{j=d-l}^d\Big(\frac{t\phi^{-1}(t)}{|x_0^{j}-y_0^{j}|\phi(|x_0^{j}-y_0^{j}|)}\Big).
	\end{align*}
	By the similar argument used in the proof of \autoref{lem:GF-dminus1}, we obtain for $t>0$ and for a.e. $(x, y)\in B(x_0, \frac{\kappa}{8})\times B(y_0, \frac{\kappa}{8})$, 
	\begin{align*}
	p(t, x, y)\le c [\phi^{-1}(t)]^{-d}
	\prod_{j=d-l}^d\Big(\frac{t\phi^{-1}(t)}{|x_0^{j}-y_0^{j}|\phi(|x_0^{j}-y_0^{j}|)}\Big)\,,
	\end{align*}
	therefore, for $t>0$ and $x_0, y_0$ satisfying the condition $\mathcal{R}(i_0)$,
	\begin{align*}
	%\label{eq:con2}
	p(t ,x_0, y_0)&\le c [\phi^{-1}(t)]^{-d}\prod_{j=d-l}^{d}\Big(\frac{t\phi^{-1}(t)}{|x_0^{j}-y_0^{j}|\phi(|x_0^{j}-y_0^{j}|)}\wedge 1\Big).
	\end{align*}
	This implies $\HHq{0}{l+1}$ for $l\in \{0, 1, \ldots, d-2\}$  by \autoref{lem:reduction_cases} and hence we have proved \autoref{lem:Fl}.
\end{proof}

\medskip

The proof of \autoref{lem:Gl} is more complicated.
We introduce
\begin{align}\label{eq:newmm}
\mathfrak N(\delta):=\frac{\phi(\kappa)}{2^{\delta}\, \phi(2^{\delta}\kappa)}\qquad\text{for $\delta\in \Z$}.
\end{align}
Then {\bf(WS)} implies that there exists $c=c(\phi)\ge 1$ such that for any $\delta\in \Z_+$,
\begin{align}\label{eq:newmm2}
c^{-1}2^{-\delta(\ua+1)}\le \fN(\delta)\le c2^{-\delta(\la+1)}.
\end{align}

\medskip

\begin{proof}[Proof of \autoref{lem:Gl}]
	Given that $\HHq{q}{l}$ holds for some $l\in \{0, 1, \ldots, d-2\}$  and $q\in\big[0, \frac{1}{1+\newa}\big)$.
	For any $t>0$, consider $x_0, y_0$ satisfying the condition $\mathcal{R}(i_0)$ for some $i_0\in\{1,\ldots, d-l\}$. 
	Recall that $\kappa=\phi^{-1}(t)$ and $R_j=2^{\lengthR_j}\kappa\asymp |x_0^j-y_0^j|$ for $j\in\{i_0, \ldots, d\}$.
	Applying \autoref{prop:main} with  $j_0\in \{i_0, \ldots, d-l\}$, we obtain that for $x\in B(x_0, \frac{\kappa}{8})$ and $\tau=\tau_{B(x_0, R_{j_0}/8)}$,
	\begin{align*}
	\E^{x}\left[\1_{\{\tau\le t/2\}}P_{t-\tau}f(X_{\tau})\right]
	\le C_{\ref{prop:main}}[\phi^{-1}(t)]^{-d} \|f\|_1
	{\mathbf G}_{j_0}(l),
	\end{align*}
	where
	\begin{align}\label{eq:gm}
	{\mathbf G}_{j_0}(l):=\fN(\lengthR_{j_0})^{\newb+q}\prod_{j={j_0+1}}^{d-l}
	\fN(\lengthR_{j})^q\prod_{j=d-l+1}^d\fN(\lengthR_{j}).
	\end{align}
	By a similar argument  
	%By an argument, which is similar to one 
	used in \autoref{lem:GF-dminus1}, we obtain for $t>0$ and for a.e. $(x, y)\in B(x_0, \frac{\kappa}{8})\times B(y_0, \frac{\kappa}{8})$,
	\begin{align*}
	p(t, x, y)\le c [\phi^{-1}(t)]^{-d}
	{\mathbf G}_{j_0}(l)\qquad\mbox{ for } j_0\in \{i_0, \ldots, d-l\},
	\end{align*}
	hence, for $t>0$ and $x_0, y_0$ satisfying the condition $\mathcal{R}(i_0)$,
	\begin{align}\label{eq:upp11}
	p(t ,x_0, y_0)&\le c [\phi^{-1}(t)]^{-d}\big({\mathbf G}_{i_0}(l)\wedge {\mathbf G}_{i_0+1}(l)\wedge \ldots\wedge {\mathbf G}_{d-l}(l)\big).
	\end{align}
	
	Given $l\in \{0, 1, \ldots, d-2\}$, in order to obtain $\theta_l$ in \autoref{lem:Gl}, we first define $\theta_l^{i_0}$ inductively for $i_0\in\{1, \ldots, d-l\}$. 
	Let $t>0$, and suppose that $x_0, y_0$ satisfy $\mathcal{R}(d-l)$. Since \eqref{eq:upp11}  implies $p(t ,x_0, y_0)\le c [\phi^{-1}(t)]^{-d} {\mathbf G}_{d-l}(l)$, we have that	
	\begin{align}\label{eq:d-l}
	&	p(t ,x_0, y_0)\le c [\phi^{-1}(t)]^{-d} \fN(\lengthR_{d-l})^{\newb+q}\prod_{j=d-l+1}^d\fN(\lengthR_{j})\nn\\
	\asymp &\  [\phi^{-1}(t)]^{-d}\prod_{j={1}}^{d-l}
	\big(\tfrac{t\phi^{-1}(t)}{|x_0^{j}-y_0^{j}|\phi(|x_0^{j}-y_0^{j}|)}\wedge1\Big)^{\theta_l^{d-l}+q}\prod_{j=d-l+1}^d\big(\tfrac{t\phi^{-1}(t)}{|x_0^{j}-y_0^{j}|\phi(|x_0^{j}-y_0^{j}|)}\wedge 1\big),
	\end{align}
	with $\theta_l^{d-l}:=\newb$. 
	Now we consider the case $x_0, y_0$ satisfying $\mathcal{R}(i_0)$ for some  $i_0\in\{1, \ldots, d-l-1\}$.
	Let
	\begin{align*}
	\theta_l^{i_0}:=\frac{\la}{2+\la+\ua}\left(\sum_{i=1}^{d-l-i_0}\Big(\frac{\ua+1}{\la}\Big)^i\right)^{-1} %\le\frac{\la}{2(\ua+1)}
	.
	\end{align*}
 For $a, b\in \N_0:=\N\cup\{0\}$, define a notation $\TT{\cdot}{\cdot}$ as follows:
	$$\TT{a}{b}:=\Tj{a}{b} \quad\text{ if }a\le b;\quad\text{ otherwise, } \quad\TT{a}{b}:=0.$$ 
	For the convenience of the notation, let $b_0:=\newb$. We discuss in two cases:
	
	Case I. Suppose that
	\begin{align}\label{lam1}
	\theta_l^{i_0}\le \frac{\lengthR_{j_0}(\la+1)b_0-(\ua+1)\TT{i_0}{j_0-1}q}{(\ua+1)\TT{i_0}{d-l}}\qquad \text{
		for some  $j_0\in\{{i_0},\ldots, d-l-1\}$.}
	\end{align} 
	Since $b_0\ge \theta_l^{i_0}$, {\eqref{eq:newmm2} and \eqref{lam1} imply} that 
	\begin{align*}
	\left(\fN(\lengthR_{j_0})^{b_0+q} \prod_{j=j_0+1}^{d-l}\fN(\lengthR_{j})^{q}\right)^{-1}
	&\prod_{j={i_0}}^{d-l}\fN(\lengthR_{j})^{q+\theta_l^{i_0}}=\prod_{j={i_0}}^{j_0-1}\fN(\lengthR_{j})^{q+\theta_l^{i_0}}\fN(\lengthR_{j_0})^{-(b_0-\theta_l^{i_0})}
	\prod_{j=j_0+1}^{d-l}\fN(\lengthR_{j})^{\theta_l^{i_0}}\\
	&\ge c\,
	2^{-(\ua+1)\TT{i_0}{j_0-1}({q+\theta_l^{i_0}})}
	2^{(\la+1)\lengthR_{j_0}(b_0-\theta_l^{i_0})} 2^{-(\ua+1) \TT{j_0+1}{d-l}\theta_l^{i_0}} \nonumber\\ 
	&\ge c \ 2^{-(\ua+1) \TT{i_0}{j_0-1}q} 2^{-(\ua+1)\TT{i_0}{d-l}\theta_l^{i_0}}
	2^{\lengthR_{j_0}(\la+1)b_0} \ge c,
	\end{align*}
	and therefore,
	\begin{align}
	\label{eq:psa1n}
	\fN(\lengthR_{j_0})^{b_0+q} \prod_{j=j_0+1}^{d-l}\fN(\lengthR_{j})^{q}
	\le c\
	\prod_{j={i_0}}^{d-l}\fN(\lengthR_{j})^{q+\theta_l^{i_0}}.
	\end{align}
	Case II.	Suppose that
	\begin{align*}
	\theta_l^{i_0}>\max_{j_0\in\{{i_0}, \ldots, d-l-1\}}
	\frac{\lengthR_{j_0}(\la+1)b_0-(\ua+1)\TT{i_0}{j_0-1}q}{(\ua+1)\TT{i_0}{d-l}}.
	\end{align*}
	Let $K_0:=\newA$.
	For $i\in\{0, \ldots, d-l-1-i_0\}$, the lower bounds of $\big(1+\tfrac{q}{K_0 b_0}\big)^i\theta_l^{i_0}$ could be obtained as follows;  
	\begin{alignat*}{3}
	j_0=i_0:&\quad 
	&&&\theta_l^{i_0} &> \tfrac{K_0 b_0\cdot\lengthR_{i_0}}{\TT{i_0}{d-l}}\,,\\
	j_0=i_0+1:&\quad	
	&\theta_l^{i_0}& > \tfrac{K_0 b_0\cdot\lengthR_{i_0+1}}{\TT{i_0}{d-l}}-\tfrac{q\cdot \lengthR_{i_0}}{\TT{i_0}{d-l}}
	\,\qquad\Longrightarrow &\big(1+\tfrac{q}{K_0b_0}\big)\theta_l^{i_0} &> \tfrac{K_0 b_0\cdot\lengthR_{i_0+1}}{\TT{i_0}{d-l}}\,,\\
	j_0=i_0+2:&\quad
	&	\theta_l^{i_0} &> \tfrac{K_0 b_0\cdot\lengthR_{i_0+2}}{\TT{i_0}{d-l}}-\tfrac{q\cdot(\lengthR_{i_0}+\lengthR_{i_0+1})}{\TT{i_0}{d-l}}
	\quad\Longrightarrow &\big(1+\tfrac{q}{K_0 b_0}\big)^2\theta_l^{i_0}&> \tfrac{K_0 b_0\cdot\lengthR_{i_0+2}}{\TT{i_0}{d-l}}\,,\\
	\vdots \qquad&&&\qquad\vdots &\vdots \qquad\\
	j_0=d-l-1:&\quad
	&\theta_l^{i_0}& > \tfrac{K_0 b_0\cdot\lengthR_{d-l-1}}{\TT{i_0}{d-l}}-\tfrac{q\cdot\TT{i_0}{d-l-2}}{\TT{i_0}{d-l}}
	\,\Longrightarrow &\big(1+\tfrac{q}{K_0 b_0}\big)^{d-l-1-i_0}\theta_l^{i_0} &> \tfrac{K_0 b_0\cdot\lengthR_{d-l-1}}{\TT{i_0}{d-l}}\, .
	\end{alignat*}
	Hence,
	\begin{align}\label{eq:lamsum}
	\theta_l^{i_0}\sum_{i=0}^{d-l-i_0-1}\big(1+\frac{q}{K_0 b_0}\big)^i> \frac{K_0 b_0\cdot\TT{i_0}{d-l-1}}{\TT{i_0}{d-l}}. 
	\end{align}	
	Now our claim is that
	\begin{align}\label{eq:psad1n} 
	\fN(\lengthR_{d-l})^{b_0+q}\le c\  \prod_{j={i_0}}^{d-l}\fN(\lengthR_{j})^{q+\theta_l^{i_0}}.
	\end{align}	
	For this, we only need to prove that 
	\begin{align}\label{eq:psad2n} -(\ua+1)\TT{{i_0}}{d-l-1}(q+\theta_l^{i_0})+(\la+1)\lengthR_{d-l}(b_0-\theta_l^{i_0})\ge 0.
	\end{align}
	If so,	using \eqref{eq:newmm2} with $b_0\ge \theta_l^{i_0}$, we have that
	\begin{align}\label{eq:psad2}
	\fN(\lengthR_{d-l})^{-b_0-q}\prod_{j={i_0}}^{d-l}\fN(\lengthR_{j})^{q+\theta_l^{i_0}}
	&\ge c\ 2^{-(\ua+1)\TT{{i_0}}{d-l-1}(q+\theta_l^{i_0})} 2^{(\la+1)\lengthR_{d-l}(b_0-\theta_l^{i_0})}\ge c.
	\end{align}	
	By \eqref{eq:lamsum} and the fact that $K_0\le  1$, we observe that
	\begin{align}\label{eq:psad3}
	& \frac{\TT{{i_0}}{d-l-1}}{\TT{{i_0}}{d-l}}(\ua+1)(q+\theta_l^{i_0})-\left(1-\frac{\TT{{i_0}}{d-l-1}}{\TT{{i_0}}{d-l}}\right) (\la+1)(b_0-\theta_l^{i_0})
	\nn\\
	\le&\ \frac{\TT{{i_0}}{d-l-1}}{\TT{{i_0}}{d-l}}(\ua+1)(b_0+q)-(\la+1)(b_0-\theta_l^{i_0})
	\nn\\
	< &\ {\theta_l^{i_0}}\sum_{i=0}^{d-l-i_0-1}\big(1+\frac{q}{K_0 b_0}\big)^i(\ua+1)\left(\frac{1}{K_0}+\frac{q}{K_0b_0}\right)- (\la+1)(b_0-\theta_l^{i_0})\nn\\
	\le& \  {\theta_l^{i_0}}\sum_{i=1}^{d-l-i_0}\Big(\frac{1}{K_0}+\frac{q}{K_0 b_0}\big)^i(\ua+1)- (\la+1)(b_0-\theta_l^{i_0}).
	\end{align}
	The definitions of $b_0$ and $\theta_l^{i_0}$ yield that  $b_0-\theta_l^{i_0}\ge \frac{\ua+1}{2+\la+\ua}b_0$ and hence,
	\begin{align}\label{eq:psad4}
	{\theta_l^{i_0}}\sum_{i=1}^{d-l-i_0}\Big(\frac{\ua+1}{\la}\Big)^i(\ua+1)- (\la+1)(b_0-\theta_l^{i_0})\le 0.
	\end{align}
	Since $\frac{1}{K_0}+\frac{q}{K_0 b_0}\le \frac{\ua+1}{\la}\  \text{ for  } q\le (1+\la)^{-1},$
	\eqref{eq:psad3}--\eqref{eq:psad4} yield \eqref{eq:psad2n}.
	
	By \eqref{eq:psa1n} and \eqref{eq:psad1n}, we obtain the upper bounds of ${\mathbf G}_{j_0}(l)$ for $j_0\in\{i_0, \ldots, d-l\}$ in \eqref{eq:gm}, and therefore, \eqref{eq:upp11} implies that
	for any $t>0$ and $x_0, y_0$ satisfying $\mathcal{R}(i_0)$, $i_0\in\{1, \ldots, d-l-1\}$, 
	\begin{align}\label{eq:others}
	&p(t, x_0,y_0)\le\,c [\phi^{-1}(t)]^{-d}\prod_{j=i_0}^{d-l}\fN(\lengthR_j)^{q+\theta_l^{i_0}}\prod_{i=d-l+1}^{d}\fN(\lengthR_j)\nn\\
	&\asymp [\phi^{-1}(t)]^{-d}\prod_{i=1}^{d-l}\big(\tfrac{t\phi^{-1}(t)}{|x_0^{j}-y_0^{j}|\phi(|x_0^{j}-y_0^{j}|)}\wedge 1\Big)^{q+\theta_l^{i_0}}\prod_{i=d-l+1}^{d}\big(\tfrac{t\phi^{-1}(t)}{|x_0^{j}-y_0^{j}|\phi(|x_0^{j}-y_0^{j}|)}\wedge 1\big),
	\end{align}
	where 
	$\theta_l^{i_0}:=\frac{\la}{2+\la+\ua}\left(\sum_{i=1}^{d-l-i_0}\Big(\frac{\ua+1}{\la}\Big)^i\right)^{-1} .$
	
	By \eqref{eq:d-l} and \eqref{eq:others} in connection with \autoref{lem:reduction_cases}, 
	we have $\HHq{q+\theta_l}{l}$ with $\theta_l:=\min\limits_{i_0\in \{1,2,\ldots,d-l\}}\theta_l^{i_0}=\frac{\la}{2+\la+\ua}\left(\sum_{i=1}^{d-l-1}\Big(\frac{\ua+1}{\la}\Big)^i\right)^{-1} $ for $l\in \{0, 1, \ldots, d-2\}$. Hence, the proof of \autoref{lem:Gl} is complete.
\end{proof}

\section{H\"older continuity and  lower bound estimates}
\label{sec:low}
In this section, we  prove the  H\"older continuity of  heat kernel $p(t,x,y)$.
Consequently, the process $X$ can be refined to start from every point in $\R^d$, and  estimates for the transition density functions hold for every $x,y\in\R^d$.  As indicated   in \cite[Remark 3.6]{BKK10}, the H\"older continuity of  transition densities can be derived from the boundedness of transition densities plus the H\"older continuity of bounded harmonic functions. According to which we also conclude that $X$ is a strong Feller process. At the end, we give the sharp lower bound for $p(t,x,y)$.
\medskip

\subsection{H\"older continuity}

\begin{proposition}\label{p:exit}
	For any $A>0$ and $B\in (0, 1)$, there exists ${ \varrho}={\varrho}(A, B, \phi)\in (0, 2^{-1})$ such that for every $r>0$ and $x\in \Rd$,
	\begin{align*}
	\Pp^x(\tau_{B(x, Ar)}<{\varrho} \phi(r))=\Pp^x\Big(\sup_{s\le {\varrho}\phi(r) }|X_s-X_0|>Ar \Big)\le B.
	\end{align*}
\end{proposition}
\begin{proof}
	Let $\zeta$ be the lifetime of $X$.
	By  \autoref{thm:cons}, $\Pp^x(\zeta\le u)=0$ for any $u>0$.
	So for any $x\in \Rd$ and $t,u>0$,  by the upper bound estimates in \autoref{thm:uphke} of $p(t, x, y)$, and {\bf(WS)} condition, we have that
	\begin{align*}
	&\Pp^x(|X_t-x|\ge u)=\int_{B(x, u)^c} p(t, x, y) dy +\Pp^x(\zeta\le u)\nn\\
	&\le  c_1 \prod_{i=1}^d\int_{\left\{y^i:|x^i-y^i|\ge \frac{u}{\sqrt{d}}\right\}} \frac{t}{|x^i-y^i|\phi(|x^i-y^i|)}dy^i\le  c_1 t^d\prod_{i=1}^d\int_{\frac{u}{\sqrt{d}}}^\infty (w\phi(w))^{-1}dw\nn\\
	&\le c_2\left(\frac{t}{\phi(u)}\right)^{d}\prod_{i=1}^du^{\la}\int_{\frac{u}{\sqrt{d}}}^\infty w^{-1-\la}dw \le  c_3\left(\frac{t}{\phi(u)}\right)^{d}.\nn
	\end{align*}
	By the similar proof as in \autoref{p:exit_new}, for $\sigma_u:=\inf\{s\ge 0:|X_s-X_0|>u\}$, we have
	%By the strong Markov property, 
	\begin{align*}
	\Pp^x(\sigma_u<t, |X_{2t}-x|\le u/2)
	\le \sup_{\substack{l<t\ \&\\ y\in\{z\in\Rd:|x-z|>u\} }} \Pp^y(|X_{2t-l}-y|>u/2),
	\end{align*}
	and for any $x\in \Rd$ and $u, t>0$,
	\begin{align}\label{eq:ex2}
	&\Pp^x\Big(\sup_{s\le t}|X_s-X_0|>u\Big)
	\le \Pp^x(\sigma_u<t, |X_{2t}-x|\le u/2)+\Pp^x(|X_{2t}-x|\ge u/2)\nn\\   
	&\le \sup_{\substack{y\in\{z\in\Rd:|x-z|>u\}\\ l<t }} \Pp^y(|X_{2t-l}-y|>u/2)+c_4\left(\frac{t}{\phi(u)}\right)^{d}\le c_5\left(\frac{t}{\phi(u)}\right)^{d}.
	\end{align}
	For any $A, B\in (0, 1)$, \eqref{eq:ex2} and {\bf(WS)} imply that 
	\begin{align}\label{eq:ex3}
	\Pp^x\Big(\sup_{s\le { \varrho}_1\phi(r)}|X_s-X_0|>Ar\Big)
	\le c_5\left(\frac{{ \varrho}_1\phi(r)}{\phi(Ar)}\right)^{d}\le c_5\left({\varrho}_1\uC A^{-\ua}\right)^d \le B,
	\end{align}
	for some constant ${ \varrho}_1:=\uC^{-1} A^{\ua}(c_5^{-1}B)^{1/d} \wedge 2^{-1}$, and it proves our assertion in this case.
	Now we consider the case $A\ge 1$ and $B\in(0, 1)$.
	In the similar way to obtain \eqref{eq:ex3}, using \eqref{eq:ex2},
	\begin{align*}
	\Pp^x\Big(\sup_{s\le {\varrho}_2\phi(r) }|X_s-X_0|>Ar\Big)\leq\Pp^x\Big(\sup_{s\le {\varrho}_2\phi(Ar) }|X_s-X_0|>Ar\Big)\le c_5\left(\frac{{ \varrho}_{2}\phi(Ar)}{\phi(Ar)}\right)^d\le B,
	\end{align*}
	for some ${\varrho}_2:=(c_5^{-1}B)^{1/d}\wedge 2^{-1}$, and therefore we have our assertion.
\end{proof}

\begin{proposition}
	\label{prop:exptau}
	For $r>0$, there exist $ a_i=a_i(\phi,  \Lambda)>0, i=1,2,3$ such that
	$$a_1\phi(r)\leq\E^x[\tau_{B(x,r)}]\leq a_2\phi(r)\qquad\text{ and }\qquad\E^x[\tau^2_{B(x,r)}]\leq a_3\phi(r)^2.$$
\end{proposition}
\begin{proof}
	We write $\tau=\tau_{B(x,r)}$ for simplicity. By \autoref{p:exit} with $A=1,B=\frac{1}{2}$, there exists ${ \varrho}={\varrho}(1,\frac{1}{2},\phi)$ such that
	\begin{align*}
	\E^x[\tau]\geq&{\varrho}\phi(r)\Pp^x(\tau\geq{\varrho}\phi(r))={ \varrho}\phi(r)[1-\Pp^x(\tau<{ \varrho}\phi(r))]\geq\frac{{\varrho}}{2}\phi(r).
	\end{align*}
	On the other hand,  the L\'evy system  in \eqref{eq:LSd} and \textbf{(WS)} imply that  for any $t>0$,
	\begin{align*}
	\Pp^x(\tau\leq t)\geq&\ \E^x\left[\sum_{s\leq t\wedge \tau}\1_{\{|X_s-X_{s-}|>2r\}}\right]=\ \E^x\left[\int_0^{t\wedge\tau}\sum_{i=1}^d\int_{|h|>2r}J(X_s, X_s+e^ih)dhds\right]\\
	\geq&\ \E^x\left[\int_0^{t\wedge\tau}\sum_{i=1}^d\int_{|h|>2r}\frac{\Lambda^{-1}}{|h|\phi(|h|)}dhds\right]\\
	\geq&\ c_1[\phi(r)]^{-1}\E^x[t\wedge\tau]\geq c_1[\phi(r)]^{-1}t\Pp^x(\tau>t).
	\end{align*}
	Thus, $\Pp^x(\tau>t)\leq 1-c_1[\phi(r)]^{-1}t\Pp^x(\tau>t)$. 
	Choose $t=c_1^{-1}\phi(r)$ so that $\Pp^x(\tau>t)\leq 1/2$. Using the Markov property at time $mt$ for $m=1,2,\dots,$ 
	$$\Pp^x(\tau>(m+1)t)\leq\E^x[\Pp^{X_{mt}}(\tau>t);\tau>mt]\leq \frac{1}{2}\Pp^x(\tau>mt).$$
	By induction, we obtain that $\Pp^x(\tau>mt)\leq 2^{-m}$ with $t=c_1^{-1}\phi(r)$, and hence
	$$\E^x[\tau]\leq\sum_{m=0}^\infty t\Pp^x(\tau>mt)\leq\sum_{m=0}^\infty t2^{-m}\leq a_2\phi(r).$$
	Similarly, with $t=c_1^{-1}\phi(r)$,
	$$\E^x[\tau^2]\leq\sum_{m=0}^\infty 2(m+1)t^2\Pp^x(\tau>mt)\leq\sum_{m=0}^\infty (m+1)2^{-(m-1)}t^2\leq a_3\phi(r)^2.$$
\end{proof}
\medskip

In order to prove the bounded harmonic functions associated with $X$ are H\"older continuous, we need a support theorem stated in the following.
\begin{theorem}[Support theorem]
	\label{thm:support}
	Let $\psi:[0,t_1]\to\R^d$  be continuous with $\psi(0)=x$ and the image of $\psi$ be contained in $B(0,1)$. For any $\varepsilon>0$, there exists a constant $c=c(\psi,\varepsilon,t_1)$ such that
	$$\Pp^x\left(\sup_{s\leq t_1}|X_s-\psi(s)|\leq \varepsilon\right)>c.$$
\end{theorem}
\begin{proof}
	This follows from the proof of \cite[Lemma 4.7, Lemma 4.8, Theorem  4.9]{Xu13}	with the fact that
	\begin{align*}
	\frac{c_1}{\phi(\delta)}\le	\int_\Rd \mathfrak J_{\delta}(x, y)dy \le \frac{c_2}{\phi(\delta)}, \qquad x\in \Rd
	\end{align*}
	for some constant $c_1, c_2>0$.
	Here $\mathfrak J_{\delta}(x,y)$ is the function defined in \eqref{eq:cJ} with $\delta$ instead of $1$.
\end{proof}
\medskip

For any Borel set $A\subset\R^d$, {we denote by $T_A:=\inf\{t>0:X_t\in A\}$  the hitting time of $X$ from $A$.}

\begin{corollary}
	\label{cor:support}
	For $r\in(0,1)$, let $x\in Q(0,1)$ with dist$(x,\partial Q(0,1))>r$. If $Q(z,r)\subset Q(0,1)$, then 
	$$\Pp^x(T_{Q(z,r)}<\tau_{Q(0,1)})\geq  c$$
	for some positive constant $c=c(r)>0$.
\end{corollary}
\begin{proof}
	Let $\psi$ be the line segment  contained in $B(0, \sqrt{d})$  from $x$ to $z$ with $\psi(0)=x$ and $\psi(1)=z$.
	Since dist$(x,\partial Q(0,1))>r$ and $Q(z,r)\subset Q(0,1)$,  by \autoref{thm:support}  with some modification, we have that
	$$\Pp^x(T_{Q(z,r)}<\tau_{Q(0,1)})\geq\Pp^x\left(\sup_{s\le1}|X_s-\psi(s)|\leq\frac{r}{2}\right)>c(r).$$
\end{proof}

\begin{proposition}
	\label{prop:Ttao}
	There exists a nondecreasing function $\varphi:(0,1)\to(0,1)$ such that {for any open set} $D\subset Q(0,1)$ with $|D|>0$,  and for any $x\in Q(0,\frac{1}{2})$, 
	$$\Pp^x(T_D<\tau_{Q(0,1)})\geq\varphi(|D|).$$
\end{proposition}
\begin{proof}
	{Let $D\subset Q(0,1)$ be  an open set such that $|D|>0$ and  $x\in Q(0,\frac{1}{2})$.}
	We first show that 
	\begin{align}\label{eq:(1)}
	\Pp^x(T_D<\tau_{Q(0,1)})\geq c\ \text{ when }|Q(0,1)\setminus D|\ \text{ is small enough.}
	\end{align}
	For simplicity, we write $\tau$ for $\tau_{Q(0,1)}$. By \autoref{prop:exptau}, we observe that 
	\begin{align}
	\label{eqn:Ttau}
	 c_1\leq\E^x[\tau]=&\ \E^x[\tau;T_D<\tau]+\E^x\left[\int_0^\tau \1_{D^c}(X_s)ds\right]\nn\\
	\le& \ (\E^x[\tau^2])^{1/2}(\Pp^x(T_D<\tau))^{1/2}+\E^x\left[\int_0^\tau \1_{D^c}(X_s)ds\right]\nn\\
	\leq&\  c_2^{1/2}(\Pp^x(T_D<\tau))^{1/2}+\E^x\left[\int_0^\tau \1_{D^c}(X_s)ds\right].
	\end{align}
	By \autoref{prop:exptau} again, since	for $t_0>0$,
	$$\E^x[\tau-\tau\wedge t_0]\leq \E^x[\tau;\tau>t_0] \leq\left(\frac{\E^x[\tau^2]\E^x[\tau]}{t_0}\right)^{1/2}\le \frac{c_3}{\sqrt t_0},$$
	for some $c_3>0$,
	we can choose $t_0$ large enough such that $\E^x[\tau-\tau\wedge t_0]\leq \frac{ c_1}{4} \le t_0$. Then
	\begin{align*}
	\E^x\left[\int_0^\tau \1_{D^c}(X_s)ds\right]=&\E^x\left[\int_0^\tau \1_{Q(0,1)\setminus D}(X_s)ds\right]\leq \frac{ c_1}{4}+\E^x\left[\int_0^{t_0}\1_{Q(0,1)\setminus D}(X_s)ds\right]\\
	=&\ \frac{ c_1}{4}+\int_0^{t_0}\Pp^x(X_s\in Q(0,1)\setminus D)ds\\
	=&\ \frac{ c_1}{4}+\int_0^{\frac{ c_1}{4}}\Pp^x(X_s\in Q(0,1)\setminus D)ds+\int_{\frac{ c_1}{4}}^{t_0}\int_{Q(0,1)\setminus D}p(s,x,y)dyds\\
	\leq&\ \frac{ c_1}{2}+|Q(0,1)\setminus D|\int_{\frac{c_1}{4}}^{t_0}{ c_4}[\phi^{-1}(s)]^{-d}ds.
	\end{align*}
	The last inequality comes from \autoref{p:on_upper}.
	Since $\int_{\frac{ c_1}{4}}^{t_0}c[\phi^{-1}(s)]^{-d}ds$ is bounded by a constant depending on $t_0$ and $ c_1$, letting $|Q(0,1)\setminus D|$ be small enough, we have that $\E^x\left[\int_0^\tau \1_{D^c}(X_s)ds\right]\leq{3 c_1}/{4}$. Therefore \eqref{eqn:Ttau} yields
	$\Pp^x(T_D<\tau)\geq  {c_1^2}/{16 c_2}$
	and we have proved \eqref{eq:(1)}.  
	
	Let 
	$\varphi(\varepsilon):=\inf \big\{\Pp^y(T_D<\tau_{Q(0, 1)}):y\in Q(0, \tfrac{1}{2}),\ D\subset Q(0,1),\ |D|\ge \varepsilon |Q(0,1)| \big\}.$
	By \eqref{eq:(1)}, $\varphi(\varepsilon)>0$ for $\varepsilon$ sufficiently close to $1$. By \cite[Proposition V.7.2]{bass1998diffusions} together with \autoref{cor:support}, 
	one can follow the proof in \cite[Theorem V.7.4]{bass1998diffusions}   to show that 
	$\varepsilon_0=\inf\{\varepsilon:\varphi(\varepsilon)>0\}=0$.
	Therefore, we conclude our assertion with $\varphi$.
\end{proof}
\medskip

A function $h$ is called harmonic with respect to $X$ in a domain $D\subset \Rd$ if $h(X_{t\wedge\tau_D})$ is a martingale with respect to $\Pp^x$  for every $x$ in $D$.

\begin{theorem}
	\label{harmonicthm}
	For any $x_0\in \Rd$ and $r\in(0,1)$, suppose that $h$ is harmonic in $B(x_0,r)$ with respect to $X$ and bounded in $\R^d$. Then there exist constants $c,\beta>0$ such that
	\begin{equation*}
	%	\label{hHolder}
	|h(x)-h(y)|\leq c\left(\frac{|x-y|}{r}\right)^\beta\|h\|_\infty,\qquad 
	\text{ for }x,y\in B(x_0,r/2).
	\end{equation*}
\end{theorem}
\begin{proof}
	The proof is similar to that of \cite[Theorem V.7.5]{bass1998diffusions}. For the reader's convenience, we give the modified proof here. Define $\text{Osc}_Dh:=\sup_{x\in D}h(x)-\inf_{x\in D}h(x)$. To prove the result, it suffices to show there exists $\rho<1$ such that for all $z\in B(x_0,r/2)$ and $s\leq r/4$, 
	\begin{equation}
	\label{eqn:Osc}
	\underset{B(z,s/2)}{\text{Osc}}h\leq \rho\underset{B(z,s)}{\text{Osc}}h.
	\end{equation}
	Without loss of generality, we can assume $\inf_{B(z,s)}h=0$ and $\sup_{B(z,s)}h=1$, 
	otherwise we can do a linear transform of $h$. 
	Let $D=\{x\in B(z,s/2):h(x)\geq1/2\}$. We can assume $|D|\geq \frac{1}{2}|B(z,s/2)|$, otherwise we replace $h$ by $1-h$.  Since $h$ is harmonic,
	\begin{align*}
	h(x)=&\E^x[h(X_{ \tau_{B(z,s)}\wedge T_D})]\geq \frac{1}{2}\ \Pp^x(T_D<\tau_{B(z,s)}) \qquad \text{ for } x\in B(z, s),
	%\geq&\frac{1}{2}\ \varphi(|D|/|B(z,s/2)|).
	\end{align*}
	and by \autoref{prop:Ttao} with the scaling, $D^s:=s^{-1}D$, we have
	$$\Pp^x(T_D<\tau_{B(z,s)})=\Pp^{ s^{-1}x}(T_{D^s}<\tau_{B(s^{-1}z,1)})\geq \varphi(|D^s|)\ge \varphi(2^{-1}|B(z, 1/2)|)\ge \varphi(2^{-(d+1)}).$$
	 for  $x\in B(z, s)$.
	Taking $\rho=1-\varphi(2^{-(d+1)})/2$ proves \eqref{eqn:Osc}.
\end{proof}
\medskip 

We now show the H\"older continuity of $\lambda$-resolvent ${U_\lambda}$ for the process $X$, that is,
$${ U_\lambda} f(x)=\E^x\left[\int_0^\infty e^{-\lambda t}f(X_t)dt\right]=\int_0^{\infty} e^{-\lambda t}P_t f(x)dt,$$
{ where $P_t$ is the corresponding semigoup operator  defined in \autoref{subsec:Nash}.}
\begin{proposition}
	\label{prop:UHolder}
	If $f$ is bounded, there exists $c=c(\lambda)$ and $ \varsigma>0$ such that $$|{ U_\lambda} f(x)-{ U_\lambda} f(y)|\leq c|x-y|^{ \varsigma}\|f\|_{\infty}, \qquad\text{ for }|x-y|<1.$$
\end{proposition}
\begin{proof}
	For any $x, y\in \Rd$ with $|x-y|<1$, consider $x_0\in\R^d$ and $ r:=|x-y|^{1/2}\in (0,1)$ so that $x,y\in B(x_0,r/2)$.
	Write $\tau_x=\tau_{B(x,r)}$ for simplicity. 
	By the strong Markov property, we have that
	$$
	{ U_\lambda} f(x)=\E^x \left[\int_0^{\tau_x}e^{-\lambda t}f(X_t)dt \right]+\E^x\left[(e^{-\lambda\tau_x}-1){U_\lambda} f(X_{\tau_x})\right]+\E^x\left[{ U_\lambda} f(X_{\tau{_x}})\right].$$
	Using the similar expression where $x$ is replaced by $y$,
	\autoref{prop:exptau} and the mean value theorem imply that
	\begin{align*}
	|{U_\lambda} f(x)-{ U_\lambda} f(y)|\leq &\left(\E^x[\tau_x]+\E^y[\tau_y]\right)\left(\|f\|_\infty+\lambda\|{ U_\lambda} f\|_\infty\right)+\left|\E^x[{U_\lambda} f(X_{\tau_x} )]-\E^y[{ U_\lambda} f(X_{\tau_y})]\right|\\
	\leq&\ { 2 a_2}\phi(r)\left(\|f\|_\infty+\lambda\|{U_\lambda} f\|_\infty\right)+\left|\E^x[{U_\lambda} f(X_{\tau_x})]-\E^y[{ U_\lambda} f(X_{\tau_y})]\right|.
	\end{align*}
	Since $z\to\E^z \left[{U_\lambda} f(X_{\tau_z})\right]$ is bounded and harmonic in $B(x_0,r)$, applying \autoref{harmonicthm}, we have that 
	$$\left|\E^x[{ U_\lambda} f(X_{\tau_x})]-\E^y[{U_\lambda} f(X_{\tau_y})]\right|
	\leq c_1\left(\frac{|x-y|}{r}\right)^\beta\|{ U_\lambda} f\|_\infty$$
	for some $c_1, \beta>0$.
	Since $r=|x-y|^{1/2}< 1$,  {\bf(WS)} yields $\phi(r)\le \lC^{-1}\phi(1) |x-y|^{\la/2}$.
	Therefore, using the fact that $\|{U_\lambda} f\|_\infty\leq \frac{1}{\lambda}\|f\|_{\infty}$, we have that
	\begin{align*}
	|{U_\lambda} f(x)-{ U_\lambda} f(y)| \leq&\ c_2\left(\phi(r)+\left(\frac{|x-y|}{r}\right)^\beta\right)\|f\|_\infty\\
	\leq & \ c_3(|x-y|^{\la/2}+|x-y|^{\beta/2})\|f\|_\infty.
	\end{align*}
\end{proof}

\begin{remark}\label{rem:sFP}
	{By \autoref{prop:UHolder} with \cite[(2.16) on page 77]{BlGe68}, we conclude that $X$ is a strong Feller process.}
\end{remark}

According to the spectral theorem, there exist projection operators $E_\mu$ on  $L^2(\R^d)$ such that
$$f=\int_0^\infty dE_\mu (f),\ P_tf=\int_0^\infty e^{-\mu t}dE_\mu(f)\quad
\text{and}\quad {U_\lambda} f=\int_0^\infty \frac{1}{\lambda+\mu}dE_\mu(f).$$
\begin{theorem}
	\label{thm:Holder}
	For $f\in L^2(\R^d)$, $P_tf$ is equal a.e. to a function that is H\"older continuous. Hence, we can refine $p(t,x,y)$ to be jointly continuous for any $t>0$ and $x,y\in\R^d$.
\end{theorem}
\begin{proof}
	Let $\lambda, \mu$ and $t>0$. For any $f\in L^2(\R^d)$,  define 
	\begin{align}\label{eq:h}
	h:=h(f):=\int_0^\infty(\lambda+\mu)e^{-\mu t}d E_\mu(f).
	\end{align}
	Then $h\in L^2(\R^d)$, since $\sup_\mu(\lambda+\mu)^2e^{-2\mu t}\leq c_1$ gives that
	\begin{align*}
	\|h\|_2^2=&\int_0^\infty (\lambda+\mu)^2e^{-2\mu t}d\langle E_\mu(f),E_\mu(f)\rangle\\
	\leq &\ c_1\int_0^\infty d\langle E_\mu(f),E_\mu(f)\rangle=c_1\|f\|_2^2<\infty.
	\end{align*}
	For any $g\in L^2(\R^d)$, note that $\|P_tg\|_1\leq\|g\|_1$ and $\|P_tg\|_\infty\leq c_2\|g\|_1$ since $p(t,x,y)$ is bounded. Then it follows that $\|P_tg\|_2\leq c_3\|g\|_1$. By Cauchy-Schwartz inequality, 
	\begin{align*}
	\langle h,g\rangle =&\int_0^\infty(\lambda+\mu)e^{-\mu t}d\langle E_\mu(f),E_\mu(g)\rangle\\
	\leq&\left(\int_0^\infty(\lambda+\mu)e^{-\mu t}d\langle E_\mu(f),E_\mu(f)\rangle\right)^{1/2}\left(\int_0^\infty(\lambda+\mu)e^{-\mu t}d\langle E_\mu(g),E_\mu(g)\rangle\right)^{1/2}\\
	\leq&\ c_4\left(\int_0^\infty d\langle E_\mu(f),E_\mu(f)\rangle\right)^{1/2}\left(\int_0^\infty e^{-\mu t/2}d\langle E_\mu(g),E_\mu(g)\rangle\right)^{1/2}\\
	=&\ c_4\|f\|_2\|P_{ t/4}g\|_2\leq c_5\|f\|_2\|g\|_1.
	\end{align*}
	Thus, we have $\|h\|_\infty\leq c_5\|f\|_2< c_*$ by taking the supremum for $g\in\{u\in L^1(\R^d):\|u\|_1\leq1\}$.
 Since
		\begin{align}\label{eq:res_sem}
		{U_\lambda} h=\int_0^\infty e^{-\mu t}dE_\mu (f)=P_tf\qquad \text{ a.e.}
		\end{align}
	we have our first assertion  by \autoref{prop:UHolder}.
	Fix $y$ and let $\widetilde f(z)=p(t/2,z,y)$. By \autoref{p:on_upper}, 
	since $\|\widetilde f\|_1=1$ and 
	$\|\widetilde f\|_\infty\leq \frac{c_6}{[\phi^{-1}(t)]^d}$, $$\|\widetilde f\|_2\leq\|\widetilde f\|_\infty\|\widetilde f\|_1\leq \frac{c_6}{[\phi^{-1}(t)]^d},\ 
	\text{ that is, }\  \widetilde f\in L^2(\Rd).$$
	Since \autoref{p:on_upper} again implies that
	$$p(t,x,y)=\int_{\R^d}p(t/2,x,z)p(t/2,z,y)dz=\int_{\R^d}p(t/2,x,z)\widetilde  f(z)dz=P_{t/2}\widetilde f(x),$$
	by	\autoref{prop:UHolder} and \eqref{eq:res_sem}, we have that 
	\begin{align*}
	|p(t,x,y)-p(t,z,y)|=&|P_{t/2}\widetilde  f(x)-P_{t/2}\widetilde  f(z)|=|{ U_\lambda} \widetilde h(x)-{ U_\lambda} \widetilde h(z)|\\
	\leq &\ c|x-z|^{\varsigma}\|\widetilde h\|_\infty\qquad\qquad \text{for $|x-z|<1$,}
	\end{align*}
	where $\widetilde h=\widetilde{h}(\widetilde{f})$ defined in \eqref{eq:h} satisfies $\|\widetilde h\|_{\infty}<c_*$.
	Hence, $p(t,x,y)$ is H\"older continuous with constants independent of $x, y$, and 
	the symmetry of $p(t, x, y)$ gives the joint continuity of $p(t, x, y)$.	
\end{proof}
\subsection{Lower bounds}
For the lower bound of $p(t, x, y)$, we first obtain the following on-diagonal estimate.
\begin{proposition}\label{p:onlower}
	There exists $c=c(\phi)\in (0, 1)$ such that
	\begin{align*}
	p(t, x, x)\ge \frac{c}{[\phi^{-1}(t)]^{d}} \qquad\text{ for }\ t>0, \  x\in \Rd.
	\end{align*}
\end{proposition}
\begin{proof}
	By  \autoref{p:exit}, there exists ${\varrho}_0\in (0, 2^{-1})$ such that for any $x\in \Rd$ and $r>0$
	\begin{align*}
	%\label{eq:on1}
	\Pp^x(\tau_{B(x,r)}\le {\varrho}_0 \phi(r))\le \tfrac{1}{2}.
	\end{align*}
	For any $t>0$, let $r_1:=\phi^{-1}\left(\tfrac{t}{2{\varrho}_0}\right)$ and $r_2:=\phi^{-1}(t)$.
	Because $2{\varrho}_0\le 1$ and $\phi^{-1}$ is increasing, we note that $r_2\le r_1$. Therefore, {\bf (WS)} implies
	\begin{align*}
	\lC \left(\frac{r_1}{r_2}\right)^{\la}\le \frac{\phi(r_1)}{\phi(r_2)}=2{\varrho}_0 \le \uC \left(\frac{r_1}{r_2}\right)^{\ua},
	\end{align*}
	so that  $r_2\le r_1\le  c_0r_2$
	for some $c_0:=\left({2{ \varrho}_0}/{\lC}\right)^{1/\la}$.
	Since 
	\begin{align*}
	\int_{\Rd\backslash B(x, c_0\phi^{-1}(t))} p(t/2, x, y)dy
	\le 	\Pp^x(\tau_{B(x,c_0r_2)}\le \tfrac{t}{2})
	\le 	\Pp^x(\tau_{B(x,r_1)}\le{\varrho}_0 \phi(r_1))\le \tfrac{1}{2},
	\end{align*}
	we have that  
	\begin{align*}
	p(t, x, x)=&\int_\Rd p^2(t/2, x, y) dy\ge  \int_{B(x, c_0\phi^{-1}(t))} p^2(t/2, x, y)dy\\
	\ge & |B(x, c_0\phi^{-1}(t))|^{-1}\left(\int_{B(x, B(x, c_0\phi^{-1}(t)))}p(t/2, x, y)dy\right)^2\ge \frac{c_1}{[\phi^{-1}(t)]^{d}}.
	\end{align*}
\end{proof}

\medskip

Using the scaling method in \autoref{subsec:upall}, we then obtain near diagonal lower estimate. 
\begin{proposition}\label{p:lower_on}
	There exist positive constants $c_1, c_2$ such that 
	\begin{align*}
	p(t, x, y)\ge \frac{c_1}{[\phi^{-1}(t)]^{d}} \qquad\text{ for } \ t>0, \ |x-y|\le c_2 \phi^{-1}(t).
	\end{align*}
\end{proposition}

\begin{proof}
	Because of \eqref{eq:scail_cut}, it is enough to show that 
	\begin{align*}
	{q^{(\phi^{-1}(t))}}(1, x, y)\ge c_1 \qquad\qquad\text{ for } |x-y|\le c_2
	\end{align*}
	for some $c_1, c_2>0$.
	Let $\kappa:=\phi^{-1}(t)$. With the same discussion as in the proof of \autoref{thm:Holder}, 
	there exist $c, {\varrho}>0$ such that {for $\lambda>0$}
	\begin{align*}
	|{ q^{(\kappa)}}(1, x, z)-{q^{(\kappa)}}(1, y, z)|&=|Q^{(\kappa)}_{1/2} f(x)- Q^{(\kappa)}_{1/2} f(y)|\nn\\
	&=|U^{(\kappa)}_{\lambda} h(x)-U^{(\kappa)}_{\lambda} h(y)|\le c|x-y|^{\varsigma}.
	\end{align*}
	{ Here $f(z):=q^{(\kappa)}(1/2, z, y)$ and $h(z):=\int_0^\infty (\lambda +\mu)e^{-\mu t}dE_\mu^{(\kappa)} (f)$ where $\mu>0$ and $E_\mu^{(\kappa)}$ is the projection operator in $L^2(\Rd)$ related to $Q^{(\kappa)}$ and $U^{(\kappa)}$. }
	Therefore, we have our assertion by the symmetry of ${q^{(\kappa)}}(t, x, y)$ and  \autoref{p:onlower}.
\end{proof}

\medskip

We finally obtain the lower bound estimates of $p(t, x, y)$ for any $t>0$ and $x,y\in \Rd$ using the similar method as in \cite[Theorem 4.21]{Xu13}.

\begin{theorem}
	There exists a positive constant $c=c(\phi)$ such that for any $t>0,\ x, y\in \Rd$,
	\begin{align*}
	p(t, x, y)\ge c [\phi^{-1}(t)]^{-d}\prod_{i=1}^d\left(1\wedge \frac{t\phi^{-1}(t)}{|x^i-y^i|\phi(|x^i-y^i|)}\right).
	\end{align*}
\end{theorem}
\begin{proof}
	We first note that by  \autoref{p:exit},  there exists ${ \varrho}_0\in (0, 2^{-1})$ such that for any $w\in\Rd$ and $r>0$,
	\begin{align}\label{eq:gam}
	\Pp^w(\tau_{B(w, 4^{-d} r)}\ge  { \varrho}_0 \phi(r))\ge 1/2.
	\end{align}
	Let ${\varrho}:=2{\varrho}_0\in (0, 1)$. By  \autoref{p:lower_on}, it remains to show that there exists $c_1>0$ such that for any $t>0$ and $x, y\in \Rd$ satisfying $|x^i-y^i|\ge  \phi^{-1}({\varrho}^{-1}t)$ for each $i\in \{1, \ldots, d\}$,
	\begin{align*}
	p(t, x, y)\ge c_1 [\phi^{-1}(t)]^{-d}\prod_{i=1}^d \frac{t\phi^{-1}(t)}{|x^i-y^i|\phi(|x^i-y^i|)}.
	\end{align*}
	Consider points $\xi_{(0)},\xi_{(1)},\ldots, \xi_{(d)}\in \Rd$ tracing from $x$ to $y$ where the only difference of $\xi_{(i-1)}$ and $\xi_{(i)}$ is the value of $i$-th coordinate, that is,
	\begin{align*}
	\xi_{(0)}=x, \ \xi_{(d)}=y\,\ \text{ and }
	\xi_{(k)}:=(\xi_{(k)}^1, \ldots, \xi_{(k)}^d)\,\ \text{ where } 
	\begin{cases}
	\xi_{(k)}^i=y^i\,\,\,\text{if } i\le k,\\
	\xi_{(k)}^i=x^i\,\,\,\text{if } i> k.
	\end{cases}
	\end{align*}
	For each $k=1, 2, \ldots, d$, consider cubes $Q_{k}:=Q(\xi_{(k)}, r_k)$ centered at $\xi_{(k)}$ with side length $r_k:= \phi^{-1}({ \varrho}^{-1}t)/4^{d-k}$.
	Since $p(t, w, y)\ge c_0 [\phi^{-1}(t)]^{-d}$ for $w\in Q_d$ by \autoref{p:lower_on},
	the semigroup property implies that 
	\begin{align}\label{eq:up_semi}
	p(d\cdot t, x, y)&
	=\int_{\Rd}\cdots\int_{\Rd}p(t, x, \eta_{(1)})\,p(t, \eta_{(1)}, \eta_{(2)})\cdots p(t,\eta_{(d)}, y)\ d\eta_{(1)}\cdots d\eta_{(d)}\nn\\
	&\ge \int_{Q_d}\cdots\int_{Q_1}p(t, x, \eta_{(1)})\,p(t, \eta_{(1)}, \eta_{(2)})\cdots p(t,\eta_{(d)}, y)\ d\eta_{(1)}\cdots d\eta_{(d)}\nn\\
	&\ge   c_2 [\phi^{-1}(t)]^{-d}\Pp^x(X_t\in Q_1)\prod_{k=1}^{d-1} \inf_{\eta_{(k)}\in Q_k}\Pp^{\eta_{(k)}} (X_t\in Q_{k+1}).
	\end{align}
	Let $ \widetilde r:=\phi^{-1}({\varrho}^{-1}t)$.
	Then $r_k/2\ge 4^{-d}\widetilde r$, and
	\eqref{eq:gam} with the fact that ${ \varrho}_0={\varrho}/2$ implies that for any $w\in \Rd$,
	\begin{align}\label{eq:ex1}
	\Pp^w(\tau_{B(w, r_k/2 )}\ge  t/2)
	&\ge \Pp^w(\tau_{B(w, 4^{-d} \widetilde r)}\ge  t/2)\nn\\
	&=\Pp^w(\tau_{B(w, 4^{-d} \widetilde r)}\ge  {\varrho}_0 \phi( \widetilde r))\ge 1/2.
	\end{align}
	Let $\overline{Q}_k:=Q(\xi_{(k)}, r_k/2)$.
	For any $t>0$, define $\mathcal M_t^{(k)}:=\{\omega: X_t(\omega) \text{ hits } \overline{Q}_k \text{ by } t\}$.
	% and $\mathcal{N}_t^{(k)}:=\{\omega:\sup_{s\le t}|X_s(\omega)-X_0(\omega)|\le r_k/2\}$.
	The strong Markov property with \eqref{eq:ex1} and L\'evy system in \eqref{eq:LSd} imply that for any $\eta\in Q_k$,
	\begin{align}\label{eq:1}
	\Pp^{\eta} (X_t\in Q_{k+1})
	&\ge 2^{-1}\Pp^{\eta}\big(\cM_{t/2}^{(k+1)}\big)\nn\\
	&\ge c_3\ \E^{\eta} \Big[\int^{t/2\wedge \tau_{Q(\eta, r_k)}}_0\int_{\overline Q_{k+1}} \frac{1}{|X_s-u|\phi(|X_s-u|)} m(du)ds\Big]
	\end{align}
	where $m(du)$ is the measure on $\sum_{i=1}^{d}\R$ restricted only on each coordinate.
	Let $k\in\{1,\ldots, d-1\}$.
	For $w\in Q(\xi_{(k)}, 2r_k)$ and $u\in Q(\xi_{(k+1)}, 2r_k)$,
	since $\xi_{(k)}^{k+1}=x^{k+1}$, $\xi_{(k+1)}^{k+1}=y^{k+1}$ and $|x^{k+1}-y^{k+1}|\ge r_k$, we have that
	\begin{align*}
	%\label{eq:dis}
	|w^{k+1}-u^{k+1}|&\le |x^{k+1}-y^{k+1}|+|x^{k+1}-w^{k+1}|+|y^{k+1}-u^{k+1}|\nn\\
	&= |x^{k+1}-y^{k+1}|+|\xi_{(k)}^{k+1}-w^{k+1}|+|\xi_{(k+1)}^{k+1}-u^{k+1}|\le 5 |x^{k+1}-y^{k+1}|.
	\end{align*}
	Also \eqref{eq:ex1} implies 
	\begin{align}\label{eq:e2}
	\E^\eta\left[ \frac{t}{2}\wedge \tau_{Q(\eta, r_k)}\right]\ge t\Pp^{\eta}\left(\tau_{Q(\eta, r_k)}> \frac{t}{2}\right)\ge 2^{-1} t.
	\end{align}
	Therefore,  \eqref{eq:1}--\eqref{eq:e2} with {\bf (WS)} and the fact that 
	${\phi^{-1}({ \varrho}^{-1}t)}/{\phi^{-1}(t)}\asymp 1$ imply that for any $\eta\in Q_k$ and $k\in \{1, \ldots, d-1\}$,
	\begin{align}\label{eq:upk}
	\Pp^{\eta} (X_t\in Q_{k+1})&\ge c_4\ \E^\eta\left[ \frac{t}{2}\wedge \tau_{Q(\eta, r_k)}\right]
	\frac{\phi^{-1}(t)}{|x^{k+1}-y^{k+1}|\phi(|x^{k+1}-y^{k+1}|)}\nn\\
	&\ge \frac{c_4}{2}\frac{t\phi^{-1}(t)}{|x^{k+1}-y^{k+1}|\phi(|x^{k+1}-y^{k+1}|)}.
	\end{align}
	In a similar way to obtain \eqref{eq:upk}, we have that
	\begin{align}\label{eq:up1}
	\Pp^{x} (X_t\in Q_{1})\ge c_5\frac{t\phi^{-1}(t)}{|x^{1}-y^{1}|\phi(|x^{1}-y^{1}|)}.
	\end{align}
	Hence we have the lower bound for $p(t, x, y)$ by plugging \eqref{eq:upk} and \eqref{eq:up1} into \eqref{eq:up_semi}.
\end{proof}

%%%%%%%%%%%%%%%%%%%%%%%

\section{Appendix: Proof of \autoref{prop:main}}\label{sec:proof-propo}
In this section, we present the proof of \autoref{prop:main}. 
For the convenience of notations, we use constants $c$ instead of $c_i$, $i=1,2,\ldots$  in the proofs even the values are changed.
The main idea in this section follows from \cite[Section 4]{KKKpre}.

\subsection{Preliminary estimates} 
We first give the definition of $D_k \subset \R^d$ as in \cite{KKKpre}.
\begin{definition}\label{d:D}{\ }
\begin{itemize}
		\item[(0)] Define $D_0 =  \bigcup_{i=1,\ldots, d} \{ |x^i| < 1\} \cup (-2,2)^d $.
		\item[(1)] Given $k\in \N, \gamma:=(\gamma^1,\cdots,\gamma^d) \in \N_0^d$ with $\sum_{i=1}^d\gamma^i = k$ and $\switch \in \{-1,1\}^d \,$,
		define a box (hyper-rectangle) $D_k^{\gamma, \switch}$ by
		\begin{align*}
		D_k^{\gamma, \switch} = \switch^1 [2^{\gamma^1},2^{\gamma^1+1}) \times 
		\switch^2 [2^{\gamma^2},2^{\gamma^2+1}) \times \ldots \times 
		\switch^d [2^{\gamma^d},2^{\gamma^d+1}) \,. 
		\end{align*}
		\item[(2)] Given $k\in \N$ and $\gamma \in \N_0^d$ with $\sum_{i=1}^d \gamma^i = k$, 
		define
		\begin{align*}
		D_k^{\gamma} = 
		\bigcup_{\switch \in \{-1,1\}^d}  D_k^{\gamma, \switch}	\,. 
		\end{align*}
		\item[(3)] Given $k\in \N$, define
		\begin{align*}
		D_k = \bigcup_{\gamma \in \N_0^d:\, \sum_{i=1}^d \gamma^i = k}  D_k^{\gamma}  \,. 
		\end{align*}
	\end{itemize}
\end{definition}

Then we have the following Remark by \cite[Lemma 4.2]{KKKpre}.

\begin{remark}{\ }\label{rem:D}	Let $k \in \N_0,  \gamma \in \N_0^d$ and $\switch \in \{-1,1\}^d \,$.
	\begin{itemize}
		\item[(1)] Given $k\in \N, \gamma$ with $\sum_{i=1}^d \gamma^i = k$, there are $2^d$ sets of the form $D_k^{\gamma, \switch}$, 
		and $|D_k^{\gamma, \switch}|= \prod_{i=1}^d 2^{\gamma^i} = 2^{k}$. 
		\item[(2)] Given $k \in \N, \switch \in \{-1,1\}^d$, there are 
		$\big( \begin{array}{c} d+k-1 \\ d-1 \end{array} \big)$ sets $D_k^{\gamma, \switch}$ with $\sum_{i=1}^d \gamma^i = k$. Thus, the set $D_k$ consists of $2^d \big(\begin{array}{c} d+k-1 \\ d-1 \end{array} \big)$ disjoint boxes. 
		\item[(3)] $D_k \cap D_l = \emptyset$ if $k \ne l$ and 
		$\bigcup_{k \in \N_0} D_k = \R^d$.
	\end{itemize}
\end{remark}

Now we give the definitions of the shifted boxes centered at $y_0$.
For $y_0\in \R^d$, $t>0$ and $\kappa=\phi^{-1}(t)$, 
let $A_0:=y_0+\kappa D_0$. For $k \in \N, \gamma \in \N_0^d$ with $\sum_{i=1}^d \gamma^i = k$ and $\switch \in \{-1,1\}^d$, 
\begin{align*}
A_{k, \gamma, \switch}:=y_0+\kappa D_k^{\gamma, \switch} \,, \quad
A_{k, \gamma}:= y_0+\kappa D_k^{\gamma}\quad\text{  and }\quad
A_k := y_0+\kappa D_k. \end{align*}
By the definition of $D_k$, 
it is easy to see that $A_{k}\cap A_{l}=\emptyset$ for $k\neq l$ and $\cup_{k=0}^{\infty}A_k =\R^d$.
\medskip

For the rest of  \autoref{sec:proof-propo}, we assume $l\in \{0,1,\ldots, d-1\}$, $i_0 \in \{1, \dots, d-l\}$ and assume $x_0, y_0 \in \R^d$  satisfy the condition $\mathcal{R}(i_0)$ for some $i_0$ (see, \autoref{def:theta_and_R}). 
For $t>0$, set $\kappa = \phi^{-1}(t)$.
Then there exist $\lengthR_i\ge 1$, $i\in\{i_0, \ldots, d\}$ such that 
\begin{align*}
\tfrac{5}{4}2^{\lengthR_i}\kappa\le  |x_0^{i}-y_0^{i}| < \tfrac{10}{4}2^{\lengthR_i}\kappa\,\qquad\mbox{and}\qquad	R_i=2^{\lengthR_i}\kappa\,.
\end{align*}

Since the proofs of the following results reveal the same geometric structures as shown in \cite[Lemma 4.3, Remark 4.4, Lemma 4.5, Remark 4.6]{KKKpre}, we omit them.
\begin{lemma}\label{lem:jAn} Assume $x_0, y_0 \in \R^d$ satisfy condition $\mathcal{R}(i_0)$ for some $i_0$.
	Set $s(j_0):= \frac{R_{j_0}}{8}$ for  $j_0\in \{i_0,\ldots, d\}$. Then the following holds true:
	\begin{align*}
	&	\bigcup\limits_{u \in B(x_0,s(j_0))} \{ u +  h {e^i} | \, h \in \R\}\subset y_0+ \big(\bigotimes_{j=1}^{i-1}\mathcal{J}_{\lengthR_{j_0}}\times \R\times\bigotimes_{j=i+1}^{j_0-1}\mathcal{J}_{\lengthR_{j_0}}\times\bigotimes_{j=j_0}^{d} \mathcal{I}_{\lengthR_j}
	\big) 
	&\mbox{ if }i< j_0\,,\\
	&	\bigcup\limits_{u \in B(x_0,s(j_0))} \{ u + h e^i | \, h \in \R\}\subset y_0+ \big(\bigotimes_{j=1}^{j_0-1}\mathcal{J}_{\lengthR_{j_0}}\times\bigotimes_{j=j_0}^{i-1} \mathcal{I}_{\lengthR_j}
	\times\R\times\bigotimes_{j=i+1}^{d} \mathcal{I}_{\lengthR_j}
	\big)
	&\mbox{ if }i\ge j_0\,,
	\end{align*}
	where $\mathcal{I}_{\lengthR_j}:= \pm[2^{\lengthR_j}\kappa, 2^{\lengthR_j+2}\kappa)$, $\mathcal{J}_{\lengthR_{j_0}}:= \pm[0, 2^{\lengthR_{j_0}+2}\kappa)$ and $\lengthR_{j_0}, \ldots, \lengthR_{d} \in \N$. 
\end{lemma}	

\medskip

For $j_0\in \{i_0, \ldots, d\}$, define
\begin{align*}
{ A_k^i}: = A_k \cap  \bigcup\limits_{u \in B(x_0,s(j_0))} \{ u +  h e^i | \,  h \in \R\} \qquad \text{ for $k \in \N_0, i \in \{1,\ldots, d\}$},
\end{align*}
that is, ${A_k^i}$ contains all possible points in $A_k$ 
when the process $X$ leaves the ball $B(x_0,s(j_0))$ by a jump in the $i$-th direction.

\medskip

\begin{remark}\label{rem:nervig}
	Using the above notations together with \autoref{lem:jAn}, the following observation holds true:
	\begin{align*}
	{A_k^i} \ne \emptyset \qquad  &\Longrightarrow 
	\begin{cases}
	k= 0\,\mbox{ or }\, k \geq \sum\limits_{j\in\{j_0,\ldots,d\}} \lengthR_j \quad &\text{ if } i< j_0\,, \\
	k= 0\,\mbox{ or }\, k \geq 
	\sum\limits_{j\in\{j_0,\ldots,d\}\backslash\{i\} }\lengthR_j &\text{ if } i\ge j_0 \,. 
	\end{cases}
	\end{align*}
\end{remark}

\begin{lemma}\label{lem:zyi}
	Let $i\in\{1, \ldots, d\}$ and $k\in\N_0$.
	For any $z\in {A_k^i}$ and $y\in B(y_0, \tfrac{\kappa}{8})$, there exists
	$\gamma\in \N_0^d$ such that
	\begin{align}\label{eq:zyjg}
	|z^j-y^j|\in [a_4 2^{\gamma^j}\kappa,  a_52^{\gamma^j+1}\kappa)
	\qquad	\mbox{ if } k\in \N\ \text{and } j\in \{1, \ldots, d\},
	\end{align}
	for some $ a_4, a_5>0$.
	Moreover,  for $z\in {A_k^i}$ 
	and $y\in B(y_0, \tfrac{\kappa}{8})$, the following holds true:
	\begin{alignat}{2}
	&|z^j-y^j|\in [0,  a_5 2^{\lengthR_{j_0}}\kappa) 
	\qquad &&\mbox{ if } j\in \{1, \ldots, j_0-1\}\backslash\{i\},\label{eq:zyj0}\\
	&|z^j-y^j|\in [a_4 2^{\lengthR_j}\kappa,  a_5 2^{\lengthR_j+1}\kappa)	\quad &&\mbox{ if } j\in \{ j_0\ldots, d\}\backslash\{i\},\qquad\label{eq:zyj}
	\end{alignat}
	for some $ a_4, a_5>0$.
\end{lemma}
\begin{remark}\label{rem:add10-4}
	Given ${A_k^i}$ and $B(y_0, \tfrac{\kappa}{8})$, 
	$\gamma\in \N_0^d$ can be chosen independently of $z\in {A_k^i}$ and $y\in B(y_0, \tfrac{\kappa}{8})$.
\end{remark}

\medskip

For $s(j_0)=\frac{R_{j_0}}{8}$, $x\in B(x_0, \frac{\kappa}{8})$ and $\tau=\tau_{B(x_0, s(j_0))}$, let
\begin{align*}
\Psi(k)&:=\E^{x}\left[\1_{\{\tau\le t/2\}}\1_{\{X_{\tau}\in A_k\}}P_{t-\tau}f(X_{\tau})\right],\,\,\qquad 
k\in \N_0,\\
\Psi^i(k)&:=\E^{x}\left[\1_{\{\tau\le t/2\}}\1_{\{X_{\tau}\in { A_k^i}\}}P_{t-\tau}f(X_{\tau})\right],\,\qquad k\in \N,\,\, i\in \{1,2,\ldots, d\}\,.\
\end{align*}
Let $f$ be a non-negative Borel function on $\Rd$ supported in $B(y_0, \kappa/8)$. Then
using  $\Psi, \Psi^i$, we decompose the left-hand side of \eqref{eq:main1} according to the following: 
\begin{align*}
\E^{x}\left[\1_{\{\tau\le t/2\}}P_{t-\tau}f(X_{\tau})\right]
&=\sum_{k=0}^{\infty}\Psi(k)
=\sum_{k=1}^{\infty}\Big(\sum_{i=1}^{d}\Psi^i(k)\Big)+\Psi(0)\nn\\
&= \sum_{k=1}^{\infty}\Psi^d(k) + \sum_{k=1}^{\infty}\Psi^{d-1}(k)  +\ldots+   \sum_{k=1}^{\infty}\Psi^1(k)  + \Psi(0).
\end{align*}
Moreover,  using the notation 
$\TT{a}{b}:=\Tj{a}{b}$ for $a\le b$ defined in \autoref{sec:uphke},
\autoref{rem:nervig} yields that the above decomposition can be separated as follows:
\begin{align} \label{eq:main}
\begin{split}
&\E^{x}\left[\1_{\{\tau\le t/2\}}P_{t-\tau}f(X_{\tau})\right] \\
&=  \sum_{k=\TT{j_0}{d}-\lengthR_d}^{\infty}\Psi^d(k) + \sum_{k=\TT{j_0}{d}-\lengthR_{d-1}}^{\infty}\Psi^{d-1}(k) + \ldots + \sum_{k=\TT{j_0}{d}-\lengthR_{j_0}}^{\infty}\Psi^{j_0}(k)\\
& \qquad +\sum_{k=\TT{j_0}{d}}^{\infty}\big(\Psi^{j_0-1}(k)+\ldots+\Psi^1(k)\big)+\Psi(0) \\
&=
\sum_{i={j_0}}^d\mathcal{S}(i)+\sum_{i={1}}^{j_0-1}\mathcal{T}(i)+\Psi(0),
\end{split}
\end{align}
where $\mathcal{S}(i):=\sum_{k=\TT{j_0}{d}-\lengthR_i}^{\infty}\Psi^i(k)$
and $\mathcal{T}(i):=\sum_{k=\TT{j_0}{d}}^{\infty}\Psi^i(k)$.

\medskip

Recall from \eqref{eq:newmm} and \eqref{eq:newmm2} that ${\fN}(\delta):=\frac{\phi(\kappa)}{2^{\delta}\, \phi(2^{\delta}\kappa)}\text{ for } \delta\in \Z$, and there is $c>0$ such that
\begin{align*}
c^{-1}2^{-\delta(\ua+1)}\le \ & \fN(\delta)\le c2^{-\delta(\la+1)}\ \qquad\text{ for }\delta\in  \Z_+.
\end{align*}

\begin{remark}\label{rem:prel}
Here we state some useful results for obtaining the upper bounds of $\Psi(0), \Psi^i(k)$.
Let $x_0, y_0 \in \R^d$ satisfy the condition $\mathcal{R}(i_0)$ for some $i_0$, so that $|x_0^j-y_0^j|\ge \frac54 R_j =\frac54 2^{\lengthR_j }\kappa$ for $j\in \{i_0, \ldots, d\}$.
For $l\in \{0,\ldots, d-1\}$, let $j_0\in\{i_0,\ldots,d-l\}$, and $s(j_0):=\frac{R_{j_0}}{8}$ and $\tau:=\tau_{B(x_0, s(j_0))}$. 	\begin{itemize}
\item [(1)]
 For any $w\in B(x_0, s(j_0))$ and $z\in \Rd$  satisfying $|z^i-y_0^i|\in [2^{\gamma^i}\kappa, 2^{\gamma^i+1}\kappa)$ with $\gamma^i+1\le \lengthR_i$ or  satisfying $|z^i-y_0^i| \in [0,  2\kappa) \cup [2^{ m-1}\kappa, 2^{m+1}\kappa)$ for some $m+1\le \lengthR_i$  for $i\ge j_0$,
\begin{align*} 
	|w^i-z^i|&\ge |x_0^i-y_0^i|-|w^i-x_0^i|-|z^i-y_0^i|\ge 5R_i/4-R_i/8-2^{\lengthR_i}\kappa= R_i/8 \,.
\end{align*}
	For $k\ge 1$, let $I_k^i:=\{\ell \in\R:|\ell - y_0^i|\in [2^{\gamma^i}\kappa, 2^{\gamma^i+1}\kappa),  \gamma^i+1\le \lengthR_i\}$.
	Then for any $x\in B(x_0, \tfrac{\kappa}{8})$ and $i\ge {j_0}$, {\bf(WS)} implies that 
	\begin{align}\label{eq:LSIg}
	\E^x\left[\int_{0}^{t/2\wedge \tau}\int_{I_k^i}
	\frac{1}{|X_s^i-\ell |\phi(|X_s^i-\ell |)} d \ell d s \right]
	\le \frac{c t}{R_i\phi(R_i)}\cdot
	2^{\gamma^i}\kappa= c 2^{\gamma^i}{\fN}(\lengthR_i).
	\end{align}
	For $k=0$, let $ I^i_{0, 0}:=\{\ell \in \R:|\ell - y_0^i| \in [0, 2\kappa)\}$ and $ I^i_{ 0, m}:=\{\ell \in \R:|\ell - y_0^i| \in [2^{m-1}\kappa, 2^{m+1}\kappa)\}$, for some ${m\in \{1, \ldots, \lengthR_i-1\}}$. Then for any $x\in B(x_0, \tfrac{\kappa}{8})$ and $i\ge {j_0}$, {\bf(WS)} implies that 
	\begin{align}\label{eq:LSIg0}
	\E^x\left[\int_{0}^{t/2\wedge \tau}\int_{ I^i_{0, m}}
	\frac{1}{|X_s^i-\ell |\phi(|X_s^i-\ell |)} d \ell d s \right]
	\le \frac{c t}{R_i\phi(R_i)}\cdot
	{2^{m}\kappa}  = c 2^{m}{\fN}(\lengthR_i)
	\end{align}
	for some $m\in \{0, 1, \ldots, \lengthR_i-1\}$.

\item[(2)]
{Let $x\in B(x_0,\tfrac{\kappa}{8})$. For any 	$y\in B(x,  s(j_0)/2 )$, since $s(j_0)=2^{n_{j_0}-3}\kappa\ge 2^{-2}\kappa$,
$$|y-x_0|\le |y-x|+|x-x_0|\le (2^{n_{j_0}-4}+8^{-1})\kappa \le 2^{n_{j_0}-3}\kappa$$ so that 
	$B(x, s(j_0)/2)\subset B(x_0, s(j_0))$. Therefore, by  \autoref{cor:exit_new}  with $r:=2^{\lengthR_{j_0}-4}\ge 2^{-3}$ and {\bf(WS)}, we have that for $x\in B(x_0, \tfrac{\kappa}{8})$ and $t>0$,
\begin{align}\label{eq:Plg}
\Pp^x(\tau\le t/2, X_{\tau}\in {A_k^i})
&\le \Pp^x(\tau\le t/2)\le \Pp^x(\tau_{B(x,  r \kappa )}\le t/2)\nn\\
&\le\frac{c\, t}{\phi(2^{n_{j_0}}\kappa )}= c\,\fN(\lengthR_{j_0})2^{\lengthR_{j_0}}.
	\end{align}}
\item[(3)]
Since $|x_0^{j}-y_0^{j}|\asymp 2^{\lengthR_j}\kappa$ for $j\in \{j_0, \ldots, d\}$,
we note that
\begin{align}
\prod_{j=j_0}^{d-l}
\Big(\frac{t\phi^{-1}(t)}{|x_0^{j}-y_0^{j}|\phi(|x_0^{j}-y_0^{j}|)}\Big)^q\prod_{j=d-l+1}^d\Big(\frac{t\phi^{-1}(t)}{|x_0^{j}-y_0^{j}|\phi(|x_0^{j}-y_0^{j}|)}\Big)&\asymp \mathbf F_{j_0}(l)\nn\\
\text { where } 
\mathbf F_{j_0}(l):=
\prod_{j=j_0}^{d-l}
\fN(\lengthR_j)^q\prod_{j=d-l+1}^d\fN(\lengthR_j).
\label{eq:prodlg}
\end{align}
\end{itemize}
\end{remark}

\subsection{Estimates for  $\Psi(0)$, $\mathcal{S}(i)$, $i\in\{j_0, \ldots, d\}$ and $\mathcal{T}(i)$, $i\in\{1,\ldots,j_0-1\}$.}
\subsubsection*{\bf Estimate of \texorpdfstring{$\Psi(0)$}{F0}.}
First, we obtain the upper bound of
\[ 
P_{t-\tau}f(z)= \int_{B(y_0, \frac{\kappa}{8})} p(t-\tau,z, y) f(y) dy 
\]
for $z\in A_0^{i}$ and $t/2\le t-\tau\le t$. 
By {(\bf WS)}, we note that 
\begin{align}\label{eq:invphi}
\phi^{-1}(t)\asymp \phi^{-1}(t-\tau),
\end{align}
and $\Big(1\wedge \frac{t\phi^{-1}(t)}{|r|\phi(|r|)}\Big)\le 1$.  
For  $z\in A_0^{i}$ and $y\in B(y_0, \kappa/8)$, by \eqref{eq:zyj0} and \eqref{eq:zyj}, the following holds:  
\begin{align*}
\begin{cases}
|z^j-y^j|\in [0, { a_5} 2^{\lengthR_{j_0}}\kappa)&\mbox{ if } j\in \{1, \ldots, j_0-1\}\backslash \{i\},\\
|z^j-y^j|\in [{a_4} 2^{\lengthR_{j}}\kappa, { a_5} 2^{\lengthR_{j}+1}\kappa)&\mbox{ if } j\in \{j_0, \ldots, d\}\backslash \{i\}\ ,
\end{cases}
\end{align*}
for some $a_4, a_5>0$. 	Let $s_j:=|z^j-y^j|$ for $j\in \{1, \ldots, d\}$. 
We first consider the case $d-l < i$.
  Since 
$\HHq{q}{l}$ (to be precise $\HHq{q}{l}'$) 
yields that for $t/2\le t-\tau\le t$, 
\begin{align*}
p(t-\tau,z, y)
&\le \ c  [\phi^{-1}(t)]^{-d}
\prod_{j\in \{1, \ldots, d-l-1\}}\left(\frac{t\phi^{-1}(t)}{s_j\phi(s_j)}\wedge 1\right)^q
\prod_{j\in \{d-l+1, \ldots, d\}\backslash\{i\}}\left(\frac{t\phi^{-1}(t)}{s_{j}\phi(s_{j})}\wedge 1\right)\nn\\
&\cdot
\left(\frac{t\phi^{-1}(t)}{s_{d-l}\phi(s_{d-l})\wedge s_{i}\phi(s_i)}\wedge 1\right)^q\left(\frac{t\phi^{-1}(t)}{s_{d-l}\phi(s_{d-l})\vee s_{i}\phi(s_i)}\wedge 1\right),
\end{align*}
we have that 
\begin{align}\label{eq:new1}
p(t-\tau,z, y)
&\le c [\phi^{-1}(t)]^{-d}\nn\\
&\cdot\begin{cases}
\prod_{j\in \{j_0,\dots, d-l-1\}}\fN(\lengthR_j)^q
\prod_{j\in \{d-l, \ldots, d\}\backslash\{i\}}\fN(\lengthR_j)
&{\text{ if }  s_i< 2^{2}\kappa\ ,}\\
\prod_{j\in \{j_0,\dots, d-l\}}\fN(\lengthR_j)^q
\prod_{j\in \{d-l+1, \ldots, d\}}\fN(\lengthR_j)
&{ \text{ if }  s_i\ge 2^{\lengthR_i-1}\kappa}.
\end{cases}
\end{align}
Otherwise, $s_i\in[2^2\kappa, 2^{\lengthR_i-1}\kappa)$. In this case, 
there exists a constant $m\in \{1, \ldots, \lengthR_i -1\}$ such that $s_i\asymp 2^m \kappa$ and 
\begin{align}\label{eq:new11}
&p(t-\tau,z, y)
\le  \ c [\phi^{-1}(t)]^{-d}\nn\\
&\cdot\begin{cases}
\prod_{j\in \{j_0,\dots, d-l-1\}}\fN(\lengthR_j)^q{ \fN(m)^q}
\prod_{j\in \{d-l, \ldots, d\}\backslash\{i\}}\fN(\lengthR_j)
&{\text{if }  m<\lengthR_{d-l}\,,}\\
\prod_{j\in \{j_0,\dots, d-l\}}\fN(\lengthR_j)^q
%{ \fN(m)}
\prod_{j\in \{d-l+1, \ldots, d\}\backslash\{i\}}\fN(\lengthR_j){\fN(m)}
&{ \text{if } \lengthR_{d-l}\le m<\lengthR_{i}}.
\end{cases}
\end{align}
We obtain the upper bounds in a similar way
for the rest of  cases: $j_0 \leq i  \le d-l$ and  $i < j_0$.
If  $ i  \le d-l$, $\HHq{q}{l}$ (to be precise $\HHq{q}{l}'$) 
yields that  for $t/2\le t-\tau\le t$, 
\begin{align*}
p(t-\tau,z, y)&
\le\ c  [\phi^{-1}(t)]^{-d}\prod_{j\in \{1, \ldots, d-l\}\backslash\{i\}}
\left(\frac{t\phi^{-1}(t)}{s_j\phi(s_i)}\wedge 1\right)^q
\prod_{j\in \{d-l+2, \ldots, d\}}\left(\frac{t\phi^{-1}(t)}{s_j\phi(s_j)}\wedge 1\right)
\nn\\
&\cdot
\left(\frac{t\phi^{-1}(t)}{s_{d-l+1}\phi(s_{d-l+1})\wedge s_{i}\phi(s_i)}\wedge 1\right)^q\left(\frac{t\phi^{-1}(t)}{s_{d-l+1}\phi(s_{d-l+1})\vee s_{i}\phi(s_i)}\wedge 1\right).
\end{align*}
Therefore, for  $j_0 \leq i  \le d-l$,

\begin{align}
p(t-\tau,z, y)
&\le  \ c [\phi^{-1}(t)]^{-d}\nn\\
&\cdot\begin{cases}
\prod_{j\in \{j_0,\dots, d-l\}\backslash\{i\}}\fN(\lengthR_j)^q
\prod_{j\in \{d-l+1, \ldots, d\}}\fN(\lengthR_j)
&{\text{ if }  s_i< 2^{2}\kappa\ ,}\\
\prod_{j\in \{j_0,\dots, d-l\}}\fN(\lengthR_j)^q
\prod_{j\in \{d-l+1, \ldots, d\}}\fN(\lengthR_j)
&{ \text{ if }  s_i\ge 2^{\lengthR_i-1}\kappa}\ ,
\end{cases}\label{eq:new2}
\end{align}
and otherwise, there exists a constant $m\in \{1, \ldots, \lengthR_i -1\}$ such that $s_i\asymp 2^m \kappa$ and 
\begin{align}\label{eq:new22}
p(t-\tau,z, y)
\le  \ c [\phi^{-1}(t)]^{-d}
\prod_{j\in \{j_0,\dots, d-l\}\backslash\{i\}}\fN(\lengthR_j)^q{ \fN(m)^q}
\prod_{j\in \{d-l+1, \ldots, d\}}\fN(\lengthR_j).
\end{align}
If  $i< j_0 $, we have that
\begin{align}\label{eq:new3}
p(t-\tau,z, y)
\le  \ c [\phi^{-1}(t)]^{-d}
\prod_{j\in \{j_0,\dots, d-l\}}\fN(\lengthR_j)^q
\prod_{j\in \{d-l+1, \ldots, d\}}\fN(\lengthR_j).
\end{align}
Recall that for $l\in \{1, \ldots, d\}$ and $j_0\in \{i_0, \ldots, d-l\}$,
$$\mathbf F_{j_0}(l)=
\prod_{j=j_0}^{d-l}
\fN(\lengthR_j)^q\prod_{j=d-l+1}^d\fN(\lengthR_j).$$
Altogether, 
$\HHq{q}{l}$ yields for $t/2\le t-\tau\le t$ and $z\in A_0^{i}$, 
we obtain the upper bound of $P_{t-\tau} f(z)$ according to the point $z\in A_0^{i}$.
For $z\in A_0^{i}$ with { $i\ge j_0$ and $|z^i-y_0^i|<2\kappa$}, by \eqref{eq:new1} and \eqref{eq:new2}, 
\begin{align}
&P_{t-\tau}f(z)
=\int_{B(y_0, \frac{\kappa}{8})} p(t-\tau,z, y) f(y) dy\nn\\
&\le \, c  [\phi^{-1}(t)]^{-d} \|f\|_1\mathbf F_{j_0}(l)
\begin{cases}
\fN(\lengthR_{d-l})^{-q+1}\fN(\lengthR_i)^{-1}
\1_{\{|z^i-y_0^i|< 2\kappa\}}
&\mbox{if $d-l< i$},\\
\fN(\lengthR_{i})^{-q}
\1_{\{|z^i-y_0^i| < 2\kappa \}}
&\mbox{if $ j_0\le i \le d-l$}.
\end{cases}\label{eq:Hlq0g1}
\end{align} 
If $z\in A_0^{i}$ with { $i\ge j_0$ and $2\kappa\le |z^i-y_0^i|<2^{\lengthR_i}\kappa$}, there exists $m\in \{1, \ldots, \lengthR_i-1\}$ such that $ 2^{m} \kappa\le |z^i-y_0^i|< 2^{m+1} \kappa$.
So \eqref{eq:new11} and \eqref{eq:new22} imply that 
\begin{align}
P_{t-\tau}f(z)
&=\int_{B(y_0, \frac{\kappa}{8})} p(t-\tau,z, y) f(y) dy
\le  c  [\phi^{-1}(t)]^{-d} \|f\|_1\nn\\
&\cdot \mathbf F_{j_0}(l)
\begin{cases}
\fN(m)^q\fN(\lengthR_{d-l})^{-q+1}\fN(\lengthR_i)^{-1}
&\mbox{if $d-l< i$ and $m<\lengthR_{d-l}$}\,,\\
\fN(m)\fN(\lengthR_{i})^{-1}
&\mbox{if $ d-l< i$ and $\lengthR_{d-l}\le m<\lengthR_{i}$}\,,\\
\fN(m)^q\fN(\lengthR_{i})^{-q}
&\mbox{if $ j_0\le i\le d-l$}\,.
\end{cases}\label{eq:Hlq0g2}
\end{align}
For the rest of cases, that is, if $z\in A_0^{i}$ with $i\ge j_0$ and $|z^i-y_0^i|\ge 2^{\lengthR_i}\kappa$, or with $i< j_0$,
\eqref{eq:new1}, \eqref{eq:new2}  and \eqref{eq:new3} imply that
\begin{align}
P_{t-\tau}f(z)=\int_{B(y_0, \frac{\kappa}{8})} p(t-\tau,z, y)  f(y) dy
\le  c  [\phi^{-1}(t)]^{-d} \|f\|_1\cdot \mathbf F_{j_0}(l)\,.\label{eq:Hlq0g3}
\end{align}
Let us estimate $\Psi(0)$. 
By \eqref{eq:LSd}, \eqref{eq:LSIg0} and \eqref{eq:Hlq0g1}, we first observe that
\begin{align}\label{eq:s01g}
&\E^{x}\left[\1_{\{\tau\le t/2\}}\1_{\{X_{\tau}\in A_0^{i}, i\ge j_0, { |X_\tau^i-y_0^i|< 2\kappa}\}}P_{t-\tau}f(X_{\tau})\right]\nn\\
\le\,&c  [\phi^{-1}(t)]^{-d} \|f\|_1 \mathbf F_{j_0}(l)
\begin{cases}
\fN(\lengthR_{d-l})^{-q+1}
&\mbox{  if $ d-l< i$}\,,
\\
\fN(\lengthR_i)^{-q+1}
&\mbox{  if $ j_0\le i\le d-l$}
\end{cases}\nn\\
\le\,& c  [\phi^{-1}(t)]^{-d} \|f\|_1 \mathbf F_{j_0+1}(l) 
\begin{cases}
\fN(\lengthR_{j_0})
\qquad\quad&\mbox{  if $ d-l< i$}\,,\\
\fN(\lengthR_{j_0})
\qquad\quad&\mbox{  if $ j_0\le i\le d-l$}.
\end{cases}
\end{align}
The last inequality holds since $q\le 1$ and $\delta\to \fN(\delta)$ is decreasing.
Note that
$2^m\le { c}\fN(m)^{-\frac{1}{1+\newa}}$ by \eqref{eq:newmm2}.
Then \eqref{eq:LSd}, \eqref{eq:LSIg0} and \eqref{eq:Hlq0g2} imply that  for any $m\in \{1, \ldots,\lengthR_i-1\}$,
\begin{align*}
%\label{eq:s01g2}
&\E^{x}\left[\1_{\{\tau\le t/2\}}\1_{\{X_{\tau}\in A_0^{i}, i\ge j_0, { 2^{ m-1}\kappa\le |X_\tau^i-y_0^i|\le 2^{m+1}\kappa}\}}P_{t-\tau}f(X_{\tau})\right]\\
\le\,&c  [\phi^{-1}(t)]^{-d} \|f\|_1 \mathbf F_{j_0}(l)
\begin{cases}
\fN(m)^{q-\frac{1}{1+\newa}}\fN(\lengthR_{d-l})^{-q+1}
&\mbox{if $ d-l< i$ and $m<\lengthR_{d-l}$}\,,\\
\fN(m)^{1-\frac{1}{1+\newa}}
&\mbox{if $ d-l< i$ and $\lengthR_{d-l}\le m<\lengthR_{i}$}\,,\\
\fN(m)^{q-\frac{1}{1+\newa}}\fN(\lengthR_{i})^{-q+1}
&\mbox{if $ j_0\le i \le d-l$}\,.
\end{cases}\nn\\
\le\,&c  [\phi^{-1}(t)]^{-d} \|f\|_1 \mathbf F_{j_0+1}(l)
\begin{cases}
\fN(m)^{q-\frac{1}{1+\newa}}\fN(\lengthR_{j_0})^q\fN(\lengthR_{d-l})^{-q+1}
&\mbox{if $d-l< i$ and $m<\lengthR_{d-l}$}\,,\\
\fN(\lengthR_{j_0})^{\newb+q}
&\mbox{if $d-l< i$ and $\lengthR_{d-l}\le m<\lengthR_{i}$}\,,\\
\fN(m)^{q-\frac{1}{1+\newa}}\fN(\lengthR_{j_0})^q\fN(\lengthR_{i})^{-q+1}
&\mbox{if $ j_0\le i \le d-l$}\,.
\end{cases}\nn
\end{align*}
Since 
$\fN(m)^{q-\frac{1}{1+\newa}}\le 1 $ for $q>\frac{1}{1+\newa}$ and
$\fN(m)^{q-\frac{1}{1+\newa}}\le \fN(\lengthR_{i}\wedge \lengthR_{d-l})^{q-\frac{1}{1+\newa}}$ for $q<\frac{1}{1+\newa}$ and $m<\lengthR_{i}\wedge \lengthR_{d-l}$, 
\begin{align}\label{eq:s01g2}
&\E^{x}\left[\1_{\{\tau\le t/2\}}\1_{\{X_{\tau}\in A_0^{i}, { i\ge j_0, 2^{ m-1}\kappa\le |X_\tau^i-y_0^i|< 2^{m+1}\kappa}\}}P_{t-\tau}f(X_{\tau})\right]\nn\\
&\quad\le\, c {[\phi^{-1}(t)]^d} \|f\|_1 \mathbf F_{j_0+1}(l) 
\begin{cases}
\fN(\lengthR_{j_0})^{\newb+q}&\mbox{ if } q\in \big[0, \frac{1}{1+\newa}\big)\,,\\
\fN(\lengthR_{j_0})&\mbox{ if }  q\in(\frac{1}{1+\newa},1).
\end{cases} 
\end{align}
{ For the rest of cases,}  by \eqref{eq:newmm2}, \eqref{eq:Plg} and \eqref{eq:Hlq0g3}, we have that 
\begin{align}\label{eq:s02g}
&\E^{x}\left[\1_{\{\tau\le t/2\}}\1_{\{X_{\tau}\in A_0^{i}, i\ge j_0, |X_\tau^i-y_0^i|\ge 2^{\lengthR_i}\}\bigcup\{X_{\tau}\in A_0^{i}, { i< j_0} \}}\ P_{t-\tau}f(X_{\tau})\right]\\
%+\E^{x}\left[\1_{\{\tau\le t/2\}}P_{t-\tau}f(X_{\tau})\right]\nn\\
\le\, & c  [\phi^{-1}(t)]^{-d} \|f\|_1 \mathbf F_{j_0}(l)
\fN(\lengthR_{j_0})
2^{\lengthR_{j_0}}\le\, c  [\phi^{-1}(t)]^{-d} \|f\|_1 \mathbf F_{j_0+1}(l) \cdot \fN(\lengthR_{j_0})^{\newb+q}.\nn
\end{align} 
Therefore, by \eqref{eq:s01g}--\eqref{eq:s02g}, we obtain 
\begin{align}\label{eq:S0g}     
\Psi(0)
=\sum_{i=1}^d\Psi^i(0)
\le c  [\phi^{-1}(t)]^{-d} \|f\|_1 
\mathbf F_{j_0+1}(l)\begin{cases}                             
\fN(\lengthR_{j_0})^{\newb+q}&\mbox{if } q\in \big[0, \frac{1}{1+\newa}\big)\,,\\
\fN(\lengthR_{j_0})&\mbox{if } q\in(\frac{1}{1+\newa},1).
\end{cases} 
\end{align}
\subsection*{\bf Estimates of \texorpdfstring{$\mathcal{S}(i):=\sum_{k=\TT{j_0}{d}-\lengthR_i}^{\infty}\Psi^i(k)$ for $i\in \{j_0, \ldots, d\}$}{Si}.}

Let $l\in\{0,1,\ldots, d-1\}$, $i_0\in\{1,\ldots, d-l\}$ and $j_0\in\{i_0,\ldots, d-l\}$. 
For $z\in {A_k^i}$ and $y\in B(y_0, \kappa/8)$, 
let $s_j:=|z^j-y^j|$ for $j\in \{1, \ldots, d\}$, and by \eqref{eq:zyjg} and \eqref{eq:zyj}, we note that
\begin{align}\label{c:sS}
s_j\asymp 2^{\gamma^j}\kappa\text{ for }j\in \{1, \ldots, j_0-1\}\cup\{i\}\text{  and } s_j\asymp 2^{\lengthR_{j}}\kappa\text{ for }j\in \{j_0, \ldots, d\}\backslash\{i\}
\end{align}

The index $\gamma=(\gamma^1,\cdots,\gamma^d)\in \N_0^d$ of $A_{k, \gamma}$ is determined when the jump occurs at $\tau$ from the ball $B(x_0, s(j_0))$ and it is independent of the choice of the elements $z\in A_k^i$ and $y\in B(y_0, \kappa/8)$, as mentioned in  \autoref{rem:add10-4}.
Hence, $\HHq{q}{l}$, together with \eqref{eq:invphi} and \eqref{c:sS}, yields that  for $t/2\le t-\tau\le t$ and $i\le d-l$,
\begin{align}\label{eq:Hlqg1}
&p(t-\tau,z, y)
\le c [\phi^{-1}(t)]^{-d}
\prod_{j\in \{1, 2, \ldots, d-l\}}\Big(\frac{t\phi^{-1}(t)}{s_j\phi(s_j)}\Big)^q
\prod_{j\in \{d-l+1, \ldots, d\}}\Big(\frac{t\phi^{-1}(t)}{s_j\phi(s_j)}\Big)\nn\\
&\le c [\phi^{-1}(t)]^{-d}
\prod_{j\in \{1, \ldots, j_0-1\}\cup\{i\}}
\Big(\frac{\phi(\kappa)}{2^{\gamma^j}\, \phi(2^{\gamma^j}\kappa)}\Big)^q
\prod_{j\in \{j_0, \ldots, d-l\}\backslash\{i\}}\Big(\frac{\phi(\kappa)}{2^{\lengthR_j}\, \phi(2^{\lengthR_j}\kappa)}\Big)^q\nn\\
&\qquad\qquad\cdot\prod_{j\in \{d-l+1, \ldots, d\}}\Big(\frac{\phi(\kappa)}{2^{\lengthR_j}\, \phi(2^{\lengthR_j}\kappa)}\Big)\nn\\
&=c [\phi^{-1}(t)]^{-d}
\prod_{j\in \{1, \ldots, j_0-1\}\cup\{i\}}\fN(\gamma^j)^q
\prod_{j\in \{j_0, \ldots, d-l\}\backslash\{i\}}\fN(\lengthR_j)^q
\prod_{j\in \{d-l+1, \ldots, d\}}\fN(\lengthR_j).
\end{align}
Similarly, for $t/2\le t-\tau\le t$, the case $i > {d-l}$ can be dealt with accordingly,
\begin{align*}
p(t-\tau,z, y)
\le c &[\phi^{-1}(t)]^{-d}
\prod_{j\in \{1, \ldots, j_0-1\}}\fN(\gamma^j)^q
\prod_{j\in \{j_0, \ldots, d-l\}}\fN(\lengthR_j)^q\nn\\
&\cdot\fN(\gamma^i)\prod_{j\in \{d-l+1, \ldots, d\}\backslash\{i\}}\fN(\lengthR_j).
\end{align*}
Since $\delta\to \fN(\delta)$ is decreasing, 
for $t/2\le t-\tau\le t$ and $z\in {A_k^i}\cap A_{k, \gamma}$, we have that 
\begin{align}
&P_{t-\tau}f(z)=\int_{B(y_0, \frac{\kappa}{8})} p(t-\tau,z, y) f(y) dy\nn\\
&\le c [\phi^{-1}(t)]^{-d} \|f\|_1\cdot \mathbf F_{j_0}(l)
\prod_{j\in \{1, \ldots, j_0-1\}}\fN(\gamma^j)^q
\begin{cases}     
\fN(\gamma^i)^q
\fN(\lengthR_i)^{-q}               
&\mbox{ if $\gamma^i< \lengthR_{i}$,}\\
\fN(\gamma^i)
\fN(\lengthR_i)^{-1}
&\mbox{  if	$\gamma^i\ge \lengthR_{i}$}.
\end{cases}\label{eq:Hlqg}
\end{align}
The last inequality is due to $q<1$ and the fact that $\delta\to\fN(\delta)$ is decreasing.
%the fact 
% (note that 

For any  $a, b\in \N_0$, we define
\begin{align*}
\gam{a}{b}:=\gj{a}{b} \text{ if }a\le b,\text{ and otherwise } \gam{a}{b}:=0\,. 
\end{align*}
Recall the definition of $A_k$ so that $k=\sum_{j=1}^d \gamma^j$ and for $z\in {A_k^i}$, there exists $\gamma\in \N_0^d$ such that $z\in A_{k, \gamma}$ with  $\gamma^j=\lengthR_j \mbox{ or }\lengthR_j+1$ for $j\in\{j_0, \ldots, d\}\backslash\{i\}$  (see, \autoref{lem:jAn}).
Therefore,  
\begin{align*}
\gam{1}{j_0-1}+\TT{j_0}{d}-\lengthR_i+\gamma^i\le k\le \gam{1}{j_0-1}+\TT{j_0}{d}-\lengthR_i+\gamma^i+d,
\end{align*}
and so that
\begin{align}
\label{eq:gig}
\gam{1}{j_0-1}+\gamma^i-\lengthR_i\ge k-\TT{j_0}{d} -d
\qquad \text{ and } \qquad\gamma^i-\lengthR_i\le k-\TT{j_0}{d}.
\end{align}
Now decompose $\mathcal{S}(i)$ as follows:
\begin{align*}
\mathcal{S}(i)
=
\sum_{k= \TT{j_0}{d}-\lengthR_i}^{\infty} 
\Psi^i(k)\1_{\{\gamma^i< \lengthR_i\}}+\sum_{k= \TT{j_0}{d}-\lengthR_i}^{\infty} 
\Psi^i(k)\1_{\{\gamma^i\ge \lengthR_i\}}
=:\I+\II.
\end{align*}
For $\I$, by \eqref{eq:LSd}, \eqref{eq:LSIg} and \eqref{eq:Hlqg}, 
\begin{align}\label{eq:s11}
&\E^{x}\left[\1_{\{\tau\ge t/2\}}\1_{\{X_{\tau}\in { A_k^i}\}}P_{t-\tau}f(X_{\tau})\right]\1_{\{\gamma^i< \lengthR_i\}}
\le \,c [\phi^{-1}(t)]^{-d}\|f\|_1 \mathbf F_{j_0}(l)\nn\\
&\cdot\prod_{j\in \{1, \ldots, j_0-1\}}\fN(\gamma^j)^q\ 
\fN(\gamma^i)^q\fN(\lengthR_i)^{-q} \E^x\left[\int_{0}^{t/2\wedge \tau}\int_{I_k^i}
\frac{\1_{\{\gamma^i<\lengthR_i\}}}{|X_s^i-z^i|\phi(|X_s^i-z^i|)}dz^ids\right]\nn\\
\le \,&c [\phi^{-1}(t)]^{-d} \|f\|_1 \mathbf F_{j_0}(l)               
\prod_{j\in \{1, \ldots, j_0-1\}}\fN(\gamma^j)^q\fN(\gamma^i)^q\fN(\lengthR_i)^{1-q} 2^{\gamma^i}.
\end{align}
By \eqref{eq:newmm2} and \eqref{eq:gig}, we have that
\begin{align*}
\prod_{j\in \{1, \ldots, j_0-1\}}\fN(\gamma^j)^q\fN(\gamma^i)^q\fN(\lengthR_i)^{1-q} 2^{\gamma^i}
&\le c 2^{-(\gam{1}{j_0-1}+\gamma^i)(\la+1)q}\cdot 2^{\gamma^i}\fN(\lengthR_i)^{1-q}\\
&\le c 2^{-(k-\TT{j_0}{d}+\lengthR_i)(\la+1)q}2^{(k-\TT{j_0}{d}+\lengthR_i)}\fN(\lengthR_i)^{1-q}
\\
&=c 2^{-(k-\TT{j_0}{d}+\lengthR_i)((\la+1)q-1)}\fN(\lengthR_i)^{1-q}
\end{align*}
where
\begin{align}\label{eq:s1g1}
&\sum_{k= \TT{j_0}{d}-\lengthR_i}^{\TT{j_0}{d}-1}
 2^{-(k-\TT{j_0}{d}+\lengthR_i)((\la+1)q-1)}\ \fN(\lengthR_i)^{1-q}\nn\\
&\le  c
\begin{cases}
2^{\lengthR_i(1-(\la+1)q)}\fN(\lengthR_i)^{1-q}
&\qquad\mbox{ if } q\in [0, \frac{1}{1+\newa})\\
\fN(\lengthR_i)^{1-q}
&\qquad\mbox{ if } q\in (\frac{1}{1+\newa}, 1)
\end{cases}\nn\\
&\le  c
\begin{cases}
\fN(\lengthR_i)^{-(\la+1)^{-1}(1-(\la+1)q)}
\cdot \fN(\lengthR_i)^{1-q}=\fN(\lengthR_i)^{\frac{\la}{\la+1}}
&\mbox{ if } q\in [0, \frac{1}{1+\newa}),\\
\fN(\lengthR_i)^{1-q}
&\mbox{ if } q\in (\frac{1}{1+\newa}, 1).
\end{cases}
\end{align}
Similarly, 
by \eqref{eq:newmm2}, \eqref{eq:gig}  and with the fact that $\gamma^i=\lengthR_i$ or $\lengthR_i+1$, we have that 
\begin{align*}
\prod_{j\in \{1, \ldots, j_0-1\}}\fN(\gamma^j)^q
\fN(\gamma^i)^q\fN(\lengthR_i)^{1-q} 2^{\gamma^i}
&\le c 2^{-(\gam{1}{j_0-1}+\gamma^i)(\la+1)q}\cdot 2^{\gamma^i}\fN(\lengthR_i)^{1-q}\\
&\le c2^{-(k-\TT{j_0}{d}+\lengthR_i)(\la+1)q}2^{\lengthR_i}\fN(\lengthR_i)^{1-q},
\end{align*}
where 
\begin{align}\label{eq:s1g2}
&\sum_{k= \TT{j_0}{d}}^{\infty}
 2^{-(k-\TT{j_0}{d}+\lengthR_i)(\la+1)q}2^{\lengthR_i}\fN(\lengthR_i)^{1-q}
\nn\\
&\le  c
\begin{cases}
\fN(\lengthR_i)^{-(\la+1)^{-1}(1-(\la+1)q)}
\cdot \fN(\lengthR_i)^{1-q}=\fN(\lengthR_i)^{\frac{\la}{\la+1}}
&\mbox{ if } q\in [0, \frac{1}{1+\newa}),\\
\fN(\lengthR_i)^{1-q}
&\mbox{ if } q\in (\frac{1}{1+\newa}, 1).
\end{cases}
\end{align}
Therefore, by \eqref{eq:s11}, \eqref{eq:s1g1} and \eqref{eq:s1g2}, we conclude that 
\begin{align}\label{eq:s1g}
\I
&\le  \Big(\sum_{k= \TT{j_0}{d}-\lengthR_i}^{\TT{j_0}{d}-1}+\sum_{k= \TT{j_0}{d}}^{\infty}\Big)\E^{x}\left[\1_{\{\tau\ge t/2\}}\1_{\{X_{\tau}\in { A_k^i}\}}P_{t-\tau}f(X_{\tau})\right]\1_{\{\gamma^i< \lengthR_i\}}\nn\\
&\le\, c [\phi^{-1}(t)]^{-d} \|f\|_1 \mathbf F_{j_0}(l)
\begin{cases}
\fN(\lengthR_i)^{\frac{\la}{\la+1}}
&\mbox{ if } q\in [0, \frac{1}{1+\newa}),\\
\fN(\lengthR_i)^{1-q}
&\mbox{ if } q\in (\frac{1}{1+\newa}, 1).
\end{cases}
\end{align}
Now we estimate $\II$.
 Note that
$\fN(\gamma^i)\fN(\lengthR_i)^{-1}\le \fN(\gamma^i)^q\fN(\lengthR_i)^{-q}$ for $\gamma^i\ge \lengthR_i$ and $q<1$.
By  \eqref{eq:LSd}, \eqref{eq:Plg} and \eqref{eq:Hlqg}, we have that 
\begin{align*}
&\E^{x}\left[\1_{\{\tau\ge t/2\}}\1_{\{X_{\tau}\in { A_k^i}\}}P_{t-\tau}f(X_{\tau})\right]\1_{\{\gamma^i\ge  \lengthR_i\}}\nn\\
\le \,&c [\phi^{-1}(t)]^{-d}\|f\|_1 \mathbf F_{j_0}(l)
\prod_{j\in \{1, \ldots, j_0-1\}}\fN(\gamma^j)^q
\fN(\gamma^i)^q\fN(\lengthR_i)^{-q}
\cdot \Pp^x(\tau\le t/2, X_{\tau}\in {A_k^i})\1_{\{\gamma^i\ge \lengthR_i\}}\nn\\
\le \,&c [\phi^{-1}(t)]^{-d}\|f\|_1 \mathbf F_{j_0}(l)
\prod_{j\in \{1, \ldots, j_0-1\}}\fN(\gamma^j)^q\fN(\gamma^i)^q\fN(\lengthR_i)^{-q}\fN(\lengthR_{j_0})2^{\lengthR_{j_0}}.
\end{align*}
By \eqref{eq:newmm2} and \eqref{eq:gig}, we have that
\begin{align*}
\prod_{j\in \{1, \ldots, j_0-1\}}\fN(\gamma^j)^q\fN(\gamma^i)^q\fN(\lengthR_i)^{-q}\fN(\lengthR_{j_0})2^{\lengthR_{j_0}}
&\le c 2^{-(\gam{1}{j_0-1}+\gamma^i)(\la+1)q}\cdot\fN(\lengthR_i)^{-q} \fN(\lengthR_{j_0})2^{\lengthR_{j_0}}\\
&\le c 2^{-(k-\TT{j_0}{d}+\lengthR_i)(\la+1)q}\fN(\lengthR_i)^{-q} \fN(\lengthR_{j_0})2^{\lengthR_{j_0}},
\end{align*}
where
\begin{align*}
&\sum_{k= \TT{j_0}{d}}^{\infty}
2^{-(k-\TT{j_0}{d}+\lengthR_i)(\la+1)q}\fN(\lengthR_i)^{-q} \fN(\lengthR_{j_0})2^{\lengthR_{j_0}}\nn\\
&\le c\, 2^{-\lengthR_i(\la+1)q}\fN(\lengthR_i)^{-q} \fN(\lengthR_{j_0})2^{\lengthR_{j_0}}\nn\\
&\le c 2^{\lengthR_{j_0}(1-(\la+1)q)}\fN(\lengthR_i)^{-q} \fN(\lengthR_{j_0})\nn\\
&\le c 
\begin{cases}
\fN(\lengthR_{j_0})^{-(\la+1)^{-1}(1-(\la+1)q)}\fN(\lengthR_i)^{-q} \fN(\lengthR_{j_0})\le \fN(\lengthR_{j_0})^{\frac{\la}{\la+1}}&\mbox{ if } q\in [0, \frac{1}{1+\newa}),\\
\fN(\lengthR_i)^{-q} \fN(\lengthR_{j_0})\le \fN(\lengthR_i)^{1-q}
&\mbox{ if } q\in (\frac{1}{1+\newa}, 1).
\end{cases}
\end{align*}
 For the last inequality, we have used the fact that $\fN(\lengthR_{j_0})\le \fN(\lengthR_{i})$.
 Note that for $\gamma^i\ge \lengthR_i$, $k\ge\TT{j_0}{d}$ by \eqref{eq:gig}.
Therefore, 
\begin{align}\label{eq:s2g}
\II\le\, c [\phi^{-1}(t)]^{-d} \|f\|_1 \mathbf F_{j_0}(l)
\begin{cases}
\fN(\lengthR_{j_0})^{\frac{\la}{\la+1}}
&\mbox{ if } q\in [0, \frac{1}{1+\newa}),\\
\fN(\lengthR_i)^{1-q}
&\mbox{ if } q\in (\frac{1}{1+\newa}, 1).
\end{cases}
\end{align}
Since $j\to  \lengthR_j$ is increasing and $q< 1$,  \eqref{eq:s1g} and \eqref{eq:s2g} imply that for any $i\in \{ j_0,\ldots,  d\}$, 
\begin{align}\label{eq:Sig}
\mathcal{S}(i)\le \,&c[\phi^{-1}(t)]^{-d}\|f\|_1   \mathbf F_{j_0}(l)
\begin{cases}
\fN(\lengthR_{j_0})^{\frac{\la}{\la+1}}&\mbox{ if } q\in [0, \frac{1}{1+\newa})\\
\fN(\lengthR_i)^{1-q}
&\mbox{ if } q\in (\frac{1}{1+\newa},1)
\end{cases}\nn\\
\le \,&c[\phi^{-1}(t)]^{-d}\|f\|_1   \mathbf F_{j_0+1}(l)
\begin{cases}
\fN(\lengthR_{j_0})^{\newb+q}&\mbox{ if } q\in [0, \frac{1}{1+\newa})\,,\\
\fN(\lengthR_{j_0})&\mbox{ if } q\in (\frac{1}{1+\newa},1).
\end{cases}
\end{align}
\subsubsection*{\bf Estimates of \texorpdfstring{$\mathcal{T}(i)$ for $i\in\{1,2,\ldots,j_0-1\}$}{Ti}.} 

Let $l\in\{0,1,\ldots, d-1\}$, $i_0\in\{1,\ldots, d-l\}$ and  $j_0\in\{i_0,\ldots, d-l\}$.
 Similar to the proof of \eqref{eq:Hlqg1}, by \eqref{eq:zyjg}--\eqref{eq:zyj} together with 
$\HHq{q}{l}$ (to be precise $\HHq{q}{l}'$), we have that for $t/2\le t-\tau\le t$, $y\in B(y_0, \frac{\kappa}{8})$ and $z\in {A_k^i}$,
\begin{align*}
p(t-\tau,z, y)&\le c [\phi^{-1}(t)]^{-d}
\prod_{j\in \{1, \ldots, j_0-1\}}\fN(\gamma^j)^q
\prod_{j\in \{j_0, \ldots, d-l\}}\fN(\lengthR_j)^q
\prod_{j\in \{d-l+1, \ldots, d\}}\fN(\lengthR_j).
\end{align*}
Hence
\begin{align}\label{eq:Hlq1g}
P_{t-\tau}f(z)&=\int_{B(y_0, \frac{\kappa}{8})} p(t-\tau,z, y) f(y) dy
\le c [\phi^{-1}(t)]^{-d} \|f\|_1 \mathbf F_{j_0}(l)  \prod_{j\in \{1, \ldots, j_0-1\}}\fN(\gamma^j)^q,
\end{align} 
where $\mathbf F_{j_0}(l)$	is defined in \eqref{eq:prodlg}. Regarding the definition of $A_k$, recall $k=\sum_{j=1}^d \gamma^j$.
For $z\in { A_k^i}$, there exists $\gamma\in \N_0^d$ such that $z\in A_{k, \gamma}$ with  $\gamma^j=\lengthR_j \mbox{ or }\lengthR_j+1$ for $j\in\{j_0, \ldots, d\}$  (see \autoref{lem:jAn}),
hence,
\begin{align*}
\gam{1}{j_0-1}+\TT{j_0}{d}\le	k\le\gam{1}{j_0-1}+\TT{j_0}{d}+d.
\end{align*}
By \eqref{eq:newmm2}, since
\begin{align*}
\prod_{j\in \{1, \ldots, j_0-1\}}\fN(\gamma^j)^q&
\le c 2^{-\gam{1}{j_0-1}(\la+1)q}\le c2^{ -(k-\TT{j_0}{d})(\la+1)q},
\end{align*}
combining \eqref{eq:Plg} and \eqref{eq:Hlq1g}, we have that 
\begin{align}\label{eq:Tig}
\mathcal{T}(i)&=\sum_{k\ge\TT{j_0}{d}}\E^{x}\left[\1_{\{\tau\le t/2\}}\1_{\{X_{\tau}\in {A_k^i}\}}P_{t-\tau}f(X_{\tau})\right]\nn\\
&\le c [\phi^{-1}(t)]^{-d} \|f\|_1
\mathbf F_{j_0}(l)\fN(\lengthR_{j_0})2^{\lengthR_{j_0}}\sum_{k\ge \TT{j_0}{d}}c2^{-(k-\TT{j_0}{d})(\la+1)q}\nn\\
&\le c [\phi^{-1}(t)]^{-d} \|f\|_1
\mathbf F_{j_0}(l) \fN(\lengthR_{j_0})2^{\lengthR_{j_0}}\nn\\
&\le c [\phi^{-1}(t)]^{-d} \|f\|_1
\mathbf F_{j_0+1}(l) \fN(\lengthR_{j_0})^{\newb+q}.
\end{align}
The last inequality holds by \eqref{eq:newmm2} with $\lengthR_{j_0}\ge 1$.
\subsubsection*{\bf Conclusion} 
Finally, by the estimates \eqref{eq:S0g}, \eqref{eq:Sig} and \eqref{eq:Tig} in the representation \eqref{eq:main},
we obtain the upper bound of \eqref{eq:main} as follows:
\begin{align*}
&\E^{x}\left[\1_{\{\tau\le t/2\}}P_{t-\tau}f(X_{\tau})\right]
=\sum_{i=1}^{j_0-1}\mathcal{T}(i)+\sum_{i=j_0}^{d}\mathcal{S}(i)+\Phi(0)\nn\\
&\le c [\phi^{-1}(t)]^{-d} \|f\|_1     \mathbf F_{j_0+1}(l)
\begin{cases}
\fN(\lengthR_{j_0})^{\newb+q}&\mbox{ if } q\in [0, \frac{1}{1+\newa})\,,\\
\fN(\lengthR_{j_0})&\mbox{ if } q\in (\frac{1}{1+\newa},1).
\end{cases}
\end{align*}
This proves \autoref{prop:main} by {(\bf WS)}, \eqref{d:thi_Ri}, \eqref{eq:newmm} and  \autoref{rem:prel} (3). 
\fop

%\bibliography{LK-bib}{}

\begin{thebibliography}{99}
	\bibitem{Aro68} D. G. Aronson. Non-negative solutions of linear parabolic equations. \textit{Ann. Scuola Norm. Sup. Pisa (3)}, 22: 607--694, 1968.
	\bibitem{MR2465826} M. T. Barlow, R. F. Bass, Z.-Q. Chen, and M. Kassmann. Non-local Dirichlet forms and symmetric jump processes. \textit{Trans. Amer. Math. Soc.}, 361(4):1963--1999, 2009.
	\bibitem{BGK09} M. T. Barlow, A. Grigor'yan, and T. Kumagai. Heat kernel upper bounds for jump processes and the first exit time. \textit{J. Reine Angew. Math.}, 626:135--157, 2009.
	\bibitem{bass1998diffusions} R. F. Bass. \textit{Diffusions and elliptic operators}. Springer Science \& Business Media, 1998.
	\bibitem{BaLe02} R. F. Bass and D. A. Levin. Transition probabilities for symmetric jump processes. \textit{Trans. Amer. Math. Soc.}, 354(7):2933--2953, 2002.
	\bibitem{BKK10} R. F. Bass, M. Kassmann, and T. Kumagai. Symmetric jump processes: localization, heat kernels and convergence. \textit{Ann. Inst. Henri Poincar\'e Probab. Stat.}, 46(1):59--71, 2010.
	\bibitem{BlGe68} R. M. Blumenthal and R. K. Getoor.\textit{ Markov processes and potential theory}. New York, 1968.
	\bibitem{CKS87} E. A. Carlen, S. Kusuoka, and D. W. Stroock. Upper bounds for symmetric markov transition functions. \textit{Ann. Inst. Henri Poincar\'e Probab. Stat.}, 23(2):245--287, 1987.
	\bibitem{ChKu03} Z.-Q. Chen and T. Kumagai. Heat kernel estimates for stable-like processes on d-sets. \textit{Stochastic Process. Appl.}, 108(1):27--62, 2003.
	\bibitem{ChKu08} Z.-Q. Chen and T. Kumagai. Heat kernel estimates for jump processes of mixed types on metric measure spaces. \textit{Probab. Theory Related Fields}, 140(1-2):277--317, 2008.
	\bibitem{ChSo03} Z.-Q. Chen and R. Song. Drift transforms and green function estimates for discontinuous processes. \textit{J. Funct. Anal.}, 201:262--281, 2003.
	\bibitem{C96} T. Coulhon. Ultracontractivity and nash type inequalities. \textit{J. Funct. Anal.}, 141(2):510--539, 1996.
	\bibitem{MR2778606} M. Fukushima, Y. Oshima, and M. Takeda. \textit{Dirichlet Forms and Symmetric Markov Processes}, volume 19 of \textit{De Gruyter Studies in Mathematics}. Walter de Gruyter \& Co., Berlin, 2011.
	\bibitem{KKKpre} M. Kassmann, K.-Y. Kim, and T. Kumagai. Heat kernel bounds for nonlocal operators with singular kernels, 2019. https://arxiv.org/abs/1910.04242.
	\bibitem{MU11} J. Masamune and T. Uemura. Conservation property of symmetric jump processes. \textit{Ann. Inst. Henri Poincar\'e Probab. Stat.}, 47(3):650--662, 2011.
	\bibitem{Xu13} F. Xu. A class of singular symmetric Markov processes. \textit{Potential Anal.}, 38(1):207--232, 2013.
\end{thebibliography}
%\bibliographystyle{alpha}

\makeatletter
\providecommand
\@dotsep{5}
\makeatother
%\listoftodos
\relax
	
\end{document}